\documentclass{amsart}
\usepackage{amssymb}
\usepackage{braket}
\usepackage{mathrsfs}
\usepackage[all]{xy}
\usepackage{graphicx}

\newtheorem{thm}[equation]{Theorem}
\newtheorem{cor}[equation]{Corollary}
\newtheorem{prop}[equation]{Proposition}
\theoremstyle{definition}
\newtheorem{dfn}[equation]{Definition}
\newtheorem{ex}[equation]{Example}
\newtheorem{claim}[equation]{Claim}
\newtheorem{lem}[equation]{Lemma}
\newtheorem{fact}[equation]{Fact}
\newtheorem{caution}[equation]{Caution}

\theoremstyle{remark}
\newtheorem{rem}[equation]{Remark}

\newtheorem{introthm}{Theorem \ref{ThmBF}\!\!}

\newcommand{\Ob}{\mathrm{Ob}}
\newcommand{\id}{\mathrm{id}}

\newcommand{\pt}{\mathbf{1}} 
\newcommand{\ppr}{^{\prime}}
\newcommand{\pprr}{^{\prime\prime}}
\newcommand{\Map}{\mathrm{Map}}
\newcommand{\Ker}{\mathit{Ker}}
\newcommand{\Id}{\mathrm{Id}}
\newcommand{\pro}{\mathrm{pr}}

\newcommand{\nwp}{\mathrm{nwp}}
\newcommand{\dfl}{\mathrm{dfl}}
\newcommand{\SIm}{\mathrm{SIm}}
\newcommand{\Obig}{\Omega_{\mathrm{big}}}
\newcommand{\Rg}{\mathfrak{R}}

\newcommand{\ind}{\mathrm{ind}}
\newcommand{\res}{\mathrm{res}}
\newcommand{\Ind}{\mathrm{Ind}}
\newcommand{\Inf}{\mathrm{Inf}}
\newcommand{\Def}{\mathrm{Def}}
\newcommand{\Res}{\mathrm{Res}}
\newcommand{\Iso}{\mathrm{Iso}}

\newcommand{\Mon}{\mathit{Mon}}
\newcommand{\Sett}{\mathit{Set}}
\newcommand{\sett}{\mathit{set}}
\newcommand{\Ab}{\mathit{Ab}}
\newcommand{\Add}{\mathit{Add}}

\newcommand{\Ring}{\mathit{Ring}}
\newcommand{\Grp}{\mathrm{Grp}}
\newcommand{\RMod}{R\mathrm{Mod}}
\newcommand{\Gs}{{}_G\mathit{set}}
\newcommand{\Hs}{{}_H\mathit{set}}
\newcommand{\Ks}{{}_K\mathit{set}}
\newcommand{\Sbb}{\mathbb{S}}
\newcommand{\Cbb}{\mathbb{C}}
\newcommand{\GrSet}{\mathrm{GrSet}}
\newcommand{\Mack}{\mathit{Mack}}
\newcommand{\SMack}{\mathit{SMack}}
\newcommand{\SMackS}{\mathit{SMack}(\Sbb)}
\newcommand{\MackS}{\mathit{Mack}(\Sbb)}
\newcommand{\MackSR}{\mathit{Mack}^R(\Sbb)}
\newcommand{\MackCR}{\mathit{Mack}^R(\Csc)}
\newcommand{\Sp}{\mathrm{Sp}}  
\newcommand{\BisetFtr}{\mathcal{F}_{\mathcal{B},R}}

\newcommand{\Spana}{\mathrm{Span}^{X}_{Y}}
\newcommand{\SpanHGa}{\mathrm{Span}^{Y}_{X}}

\newcommand{\Zbb}{\mathbb{Z}} 
\newcommand{\Csc}{\mathscr{C}}
\newcommand{\Mbf}{\mathbf{M}} 
\newcommand{\Rbf}{\mathbf{R}} 
\newcommand{\Tbf}{\mathbf{T}} 
\newcommand{\Bcal}{\mathcal{B}}
\newcommand{\Ecal}{\mathcal{E}}
\newcommand{\Ical}{\mathcal{I}}

\newcommand{\Scal}{\mathcal{S}}
\newcommand{\Tcal}{\mathcal{T}}

\newcommand{\xg}{\frac{X}{G}} 
\newcommand{\yh}{\frac{Y}{H}} 
\newcommand{\zk}{\frac{Z}{K}} 
\newcommand{\althh}{\frac{\alpha}{\theta}} 
\newcommand{\althhp}{\frac{\alpha\ppr}{\theta\ppr}} 
\newcommand{\tab}{\frac{\beta}{\tau}} 
\newcommand{\tabp}{\frac{\beta\ppr}{\tau\ppr}} 
\newcommand{\al}{\alpha}
\newcommand{\ax}{\al(x)}

\newcommand{\be}{\beta}
\newcommand{\lam}{\lambda}
\newcommand{\tc}{\Rightarrow}
\newcommand{\ep}{\varepsilon}

\newcommand{\thh}{\theta}
\newcommand{\iv}{^{-1}}
\newcommand{\co}{\colon}
\newcommand{\ci}{\circ}
\newcommand{\ups}{\upsilon}
\newcommand{\und}{\underline}
\newcommand{\thax}{\thh_{\al,x}}
\newcommand{\thby}{\thh_{\be,y}}
\newcommand{\wl}{\frac{W}{L}} 
\newcommand{\am}{\amalg}
\newcommand{\iog}{\iota^{(G)}}
\newcommand{\ioh}{\iota^{(H)}}
\newcommand{\prg}{\pro^{(G)}}
\newcommand{\prh}{\pro^{(H)}}
\newcommand{\prgn}{\pro^{(G/N)}}
\newcommand{\gad}{\gamma\underset{\ep}{\ast}\delta}
\newcommand{\nm}{\vartriangleleft}
\newcommand{\wt}{\widetilde}

\newcommand{\SoverX}{\Sbb/\! _{\xg}}
\newcommand{\SoverXa}{\Cbb/\! _X}
\newcommand{\aax}{\frac{A}{K}\ov{\al}{\lra}\xg} 

\newcommand{\aaxa}{A\ov{\al}{\lra}X}
\newcommand{\bbxa}{B\ov{\be}{\lra}X}
\newcommand{\ccxa}{C\ov{\gamma}{\lra}X}

\newcommand{\kagx}{\frac{A}{K}\ov{\al}{\to}\xg} 
\newcommand{\hbgx}{\frac{B}{H}\ov{\be}{\to}\xg} 
\newcommand{\lsws}{\frac{W_S}{L_S}} 
\newcommand{\ltwt}{\frac{W_T}{L_T}} 

\newcommand{\gsha}{{}_{X}\! S_{Y}}

\newcommand{\rta}{\rightharpoonup}
\newcommand{\lta}{\leftharpoonup}
\newcommand{\ov}{\overset}

\newcommand{\lla}{\longleftarrow}
\newcommand{\lra}{\longrightarrow}

\newcommand{\spS}{\xg\ov{\al_S}{\lla}\lsws\ov{\be_S}{\lra}\yh}
\newcommand{\spT}{\xg\ov{\al_T}{\lla}\ltwt\ov{\be_T}{\lra}\yh}

\newcommand{\spSc}{\yh\ov{\al_S}{\lla}\lsws\ov{\be_S}{\lra}\xg}
\newcommand{\spTc}{\yh\ov{\al_T}{\lla}\ltwt\ov{\be_T}{\lra}\xg}
\newcommand{\spTzc}{\zk\ov{\al_T}{\lla}\ltwt\ov{\be_T}{\lra}\yh}

\newcommand{\spSa}{X\ov{\al_S}{\lla}W_S\ov{\be_S}{\lra}Y}
\newcommand{\spTa}{X\ov{\al_T}{\lla}W_T\ov{\be_T}{\lra}Y}

\newcommand{\spPa}{X\ov{\al_P}{\lla}W_P\ov{\be_P}{\lra}Y}
\newcommand{\spSca}{Y\ov{\al_S}{\lla}W_S\ov{\be_S}{\lra}X}

\newcommand{\spTzca}{Z\ov{\al_T}{\lla}W_T\ov{\be_T}{\lra}Y}

\newcommand{\xk}{\kappa,\xi,(b,a,h)}
\newcommand{\xkp}{\kappa\ppr,\xi\ppr,(b\ppr,a\ppr,h\ppr)}

\newcommand{\fa}{\forall}

\newcommand{\eix}{\eta_i^{(X)}}

\newcommand{\ABR}{\Add(\Bcal,\RMod)}

\newcommand{\coCone}{\mathrm{coCone}}

\numberwithin{equation}{subsection}

\begin{document}

\title{A Mackey-functor theoretic interpretation of biset functors}

\author{Hiroyuki NAKAOKA}
\address{Research and Education Assembly, Science and Engineering Area, Research Field in Science, Kagoshima University, 1-21-35 Korimoto, Kagoshima, 890-0065 Japan\ /\ LAMFA, Universit\'{e} de Picardie-Jules Verne, 33 rue St Leu, 80039 Amiens Cedex1, France}

\email{nakaoka@sci.kagoshima-u.ac.jp}
\urladdr{http://www.lamfa.u-picardie.fr/nakaoka/}


\thanks{The author wishes to thank Professor Fumihito Oda for his comments and interest.}
\thanks{The author wishes to thank Professor Serge Bouc and Professor Radu Stancu for their comments and useful suggestions.}
\thanks{The author wishes to thank Professor Erg\"{u}n Yal\c{c}\i n for his comments and advices.}
\thanks{This work is supported by JSPS KAKENHI Grant Numbers 25800022,\, 24540085.}
\thanks{The author wishes to thank the referee for his thorough reading and lots of instructive advices.}

\begin{abstract}
In this article, we consider a formulation of biset functors using the 2-category of finite sets with variable finite group actions. We introduce a 2-category $\Sbb$, on which a biset functor can be regarded as a special kind of Mackey functors. This gives an analog of Dress' definition of a Mackey functor, in the context of biset functors.
\end{abstract}

\maketitle

\tableofcontents

\section{Introduction and Preliminaries}

\setcounter{subsection}{1}

For a fixed finite group $G$ and a commutative coefficient ring $R$, an $R$-linear {\it Mackey functor} can be defined in three ways, which give essentially the same notion (\cite[section 1]{Bouc}):
\begin{itemize}
\item Naive definition, which defines a Mackey functor $M$ as a family $M=(M(H),\ind^G_H,\res^G_H,c_{g,H})$ of $R$-modules and homomorphisms.
\item Bifunctorial definition by Dress \cite[section 4]{Dress}, which defines $M$ as a pair of functors $M=(M^{\ast}, M_{\ast})$ on the category $\Gs$ of finite $G$-sets to the category of $R$-modules $\RMod$, satisfying some compatibilities with respect to coproducts and fibered products in $\Gs$.
\item The way shown by Lindner \cite[Theorem 4]{Lindner}, which regards $M$ as a functor $\Tcal\to\RMod$ preserving finite products, from a category $\Tcal$ constructed from the span category of $\Gs$, to $\RMod$.
\end{itemize}
A Mackey functor is a useful tool to describe how an algebraic system associated to finite groups (such as Burnside rings or representation rings, or cohomology groups, etc.) behaves under the change of subgroups of a fixed group $G$.
Moreover if one expects to be freed from the constraint of the container group $G$, we may consider {\it all} finite groups, and inclusions among them. This leads to the notion of a {\it global Mackey functor}.

Recently, Bouc \cite{Bouc_Biset} has defined the notion of a {\it biset functor}, which moreover enables us to deal with the behavior of algebraic systems named as above, with respect to {\it all} group homomorphisms between {\it all} finite groups. An $R$-linear biset functor $B$ is defined to be an $R$-linear functor $B\co \Bcal_R\to\RMod$, from the {\it biset category} $\Bcal_R$ to $\RMod$.
The biset category which we deal with in this article is the following one.
\begin{dfn}(cf. \cite[Definitions 3.1.1, 3.1.6]{Bouc_Biset})
The $R$-linear category $\Bcal_R$ is defined as follows.
\begin{enumerate}
\item An object in $\Bcal_R$ is a finite group.
\item Let $G,H$ be objects in $\Bcal_R$. A finite $H$-$G$-biset $U$ is, by definition, a finite set $U$ equipped with a left $H$-action and a right $G$-action satisfying
\[ (hu)g=h(ug) \]
for any $h\in H,g\in G,u\in U$. The set of isomorphism classes of finite $H$-$G$-bisets forms a commutative monoid with addition $\am$ and unit $\emptyset$, and thus we can take its additive completion $\Bcal(G,H)$. We define $\Bcal_{R}(G,H)$ by $\Bcal_{R}(G,H)=\Bcal(G,H)\otimes R$. This is the set of morphisms from $G$ to $H$ in $\Bcal_R$.

An $H$-$G$-biset $U$ is written as ${}_HU_G$.
The composition of two consecutive bisets ${}_{H}U_G$ and ${}_{K}V_H$ is given by
\[ V\times_HU=(V\times U)/\sim, \]
where the equivalence relation is defined as
\begin{itemize}
\item[-] $(v,u),(v\ppr,u\ppr)\in V\times U$ are equivalent if there exists $h\in H$ satisfying $v=v\ppr h$ and $u\ppr=hu$.
\end{itemize}
This defines the composition of morphisms in $\Bcal_R$, by linearity.
\end{enumerate}
When $R=\Zbb$, we denote $\Bcal_{\Zbb}$ simply by $\Bcal$. This is a preadditive category.
\end{dfn}

We denote the category of $R$-linear biset functors by $\BisetFtr$, whose morphisms are natural transformations. This is naturally equivalent to the category $\Add(\Bcal,\RMod)$ of additive functors from $\Bcal$ to $\RMod$.

Remark that an $H$-$G$-biset $U$ is identified with an $H\times G$-set, with the action
\[ (h,g)u=hug\iv\quad(\fa (h,g)\in H\times G, \fa u\in U). \]
If $U$ is transitive as an $H\times G$-set, then it can be decomposed as follows \cite[Lemma 2.3.26]{Bouc_Biset};
\[ U\cong\Ind^H_C\underset{C}{\times}\Inf^C_D\underset{C/D}{\times}\Iso(f)\underset{B/A}{\times}\Def^B_A\underset{B}{\times}\Res^G_B,
\]
using a sequence of inclusions, quotients, and an isomorphism of groups
\begin{equation}\label{DecompFund}
H\hookleftarrow C\twoheadrightarrow C/D\underset{\cong}{\ov{f}{\lla}} B/A\twoheadleftarrow B\hookrightarrow G.
\end{equation}
Thus a biset functor is regarded as a family of $R$-modules $\{ B(G)\}_{G\in\Ob(\Grp)}$ equipped with operations associated to elementary bisets (\cite[Elementary Bisets 2.3.9]{Bouc_Biset})
\begin{itemize}
\item[-] $\Ind^G_H={}_GG_H$, for a subgroup $H\le G$,
\item[-] $\Res^G_H={}_HG_G$, for a subgroup $H\le G$,
\item[-] $\Inf^G_N={}_G(G/N)_{(G/N)}$, for a normal subgroup $N\nm G$,
\item[-] $\Def^G_N={}_{(G/N)}(G/N)_{G}$, for a normal subgroup $N\nm G$,
\item[-] $\Iso(f)={}_HH_G$, for a group isomorphism $f\co G\ov{\cong}{\lra}H$.
\end{itemize}
These operations satisfy some fundamental relations (\cite[section 2.3]{Bouc_Biset}), together with an extra relation corresponding to
\begin{equation}\label{DefInfRel}
\Def^G_N\underset{G}{\times}\Inf^G_N\cong\Id
\end{equation}
for any normal subgroup $N\nm G$. Here $\Id$ stands for the identity biset $\Id={}_{G/N}\Id_{G/N}$, which gives the identity morphism for $G/N$ in $\Bcal_R$. Because of $(\ref{DefInfRel})$, any biset functor $B\in\Ob(\BisetFtr)$ should satisfy
\begin{equation}\label{Eq1.3}
B(\Def^G_N)\ci B(\Inf^G_N)=\id.
\end{equation}

Remark that the sequence $(\ref{DecompFund})$ can be {\it flipped up} by taking fibered product to obtain a sequence
\[
\xy
(-32,0)*+{H}="0";
(-16,0)*+{C}="2";
(0,12)*+{F}="4";
(0,-12)*+{C/D\cong B/A}="6";
(16,0)*+{B}="8";
(32,0)*+{G}="10";
{\ar "2";"0"};
{\ar "2";"6"};
{\ar "8";"10"};
{\ar "8";"6"};
{\ar "4";"2"};
{\ar "4";"8"};
{\ar@{}|\circlearrowright "2";"8"};
\endxy
\]
which can be regarded as a span of group homomorphisms
\[ H\ov{\varphi}{\lla}F\ov{\psi}{\lra}G \]
up to some isomorphism.
Since it is not always possible to \lq flip down' conversely, {\it spans of group homomorphisms} are treating a bit wider class than {\it bisets}.
In analogy with the \lq three definitions' of Mackey functors, this observation gives us an impression that a biset functor is defined by some compound of \lq naive' and \lq Lindner-type' definitions.
In this article, as a platform for further developments of biset functor theory, we introduce an analog of Dress' definition for biset functors.

One of the motivations for this interpretation is to provide a framework for biset functors equipped with multiplicative inductions, such as Burnside functor and the representation ring functor. In the ordinary Mackey functor theory, those with compatible multiplicative transfers are called {\it Tambara functors}, whose definition essentially requires Dress' definition. In \cite{N_MackBisetTam} and forthcoming works, we will formalize \lq Tambara' properties for biset functors, using the interpretation obtained in this article. Moreover, as a by-product, this interpretation gives some constructions of biset functors, whose analogs are known for Mackey functors which involve Dress' definition.
For example in \cite{N_Several}, we show that analogs of Jacobson's $F$-Burnside construction (\cite[section 2]{Jac}) and Boltje's $(-)_+$-construction (\cite[section 2]{Boltje}) can be applied to biset functors.

The central mechanism for the Dress' definition of Mackey functors was that, for a fixed finite group $G$, a $G$-set can be regarded as a parallel array of subgroups of $G$ by taking stabilizers. A $G$-map then corresponds to a parallel array of inclusions of subgroups. In the case of biset functors, it will be natural to prepare a category which can encode {\it all} finite groups and {\it all} homomorphisms not only the inclusions of subgroups.

To realize this, we define a category $\Csc$ whose object is a pair $(G,X)$ of a finite group $G$ and a finite $G$-set $X$.
By taking stabilizers, an object in $\Csc$ can be regarded as a parallel array of finite groups. 
Moreover, with an appropriate definition of morphisms, we can regard a morphism in $\Csc$ as a parallel array of group homomorphisms between them, classified up to some conjugates.
If one could show $\Csc$ admits fibered products and coproducts, then it would be possible to find some analog of Dress' definition.

However, it soon turns out that $\Csc$ does not have strict fibered products, and that it is more natural to use a 2-categorical framework. 
Indeed, we introduce a 2-category $\Sbb$ with invertible 2-cells, which recovers $\Csc$ as the category $\Sbb/\text{{\it 2-cells}}$ defined in the following way.
\begin{dfn}\label{DefClass}
Let $\Cbb$ be a 2-category, whose 2-cells are invertible with respect to the vertical composition. Then the category $\Cbb/\text{{\it 2-cells}}$ is associated in the following way.
\begin{itemize}
\item[{\rm (i)}] $\Ob(\Cbb/\text{{\it 2-cells}})$ is equal to the class of 0-cells in $\Cbb$.
\item[{\rm (ii)}] A pair of 1-cells $\al,\be\co X\to Y$ in $\Cbb$ is defined to be equivalent if and only if there exists a 2-cell $\ep\co\al\tc\be$ in $\Cbb$.
\item[{\rm (iii)}] For any $X,Y\in\Ob(\Cbb/\text{{\it 2-cells}})$, define the morphism set $(\Cbb/\text{{\it 2-cells}})(X,Y)$ to be the set of equivalence classes defined in {\rm (ii)}.
\item[{\rm (iv)}] The composition and the identities are induced from the horizontal composition and the identity 1-cells in $\Cbb$.
The equivalence class of $\al\co X\to Y$ is denoted by $\und{\al}$.
\end{itemize}
\end{dfn}

We will show $\Sbb$ admits bicoproducts and bipullbacks (Propositions \ref{Prop2CoprodVari}, \ref{Prop2Pullback}), and thus we can define a \lq Mackey functor' on $\Sbb$.
Because of the \lq flipping-gap' between spans and bisets, biset functors correspond to some special kind of Mackey functors on $\Sbb$, which we call {\it deflative Mackey functors} (Definition \ref{DefDeflMack}). A Mackey functor is called deflative if it satisfies a condition corresponding to $(\ref{Eq1.3})$.
As the main theorem (Theorem \ref{ThmBF}), we establish an equivalence between the category of deflative Mackey functors on $\Sbb$ and the category of biset functors, as follows.

\begin{introthm}
There is an equivalence of categories
\[ \Mack_{\dfl}^R(\Sbb)\simeq\BisetFtr, \]
where the left hand side denotes the category of deflative $R$-Mackey functors on $\Sbb$.
\end{introthm}

\bigskip

In section 2, we define our base 2-category $\Sbb$. In section 3, we investigate its first properties. Especially, we show the existence of bicoproducts and bipullbacks in $\Sbb$.
In section 4, we introduce the notion of stab-surjective 1-cells in $\Sbb$. This is an analog of surjective group homomorphisms, and gives a kind of factorization of 1-cells in $\Sbb$. Indeed, we show that any 1-cell in $\Sbb$ is equivalent to a composition of an equivariant 1-cell and a stab-surjective 1-cell. In section 5, we define (deflative) Mackey functors on $\Sbb$, using bicoproducts and bipullbacks. We also give a formulation using the span category associated to $\Sbb$. We mention that Ibarra's work \cite{Ibarra} gives a more direct construction of a category related to this span category. In section 6, we show our main theorem.

\bigskip

Throughout this article, any group $G$ is assumed to be finite. The category of finite groups and homomorphisms is denoted by $\Grp$.
The unit of a (finite) group will be denoted by $e$. Abbreviately we denote the trivial group by $e$, instead of $\{ e\}$. For an element $g$ in a group $G$ and its subgroup $H\le G$, we denote the conjugation map by $\sigma_g\colon H\to gHg\iv\ ;\ x\mapsto gxg\iv$.
For a group $G$, the symbol $\Gs$ denotes the category of finite $G$-sets and $G$-equivariant maps. 
A one-point set is denoted by $\pt$. Abbreviately, the unique map from any set $X$ to $\pt$ is denoted by $\pt\co X\to \pt$.

 In this article, a biset is always assumed to be finite.
A monoid is always assumed to be unitary and commutative. Similarly a ring is assumed to be commutative, with an additive unit $0$ and a multiplicative unit $1$.
We denote the category of monoids by $\Mon$, the category of rings by $\Ring$. 
A monoid homomorphism preserves units, and a ring homomorphism preserves 
$1$.
For any category $\mathscr{K}$ and any pair of objects $X$ and $Y$ in $\mathscr{K}$, the set of morphisms from $X$ to $Y$ in $\mathscr{K}$ is denoted by $\mathscr{K}(X,Y)$. 

Any 2-category is assumed to be strict (\cite[Definition 7.1.1]{Borceux},\cite[XII.3]{MacLane}). For a 2-category $\Cbb$, the entity of 0-cells (respectively 1-cells, 2-cells) is denoted by $\Cbb^0$ (resp. $\Cbb^1$, $\Cbb^2$). For a pair of 0-cells $X, Y$ in $\Cbb$, the set of 1-cells from $X$ to $Y$ is denoted by $\Cbb^1(X,Y)$. Together with the 2-cells among them, they form a category $\Cbb(X,Y)$ satisfying $\Ob(\Cbb(X,Y))=\Cbb^1(X,Y)$.

\section{The 2-category of finite sets with group actions}

In this article, we work on (2-)categories whose objects are finite sets equipped with group actions.
First, we introduce a naive one.

\subsection{Category $\GrSet$}
\begin{dfn}\label{DefGrSet}
The category $\GrSet$ is defined as follows.
\begin{enumerate}
\item An object in $\GrSet$ is a pair of a finite group $G$ and a finite $G$-set $X$. We denote\footnote{This notation is thanks to Professor Serge Bouc.} this pair by $\xg$. 
\item If $\xg$ and $\yh$ are two objects in $\GrSet$, then a morphism $\althh\co \xg\to\yh$ is a pair of a map $\al\co X\to Y$ and a map
\[ \thh\co X\to\Map(G,H)\ ;\ x\mapsto \thh_x \]
(namely, $\thh$ is a family of maps $\{\thh_x \}_{x\in X}$)
satisfying
\begin{itemize}
\item[{\rm (i)}] $\al(gx)=\thh_x(g)\al(x)$ 
\item[{\rm (ii)}] $\thh_x(gg\ppr)=\thh_{g\ppr x}(g)\thh_x(g\ppr)$
\end{itemize}
for any $x\in X$ and any $g,g\ppr\in G$.
\end{enumerate}
$\theta$ is called the {\it acting part} or the {\it denominator} of $\althh$.

For any consecutive pair of morphisms
\[ \xg\ov{\althh}{\lra}\yh\ov{\tab}{\lra}\zk, \]
we define their composition $(\tab)\ci(\althh)=\frac{\be\ci\al}{\tau\ci\thh}$
 by
\begin{itemize}
\item[-] $\be\ci\al\co X\to Z$ is the usual composition of maps of sets,
\item[-] $\tau\ci\thh$ is defined by
\[ (\tau\ci\thh)_x(g)=\tau_{\ax}(\thh_x(g))\quad(\fa g\in G), \]
namely, $(\tau\ci\thh)_x=\tau_{\ax}\ci\thh_x$ for any $x\in X$.
\end{itemize}
This $\frac{\be\ci\al}{\tau\ci\thh}\co \xg\to\zk$ becomes in fact a morphism. We leave the details to the reader.

If $\xg\ov{\althh}{\lra}\yh\ov{\tab}{\lra}\zk\ov{\frac{\gamma}{\mu}}{\lra}\frac{W}{L}$ is a sequence of morphisms, then 
the associativity of the composition is satisfied.
The identity morphism for $\xg$ in $\GrSet$ is given by $\id_{\xg}=\frac{\id_X}{\id_G}$.
\end{dfn}

\begin{rem}
In $\GrSet$, the object $\frac{\pt}{e}$ is terminal. Besides, for any finite group $G$, the object $\frac{\emptyset}{G}$ is initial in $\GrSet$. In particular, there is a $($unique$)$ isomorphism
\[ \frac{\emptyset}{G}\ov{\cong}{\lra}\frac{\emptyset}{H} \]
for any pair of finite groups $G$ and $H$.
We will often denote this isomorphic initial object simply by $\emptyset\in\Ob(\GrSet)$.
\end{rem}

\begin{dfn}\label{DefEquiv}
Let $f\co G\to H$ be a group homomorphism.
A morphism $\althh\co\xg\to\yh$ is $f$-{\it equivariant} if it satisfies $\thh_x=f$ for any $x\in X$. (Remark that, condition {\rm (2)} {\rm (ii)} in Definition \ref{DefGrSet} is automatically satisfied.) In this case, we simply write the morphism as $\frac{\al}{f}$. 

When $f=\id_G$, we say the morphism is {\it equivariant}, or $G$-{\it equivariant} if we specify the group $G$, and denote it by $\frac{\al}{G}\co \xg\to \frac{Y}{G}$. In this case, $\al$ is nothing but a usual $G$-map $\al\co X\to Y$. 
\end{dfn}

\begin{rem}\label{RemGrSet}
$\ \ $
\begin{enumerate}
\item We sometimes express a morphism $\althh$ simply by $\al\co\xg\to\yh$. This abbreviation does not mean that $\thh$ is determined by $\al$. For example, if $X=Y=\pt$ and $\al$ is the unique constant map $\al=\pt\colon \pt\to\pt$, then $\frac{\pt}{f}\co \frac{\pt}{G}\to \frac{\pt}{H}$ becomes a morphism in $\GrSet$ for any group homomorphism $f\co G\to H$. (See Proposition \ref{PropFunctHom}.)
\item If $\althh\co\xg\to\yh$ is a morphism, then for each $G$-orbit $X_0\subseteq X$, its image $\al(X_0)$ is contained in some single $H$-orbit in $Y$.
\item If $\althh\co \xg\to\yh$ is a morphism in $\GrSet$, then for any $x\in X$, the restriction of $\thh_x\co G\to H$ to the stabilizer $G_x$ gives a group homomorphism
\[ \thh_x\co G_x\to H_{\ax}. \]
In particular, we always have $\thh_x(e)=e$ for any $x\in X$.
\end{enumerate}
\end{rem}

\begin{rem}\label{RemGrSet2}
Let $\althh\co\xg\to \yh$ be a morphism in $\GrSet$. Remark that for $x\in X$, the $G$-orbit $Gx$ is isomorphic to $G/G_x$ as a $G$-set by
\[ Gx\ov{\cong}{\lra}G/G_x\ ;\ gx\mapsto gG_x. \]
Similarly for $\ax\in Y$, we have an isomorphism of $H$-sets
\[ H\ax\ov{\cong}{\lra}H/H_{\ax}\ ;\ h\ax\mapsto hH_{\ax}. \]
Then $\thh_x$ gives a map
\[ \Theta_{\al,x}\co G/G_x\to H/H_{\ax}\ ;\ gG_x\mapsto \thh_x(g)H_{\ax}, \]
which is compatible with $\al$ and the above isomorphisms:
\[
\xy
(-11,8)*+{X}="0";
(11,8)*+{Y}="2";
(-11,4)*+{\rotatebox{90}{$\subseteq$}}="3";
(-11,0)*+{Gx}="4";
(11,4)*+{\rotatebox{90}{$\subseteq$}}="5";
(11,0)*+{H\ax}="6";
(-11,-10)*+{G/G_x}="8";
(11,-10)*+{H/H_{\ax}}="10";
{\ar^{\al} "0";"2"};
{\ar_{\cong} "4";"8"};
{\ar^{\cong} "6";"10"};
{\ar_{\Theta_{\al,x}} "8";"10"};
{\ar@{}|\circlearrowright "0";"10"};
\endxy
\]
\end{rem}

\begin{prop}\label{PropFunctEqui}
Let $G$ be any finite group. The following correspondence gives a faithful $($but not full$)$ functor
\[ \frac{\bullet}{G}\co \Gs\to\GrSet. \]
\begin{itemize}
\item[-] To any $X\in\Ob(\Gs)$, we associate $\xg\in\Ob(\GrSet)$.
\item[-] To any $\al\in \Gs(X,Y)$, we associate $\frac{\al}{G}\in\GrSet(\xg,\frac{Y}{G})$.
\end{itemize}
\end{prop}
\begin{proof}
This is straightforward.
\end{proof}

\begin{prop}\label{PropFunctHom}
The following correspondence gives a fully faithful functor
\[ \frac{\pt}{\bullet}\co \Grp\to\GrSet. \]
\begin{itemize}
\item[-] To any $G\in\Ob(\Grp)$, we associate $\frac{\pt}{G}\in\Ob(\GrSet)$.
\item[-] To any $f\in \Grp(G,H)$, we associate $\frac{\pt}{f}\in\GrSet(\frac{\pt}{G},\frac{\pt}{H})$.
\end{itemize}
\end{prop}
\begin{proof}
This is straightforward.
\end{proof}

\subsection{2-category $\Sbb$ and category $\Csc$}

Remarks \ref{RemGrSet}, \ref{RemGrSet2} suggest that an object (respectively a morphism) in $\GrSet$ can be regarded as an array of finite groups (resp. of homomorphisms). This eventually leads to relating it to biset functors. To this end, we need to consider a weaker equivalence relation, with which two objects $\xg$ and $\yh$ become equivalent when they have the same array of stabilizers. For example, we expect $\frac{G/K}{G}$ and $\frac{H/K}{H}$ to be equivalent, for any sequence of subgroups $K\le H\le G$. Remark that they are never isomorphic in $\GrSet$ unless $G=H$, as an isomorphism in $\GrSet$ never changes the cardinality $|X|$ of an object $\xg$.

Thus what we really need is a category $\Csc$ obtained by modifying $\GrSet$, in which the above weaker equivalence is realized as an isomorphism.
Indeed, we will define $\Csc$ to be a category satisfying $\Ob(\Csc)=\Ob(\GrSet)$, whose morphisms are equivalence classes of morphisms in $\GrSet$ with respect to an equivalence relation defined later (Definition \ref{DefC}).

In order to make this construction work well, we use a formalism of 2-categories.
We add a class of 2-cells to $\GrSet$, so as to make it into a 2-category $\Sbb$. With this view, from now on we regard an object in $\GrSet$ as a 0-cell in $\Sbb$, and a morphism in $\GrSet$ as a 1-cell in $\Sbb$.

\begin{dfn}\label{Def2cell}
Let $\althh,\althhp\co\xg\to\yh$ be any pair of 1-cells. A 2-cell $\ep\co\althh\tc\althhp$ is a map
\[ \ep\co X\to H\ ;\ x\mapsto \ep_x \]
satisfying
\begin{itemize}
\item[{\rm (i)}] $\al\ppr(x)=\ep_x\ax $,
\item[{\rm (ii)}] $\ep_{gx}\thh_x(g)\ep_x\iv=\thh\ppr_x(g)$
\end{itemize}
for any $x\in X$ and $g\in G$.

If we are given a consecutive pair of 2-cells
\[
\xy
(-14,0)*+{\xg}="0";
(14,0)*+{\yh}="2";
{\ar@/^2.0pc/^{\althh} "0";"2"};
{\ar|*+{_{\althhp}} "0";"2"};
{\ar@/_2.0pc/_{\frac{\al\pprr}{\thh\pprr}} "0";"2"};
{\ar@{=>}^{\ep} (0,6);(0,3)};
{\ar@{=>}^{\ep\ppr} (0,-3);(0,-6)};
\endxy
\]
then their vertical composition $\ep\ppr\cdot\ep\co \althh\tc \frac{\al\pprr}{\thh\pprr}$ is defined by
\[ (\ep\ppr\cdot\ep)_x=\ep\ppr_x\ep_x\quad(\fa x\in X). \]
This becomes indeed a 2-cell, since we have
\[ \al\pprr(x)=\ep\ppr_x\al\ppr(x)=\ep\ppr_x\ep_x\ax, \]
\[ \ep\ppr_{gx}\ep_{gx}\thh_x(g)\ep_x\iv\ep^{\prime-1}_x=\ep\ppr_{gx}\thh\ppr_x(g)\ep^{\prime-1}_x=\thh\pprr_x(g) \]
for any $x\in X$ and $g\in G$.

Associativity of this vertical composition is trivially satisfied. The identity 2-cell $\id\co\id_{\althh}\tc\id_{\althh}$ is given by
$\id_x=e\ (\fa x\in X)$.
\end{dfn}

\begin{rem}
In the above definition, if $\xg=\emptyset$ and $\althh=\althhp$ is the unique morphism $\emptyset\to\yh$, then the 2-cell between them is also unique, which is regarded as the identity 2-cell.
\end{rem}

\begin{rem}\label{Rem2cell}
For any 2-cell $\ep\co\althh\tc\althhp$ (as in the notation in Definition \ref{Def2cell}), we have the following.
\begin{enumerate}
\item $\ep$ is invertible with respect to the vertical composition. Indeed, its inverse $\ep\iv\co\althhp\tc\althh$ is given by
\[ (\ep\iv)_x=\ep_x\iv\in H\quad(\fa x\in X) \]
where $\ep_x\iv$ is the inverse element of $\ep_x$ in $H$.
\item $\ep$ preserves orbits. Namely, for any $G$-orbit $X_0\subseteq X$, its images $\al(X_0)$ and $\al\ppr(X_0)$ are contained in the same $H$-orbit in $Y$.
\item For any $x\in X$, the group homomorphisms
\begin{eqnarray*}
&\thh_x\co G_x\to H_{\ax},&\\
&\thh\ppr_x\co G_x\to H_{\al\ppr(x)}&
\end{eqnarray*}
obtained in Remark \ref{RemGrSet} are related by the conjugation by $\ep_x\in H$. In fact we have the following commutative diagram of group homomorphisms.
\[
\xy
(-10,0)*+{G_x}="0";
(13,0)*+{}="1";
(6,8)*+{H_{\ax}}="2";
(6,-8)*+{H_{\al\ppr(x)}}="4";
{\ar^{\thh_x} "0";"2"};
{\ar_{\thh\ppr_x} "0";"4"};
{\ar^{\sigma_{\ep_x}} "2";"4"};
{\ar@{}|\circlearrowright "0";"1"};
\endxy
\]
\end{enumerate}
\end{rem}

\begin{dfn}\label{DefHoriz}
Let $\xg\ov{\althh}{\lra}\yh\ov{\tab}{\lra}\zk$ be a sequence of 1-cells.
\begin{enumerate}
\item For a 2-cell
\[
\xy
(-14,0)*+{\xg}="0";
(14,0)*+{\yh}="2";
{\ar@/^1.2pc/^{\althh} "0";"2"};
{\ar@/_1.2pc/_{\althhp} "0";"2"};
{\ar@{=>}^{\ep} (0,2);(0,-2)};
\endxy,
\]
define $(\tab)\ci\ep\co(\tab)\ci(\althh)\tc(\tab)\ci(\althhp)$ by
\begin{equation}\label{EqHor1}
((\tab)\ci\ep)_x=\tau_{\ax}(\ep_x)\quad(\fa x\in X).
\end{equation}
\item For a 2-cell
\[
\xy
(-14,0)*+{\yh}="0";
(14,0)*+{\zk}="2";
{\ar@/^1.2pc/^{\tab} "0";"2"};
{\ar@/_1.2pc/_{\tabp} "0";"2"};
{\ar@{=>}^{\delta} (0,2);(0,-2)};
\endxy,
\]
define $\delta\ci(\althh)\co(\tab)\ci(\althh)\tc(\tabp)\ci(\althh)$ by
\begin{equation}\label{EqHor2}
(\delta\ci(\althh))_x=\delta_{\ax}\quad(\fa x\in X).
\end{equation}
\end{enumerate}
\end{dfn}

\begin{rem}\label{RemHoriz}
By the same abbreviation as in Remark \ref{RemGrSet}, we abbreviate $(\tab)\ci\ep$ and $\delta\ci(\althh)$ to $\be\ci\ep$ and $\delta\ci\al$.
Thus equations $(\ref{EqHor1}), (\ref{EqHor2})$ are written as
\[ (\be\ci\ep)_x=\tau_{\ax}(\ep_x),\ \ (\delta\ci\al)_x=\delta_{\ax}\quad(\fa x\in X). \]
\end{rem}

\begin{claim}\label{ClaimHoriz}
In the notation in Definition \ref{DefHoriz}, the following holds.
\begin{enumerate}
\item $\be\ci\ep\co\be\ci\al\tc\be\ci\al\ppr$ is in fact a 2-cell.
\item $\delta\ci\al\co\be\ci\al\tc\be\ppr\ci\al$ is in fact a 2-cell.
\end{enumerate}
\end{claim}
\begin{proof}
{\rm (1)} For any $x\in X$, we have
\[ (\be\ci\al\ppr)(x)=\be(\ep_x\ax)=(\be\ci\ep)_x\cdot(\be\ci\ax). \]
Since the equality $\tau_{\ax}(\ep_x)\cdot\tau_{\al\ppr(x)}(\ep\iv_x)=e$ implies $\tau_{\ax}(\ep_x)\iv=\tau_{\al\ppr(x)}(\ep\iv_x)$, we have
\begin{eqnarray*}
(\be\ci\ep)_{gx}\cdot(\tau\ci\thh)_x(g)\cdot(\be\ci\ep)\iv_x&=&\tau_{\al(gx)}(\ep_{gx})\cdot\tau_{\ax}(\thh_x(g))\cdot\tau_{\ax}(\ep_x)\iv\\
&=&\tau_{\ep\iv_x\al\ppr(x)}(\ep_{gx}\thh_x(g))\cdot\tau_{\al\ppr(x)}(\ep_x\iv)\\
&=&(\tau\ci\thh\ppr)_x(g)
\end{eqnarray*}
for any $x\in X$ and $g\in G$.

{\rm (2)} This is also straightforward.
\end{proof}

\begin{lem}
The whiskering defined in Definition \ref{DefHoriz} satisfies the following.
\begin{enumerate}
\item {[}Compatibility with vertical compositions] For any diagram
\[
\xy
(-16,0)*+{\xg}="-2";
(-1,0)*+{\yh}="0";
(18,0)*+{\zk}="2";
{\ar_{\althh} "-2";"0"};
{\ar@/^1.5pc/ "0";"2"};
{\ar "0";"2"};
{\ar@/_1.5pc/ "0";"2"};
{\ar@{=>}^{\delta} (8.5,5);(8.5,1.5)};
{\ar@{=>}^{\delta\ppr} (8.5,-1.5);(8.5,-5)};
\endxy,
\]
in $\Sbb$, we have $(\delta\ppr\cdot\delta)\ci\althh=(\delta\ppr\ci\althh)\cdot(\delta\ci\althh)$. Similarly on the other side, namely for
\[
\xy
(16,0)*+{\zk}="-2";
(1,0)*+{\yh}="0";
(-18,0)*+{\xg}="2";
{\ar_{\tab} "0";"-2"};
{\ar@/^1.5pc/ "2";"0"};
{\ar "2";"0"};
{\ar@/_1.5pc/ "2";"0"};
{\ar@{=>}^{\ep} (-8.5,5);(-8.5,1.5)};
{\ar@{=>}^{\ep\ppr} (-8.5,-1.5);(-8.5,-5)};
\endxy.
\]
\item {[}Associativity] For any diagram
\[
\xy
(-26,0)*+{\xg}="0";
(-7,0)*+{\yh}="2";
(7,0)*+{\zk}="4";
(21,0)*+{\wl}="6";
{\ar@/^1.2pc/ "0";"2"};
{\ar@/_1.2pc/ "0";"2"};
{\ar_{\tab} "2";"4"};
{\ar_{\frac{\gamma}{\mu}} "4";"6"};
{\ar@{=>}^{\delta} (-16.5,2);(-16.5,-2)};
\endxy,
\]
in $\Sbb$, we have $(\frac{\gamma}{\mu}\ci\tab)\ci\delta=\frac{\gamma}{\mu}\ci(\tab\ci\delta)$. Similarly for the case where the 2-cell is in the middle or on the left.
\item {[}Unicity] For any diagram
\[
\xy
(-21,0)*+{\xg}="0";
(-8,0)*+{\xg}="2";
(8,0)*+{\yh}="4";
(21,0)*+{\yh}="6";
{\ar_{\id} "0";"2"};
{\ar@/^1.2pc/ "2";"4"};
{\ar@/_1.2pc/ "2";"4"};
{\ar_{\id} "4";"6"};
{\ar@{=>}^{\delta} (0,2);(0,-2)};
\endxy
\]
in $\Sbb$, we have $\id\ci\delta=\delta=\delta\ci\id$.
\end{enumerate}
\end{lem}
\begin{proof}
This is straightforward.
\end{proof}

To show that the category $\GrSet$ together with these 2-cells forms a 2-category $\Sbb$, it remains to show the following.
\begin{prop}\label{PropHoriz}
For any diagram
\begin{equation}\label{DiagDE}
\xy
(-28,0)*+{\xg}="0";
(0,0)*+{\yh}="2";
(28,0)*+{\zk}="4";
{\ar@/^1.2pc/^{\althh} "0";"2"};
{\ar@/_1.2pc/_{\althhp} "0";"2"};
{\ar@/^1.2pc/^{\tab} "2";"4"};
{\ar@/_1.2pc/_{\tabp} "2";"4"};
{\ar@{=>}^{\ep} (-14,2);(-14,-2)};
{\ar@{=>}^{\delta} (14,2);(14,-2)};
\endxy
\end{equation}
where $\ep$ and $\tau$ are 2-cells, we have
\[ (\delta\ci\al\ppr)\cdot(\be\ci\ep)=(\be\ppr\ci\ep)\cdot(\delta\ci\al). \]
Namely, the following diagram of 2-cells is commutative.
\[
\xy
(-10,6)*+{\be\ci\al}="0";
(10,6)*+{\be\ci\al\ppr}="2";
(-10,-6)*+{\be\ppr\ci\al}="4";
(10,-6)*+{\be\ppr\ci\al\ppr}="6";
{\ar@{=>}^{\be\ci\ep} "0";"2"};
{\ar@{=>}_{\delta\ci\al} "0";"4"};
{\ar@{=>}^{\delta\ci\al\ppr} "2";"6"};
{\ar@{=>}_{\be\ppr\ci\ep} "4";"6"};
{\ar@{}|\circlearrowright "0";"6"};
\endxy
\]

\end{prop}
\begin{proof}
Since $\delta\co\tab\tc\tabp$ is a 2-cell, it satisfies $\delta_{hy}\tau_y(h)=\tau\ppr_y(h)\delta_y$ for any $y\in Y$ and $h\in H$.
Thus we obtain
\begin{eqnarray*}
(\delta\ci\al\ppr)_x\cdot(\be\ci\ep)_x&=&\delta_{\al\ppr(x)}\tau_{\ax}(\ep_x)\ =\ \delta_{\ep_x\ax}\tau_{\ax}(\ep_x)\\
&=&\tau\ppr_{\ax}(\ep_x)\delta_{\ax}\ =\ (\be\ppr\ci\ep)_x\cdot(\delta\ci\al)_x
\end{eqnarray*}
for any $x\in X$.
\end{proof}

By Proposition \ref{PropHoriz}, we define horizontal composition $\delta\ci\ep$ of 2-cells $\delta$ and $\ep$ (as in diagram $(\ref{DiagDE})$) by
\[ (\delta\ci\ep)_x=\delta_{\al\ppr(x)}\tau_{\ax}(\ep_x)=\tau\ppr_{\ax}(\ep_x)\delta_{\ax}\quad(\fa x\in X). \]

The arguments so far allow us the following definition.
\begin{dfn}\label{DefS}
2-category $\Sbb$ is defined as follows.
\begin{enumerate}
\item[{\rm (0)}] $\Sbb^0=\Ob(\GrSet)$.
\item[{\rm (1)}] For any 0-cells $\xg$ and $\yh$,
\[ \Sbb^1(\xg,\yh)=\GrSet(\xg,\yh). \]
\item[{\rm (2)}] For any 1-cells $\althh,\althhp\co\xg\to\yh$, 2-cells $\ep\co\althh\tc\althhp$ are those defined in Definition \ref{Def2cell}. Thus any 2-cell in $\Sbb$ is invertible with respect to the vertical composition.
\end{enumerate}
\end{dfn}

\begin{dfn}\label{DefAdjEq}
Let $\Cbb$ be a strict 2-category. 
\begin{enumerate}
\item A 1-cell $\al\co X\to Y$ is called an {\it equivalence} if there is a 1-cell $\be\co Y\to X$ and invertible 2-cells
\[ \eta\co\id_X\tc\be\ci\al,\ \ \ep\co\al\ci\be\tc\id_Y. \]
$\be$ is called a {\it quasi-inverse} of $\alpha$.
\item A 1-cell $\al\co X\to Y$ is called an {\it isomorphism} if there is a 1-cell $\be\co Y\to X$ which satisfies
\[ \be\ci\al=\id_X,\ \ \al\ci\be=\id_Y. \]
Note that this is equivalent to say that $\al$ is an isomorphism in the category($=\GrSet$, in the case of $\Sbb=\Cbb$) obtained by forgetting the 2-cells in $\Cbb$. Obviously, an isomorphism is in particular an equivalence.

\item  A quadruple $(\al,\be,\eta,\ep)$ as in the following diagram
\[
\xy
(-40,0)*+{X}="0";
(-20,0)*+{Y}="2";
(0,0)*+{X}="4";
(20,0)*+{Y}="6";
%
{\ar^{\al} "0";"2"};
{\ar^{\be} "2";"4"};
{\ar_{\al} "4";"6"};
%
{\ar@/_1.8pc/_{\id} "0";"4"};
{\ar@/^1.8pc/^{\id} "2";"6"};
%
{\ar@{=>}^{\eta} (-20,-6);(-20,-3)};
{\ar@{=>}^{\ep} (0,3);(0,6)};
\endxy
\]
where $\eta$ and $\ep$ are invertible, is called an {\it adjoint equivalence} if it satisfies
\[ (\ep\ci\al)\cdot(\al\ci\eta)=\id_{\al}\quad\text{and}\quad(\be\ci\ep)\cdot(\eta\ci\be)=\id_{\be}. \]
\end{enumerate}
\end{dfn}

\begin{rem}\label{AdjEqPlus}(cf. \cite[P.155]{JS})
An equivalence $\al\co X\to Y$ is always a part of an adjoint equivalence. In fact, if there exist a 1-cell $\be\co Y\to X$ and invertible 2-cells $\eta\co\id_X\tc\be\ci\al,\ \ep\ppr\co\al\ci\be\tc\id_Y$, then the 2-cell
\[ \ep=\ep\ppr\cdot(\al\ci\eta\iv\ci\be)\cdot(\al\ci\be\ci\ep^{\prime-1}) \]
is shown to give an adjoint equivalence $(\al,\be,\eta,\ep)$.
\end{rem}

\begin{prop}\label{Rem0519_2}
$\ \ $
\begin{enumerate}
\item A $G$-equivariant 1-cell $\frac{\al}{G}\co\xg\to\frac{Y}{G}$ is an isomorphism if $\al\co X\ov{\cong}{\lra}Y$ is an isomorphism in $\Gs$. (See also Corollary \ref{CorAddRev}.)
\item For a 1-cell $\frac{\pt}{f}\co \frac{\pt}{G}\to\frac{\pt}{H}$, the following are equivalent.
\begin{itemize}
\item[{\rm (i)}] $f$ is an isomorphism of groups.
\item[{\rm (ii)}] $\frac{\pt}{f}$ is an isomorphism in $\Sbb$.
\item[{\rm (iii)}] $\frac{\pt}{f}$ is an equivalence in $\Sbb$.
\end{itemize}
\end{enumerate}
\end{prop}
\begin{proof}
{\rm (1)} follows from Proposition \ref{PropFunctEqui}. 

{\rm (2)} {\rm (i)}$\Rightarrow${\rm (ii)} follows from Proposition \ref{PropFunctHom}. It remains to show {\rm (iii)}$\Rightarrow${\rm (i)}. Suppose $\frac{\pt}{f}\co\frac{\pt}{G}\to\frac{\pt}{H}$ has a quasi-inverse $\frac{\pt}{q}\co\frac{\pt}{H}\to\frac{\pt}{G}$. By the existence of 2-cells $\frac{\pt}{q}\ci\frac{\pt}{f}\tc \id$ and $\frac{\pt}{f}\ci\frac{\pt}{q}\tc\id$, there are $g\in G$ and $h\in H$ which satisfy
\[ \sigma_g\ci q\ci f=\id_G,\ \ \sigma_h\ci f\ci q=\id_H. \]
This implies
\[ \sigma_g\ci q\ci\sigma_h^{-1}=\sigma_g\ci q\ci \sigma_h^{-1}\ci(\sigma_h\ci f\ci q)=(\sigma_g\ci q\ci f)\ci q=q, \]
namely $\sigma_g\ci q=q\ci \sigma_h$. If we put $r=\sigma_g\ci q$, then
\begin{eqnarray*}
&r\ci f=\sigma_g\ci q\ci f=\id_G,&\\
&f\ci r=f\ci q\ci \sigma_h=\id_H&
\end{eqnarray*}
holds, which means $f\co G\to H$ is a group isomorphism.
\end{proof}

\begin{dfn}\label{DefC}
Define $\Csc$ by $\Csc=\Sbb/\text{{\it 2-cells}}$ (Definition \ref{DefClass}).
As in Definition \ref{DefClass}, the equivalence class of $\althh$ is denoted by $\und{\big(\althh\big)}$, or simply by $\und{\al}$.
\end{dfn}

\begin{rem}\label{RemSC}
$\ \ $
\begin{enumerate}
\item A 1-cell $\althh\co\xg\to\yh$ is an equivalence in $\Sbb$ if and only if $\und{\big(\althh\big)}$ is an isomorphism in $\Csc$.
\item There is a functor $\GrSet\to\Csc$ which sends $\althh\co\xg\to\yh$ in $\GrSet$ to $\und{\Big(\althh\Big)}\co\xg\to\yh$ in $\Csc$.
\end{enumerate}
\end{rem}

\section{First properties of $\Sbb$ and $\Csc$}
In this section, we investigate first categorical properties satisfied by $\Sbb$ and $\Csc$.

\subsection{$\Ind$-equivalence}

\begin{dfn}\label{DefInd}
Let $\iota\co H\hookrightarrow G$ be a monomorphism of groups.
For any $X\in\Ob(\Hs)$, we define $\Ind_{\iota}X\in\Ob(\Gs)$ by
\[ \Ind_{\iota}X=(G\times X)/\sim, \]
where the equivalence relation $\sim$ is defined by
\begin{itemize}
\item[-] $(\xi,x)$ and $(\xi\ppr,x\ppr)$ in $G\times X$ are equivalent if there exists $h\in H$ satisfying
\[ x\ppr=hx,\ \ \xi=\xi\ppr\iota(h). \]
\end{itemize}
We denote the equivalence class of $(\xi,x)$ by $[\xi,x]\in\Ind_{\iota}X$. The $G$-action on $\Ind_{\iota}X$ is defined by
\[ g[\xi,x]=[g\xi,x] \]
for any $g\in G$ and $[\xi,x]\in\Ind_{\iota}X$.
\end{dfn}

\begin{prop}\label{PropIndEquiv}
Let $\iota\co H\hookrightarrow G$ be a monomorphism of groups.
For any $X\in\Ob(\Hs)$, if we define a map $\ups\co X\to\Ind_{\iota}X$ by
\[ \ups(x)=[e,x]\quad(\fa x\in X), \]
then the 1-cell
\[ \frac{\ups}{\iota}\co\frac{X}{H}\to \frac{\Ind_{\iota}X}{G} \]
becomes an equivalence.
\end{prop}
\begin{proof}
It can be easily checked that $\frac{\ups}{\iota}$ is in fact a 1-cell.
We construct a quasi-inverse of $\ups$. Take a coset decomposition of $G$ by $\iota(H)$
\[ G=g_1\iota(H)\am \cdots \am g_s\iota(H) \]
with $g_1,\ldots,g_s\in G$, satisfying $g_1=e$.
Then for any $g\in G$, there uniquely exist $1\le i\le s$ and $h\in H$ satisfying $g=g_i\iota(h)$. We denote these by
\[ a(g)=g_i,\ \ b(g)=h \]
for each $g\in G$.
This gives maps $a\co G\to G$ and $b\co G\to H$, which satisfy
\begin{eqnarray*}
&g=a(g)\iota(b(g)),\ \ b(g\iota(h))=b(g)h,&\\
&a(g\iota(h))=a(g),\ \ b(a(g))=e\ \ \ \ &
\end{eqnarray*}
for any $g\in G$ and $h\in H$.

We define $\althh\co \frac{\Ind_{\iota}X}{G}\to \frac{X}{H}$ by
\begin{eqnarray*}
\al([\xi,x])=b(\xi)x\quad\ \  &&(\fa [\xi,x]\in \Ind_{\iota}X),\\
\thh_{[\xi,x]}(g)=b(g\xi)\cdot b(\xi)\iv&&(\fa [\xi,x]\in \Ind_{\iota}X,\fa g\in G).
\end{eqnarray*}
It can be easily checked that $\al([\xi,x])$ and $\thh_{[\xi,x]}(g)$ are well-defined, independently from the choice of a representative of $[\xi,x]$.

It suffices to show the following. 
\begin{claim}
$\ \ $
\begin{enumerate}
\item $\althh\co \frac{\Ind_{\iota}X}{G}\to \frac{X}{H}$ is a 1-cell.
\item $(\althh)\ci(\frac{\ups}{\iota})=\id_{\frac{X}{H}}$.
\item There exists a 2-cell $\ep\co(\frac{\ups}{\iota})\ci(\althh)\tc\id_{\frac{\Ind_{\iota}X}{G}}$.
\end{enumerate}
\end{claim}
\begin{proof}
{\rm (1)} For any $[\xi,x]\in\Ind_{\iota}X$ and $g,g\ppr\in G$, we have
\begin{eqnarray*}
\al(g[\xi,x])&=&\al([g\xi,x])\ =\ b(g\xi)x\\
&=&b(g\xi)b(\xi)\iv b(\xi)x\ =\ \thh_{[\xi,x]}(g)\cdot\al([\xi,x]),
\end{eqnarray*}
\begin{eqnarray*}
\thh_{[\xi,x]}(gg\ppr)&=&b(gg\ppr\xi)b(\xi)\iv\ =\ b(gg\ppr\xi)b(g\ppr\xi)\iv \cdot b(g\ppr\xi)b(\xi)\iv\\
&=&\thh_{g\ppr [\xi,x]}(g)\cdot\thh_{[\xi,x]}(g\ppr).
\end{eqnarray*}

{\rm (2)} For any $x\in X$ and $h\in H$, we have $\al\ci\ups(x)=\al([e,x])=x$ and $(\thh\ci\iota)_x(h)=\thh_{[e,x]}(\iota(h))=b(\iota(h))\cdot b(e)\iv=h$.

{\rm (3)} Define $\ep$ by $\ep_{[\xi,x]}=a(\xi)\ (\fa [\xi,x]\in\Ind_{\iota}X)$. This is well-defined, and satisfies
\[ \ep_{[\xi,x]}\cdot(\ups\circ\al([\xi,x]))=a(\xi)[e,b(\xi)x]=[a(\xi)\iota(b(\xi)),x]=[\xi,x], \]
\[ \ep_{g[\xi,x]}\cdot(\iota\ci\thh)_{[\xi,x]}(g)\cdot\ep\iv_{[\xi,x]}=a(g\xi)\cdot\iota(b(g\xi))\cdot \iota(b(\xi))\iv a(\xi)\iv=(g\xi)\xi\iv=g \]
for any $[\xi,x]\in\Ind_{\iota}X$ and $g\in G$. Thus $\ep$ gives a 2-cell $\ep\co (\frac{\ups}{\iota})\ci(\althh)\tc \id_{\frac{\Ind_{\iota}X}{G}}$. 
%
(We can also confirm that $(\frac{\ups}{\iota},\althh,\id,\ep)$ is in fact an adjoint equivalence.)
\end{proof}
\end{proof}

\begin{rem}\label{Rem0519_1}
The equivalence in Proposition \ref{PropIndEquiv} can be thought of as \lq\lq{\it reduction of the fraction}": For any sequence of subgroups $K\le H\le G$, we have equivalences
\[ \frac{(G/K)}{G}\simeq\frac{(H/K)}{H}\simeq\frac{(K/K)}{K}=\frac{\pt}{K}. \]
\end{rem}

\begin{cor}\label{CorAdded1}
Let $G$ be a finite group and let $X$ be a transitive finite $G$-set. Then there exists a finite group $H$ and an equivalence $\frac{\pt}{H}\simeq\xg$. This $H$ is unique up to group isomorphism.
\end{cor}
\begin{proof}
For any $x\in X$, existence of an equivalence $\frac{\pt}{G_x}\simeq\frac{X}{G}$ follows from Remark \ref{Rem0519_1}. Uniqueness follows from Proposition \ref{Rem0519_2}.
\end{proof}

\begin{cor}\label{CorIndIso}
Under the same assumption as in Proposition \ref{PropIndEquiv},
\[ \und{\Big(\frac{\ups}{\iota}\Big)}\co \frac{X}{H}\to \frac{\Ind_{\iota}X}{G} \]
gives an isomorphism in $\Csc$.
\end{cor}

\subsection{Bicoproducts and bipullbacks in $\Sbb$}

In this subsection, we show $\Sbb$ has bicoproducts and bipullbacks. These are uniquely determined up to equivalence, which enable us to define the notion of a Mackey functor on $\Sbb$ in the subsequent sections.

From now on, to avoid lack of Greek letters, we usually denote the acting part of 1-cell $\alpha$ by $\thh_{\al}$. Thus the abbreviated expression like \lq\lq {\it Let $\al\co\xg\to\yh$ be a 1-cell}" will mean that a family of maps $\thh_{\al}=\{\thax\co G\to H\}_{x\in X}$ is implicitly given as a part of the defining datum for this 1-cell.

\bigskip

First we give the definition of bicoproducts, biproducts and bipullbacks. Since 2-categorical limits are often defined differently in their level of strictness in several places in the literature, let us precise their definitions.

A bicoproduct is defined as a bicolimit (\cite[Definition 7.4.4]{Borceux}).
\begin{dfn}\label{Def2Coproduct}
Let $\Cbb$ be a 2-category with invertible 2-cells.
For any $A_1$ and $A_2$ in $\Cbb^0$, their {\it bicoproduct} $(A_1\am A_2,\iota_1,\iota_2)$ is defined to be a triplet of $A_1\am A_2\in\Cbb^0$ and
\[ \iota_1\in\Cbb^1(A_1, A_1\am A_2), \ \ \iota_2\in\Cbb^1(A_2, A_1\am A_2), \]
satisfying the following conditions.
\begin{itemize}
\item[{\rm (i)}]
For any $X\in\Cbb^0$ and $f_i\in\Cbb^1(A_i,X)\ (i=1,2)$, there exist $f\in\Cbb^1(A_1\am A_2,X)$ and $\xi_i\in\Cbb^2(f\circ \iota_i,f_i)\ (i=1,2)$ as in the following diagram.
\[
\xy
(0,-6)*+{X}="0";
(-18,10)*+{A_1}="2";
(0,10)*+{A_1\am A_2}="4";
(18,10)*+{A_2}="6";
{\ar_{f_1} "2";"0"};
{\ar_{f} "4";"0"};
{\ar^{f_2} "6";"0"};
{\ar^(0.4){\iota_1} "2";"4"};
{\ar_(0.4){\iota_2} "6";"4"};
{\ar@{=>}_{\xi_1} (-4,6);(-7.5,3)};
{\ar@{=>}^{\xi_2} (4,6);(7.5,3)};
\endxy
\]

\item[{\rm (ii)}]
Given $(X,f_1,f_2)$, for any triplets $(f,\xi_1,\xi_2)$ and $(f\ppr,\xi_1\ppr,\xi_2\ppr)$ as in {\rm (i)}, there exists a unique 2-cell $\eta\in\Cbb^2(f,f\ppr)$ such that $\xi_i\ppr\cdot(\eta\circ \iota_i)=\xi_i\ (i=1,2)$, namely, the following diagram of 2-cells is commutative.
\[
\xy
(-10,6)*+{f\circ \iota_i}="0";
(10,6)*+{f\ppr\circ \iota_i}="2";
(0,-7)*+{f_i}="4";
(0,9)*+{}="6";
{\ar@{=>}^{\eta\circ \iota_i} "0";"2"};
{\ar@{=>}_{\xi_i} "0";"4"};
{\ar@{=>}^{\xi_i\ppr} "2";"4"};
{\ar@{}|\circlearrowright"6";"4"};
\endxy
\]
\end{itemize}
\end{dfn}

\begin{rem}\label{Rem2Coproduct}
$\ \ $
\begin{enumerate}
\item Since 2-cells are invertible, condition {\rm (ii)} only needs to be checked for a fixed $(f,\xi_1,\xi_2)$. Namely, it is equivalent to the following. 
\begin{itemize}
\item[-]
For some fixed triplet $(f,\xi_1,\xi_2)$, for any triplet $(f\ppr,\xi_1\ppr,\xi_2\ppr)$ there exists a unique 2-cell $\eta\in\Cbb^2(f,f\ppr)$ such that $\xi_i\ppr\cdot(\eta\circ \iota_i)=\xi_i\ (i=1,2)$.
\end{itemize}
\item If there are equivalences $\eta_1\co A_1\ov{\simeq}{\lra}A_1\ppr$ and $\eta_2\co A_2\ov{\simeq}{\lra}A_2\ppr$, then an equivalence $A_1\am A_2\ov{\simeq}{\to} A_1\ppr\am A_2\ppr$ is obtained.
In fact, if
\[ A_1\ppr\ov{\ups_1\ppr}{\lra}A_1\ppr\am A_2\ppr\ov{\ups_2\ppr}{\lla}A_2\ppr \]
is a bicoproduct, then
\[ A_1\ov{\ups_1\ppr\ci\eta_1}{\lra}A_1\ppr\am A_2\ppr\ov{\ups_2\ppr\ci\eta_2}{\lla}A_2 \]
gives a bicoproduct.
\end{enumerate}
\end{rem}

\begin{claim}\label{Claim2Coproduct}
Let $\Cbb$ and $A_1,A_2$ be as above. Then the bicoproduct $(A_1\am A_2,\iota_1,\iota_2)$ can be characterized by the following universal property.
\begin{itemize}
\item For any $X\in\Cbb^0$, the functor induced by the composition
\begin{equation}\label{CoprUnivFtr}
(-\ci\iota_1,-\ci\iota_2)\co\Cbb(A_1\am A_2,X)\to\Cbb(A_1,X)\times\Cbb(A_2,X)
\end{equation}
is an equivalence of categories.
\end{itemize}
By this universal property, the bicoproduct is determined uniquely up to equivalence.
\end{claim}
\begin{proof}
In fact, the condition {\rm (i)} in Definition \ref{Def2Coproduct} says $(\ref{CoprUnivFtr})$ is essentially surjective. By using the invertibility of 2-cells, we can easily confirm that the condition {\rm (ii)} is equivalent to the following.
\begin{itemize}
\item For any $X\in\Cbb^0$, any $f,g\in\Cbb^1(A_1\am A_2,X)$ and any pair of 2-cells $\xi_i\co f\ci\iota_i\tc g\ci\iota_i\ (i=1,2)$, there exists a unique 2-cell $\zeta\co f\tc g$ satisfying $\zeta\ci\iota_i=\xi_i\ (i=1,2)$.
\end{itemize}
This means $(\ref{CoprUnivFtr})$ is fully faithful.
\end{proof}

\begin{rem}\label{Rem2CoproductB}
The universal property in Claim \ref{Claim2Coproduct} is an instance of that of bicolimit (= dual notion of bilimit \cite[Definition 7.4.4]{Borceux}). Indeed, let $\Ical$ be the discrete category with two objects $I_1$ and $I_2$, let $F\co\Ical\to\Cbb$ be the functor (or equivalently, strict 2-functor if we regard $\Ical$ as a 2-category with identity 2-cells) determined by $F(I_i)=A_i\ (i=1,2)$.
Then $\iota_1,\iota_2$ induce a functor $\iota\co F\tc\Delta_{A_1\am A_2}$, and $(\ref{CoprUnivFtr})$ becomes equal to the functor
\begin{equation}\label{BicolF}
-\ci\iota\co\Cbb(A_1\am A_2,X)\to2\text{-}\coCone(F,X).
\end{equation}

Here, $\Delta_{A_1\am A_2}\co\Ical\to\Cbb$ denotes the constant functor, and $2\text{-}\coCone(F,X)$ denotes the category of 2-cocones. The equivalence of $(\ref{CoprUnivFtr})$ means $(A_1\am A_2,\iota)$ is a {\it bicolimit} of $F$. This is slightly weaker than the notion of a {\it 2-colimit} of $F$, which requires $(\ref{BicolF})$ to be an isomorphism. (See \cite[Definitions 7.4.1 and 7.4.4]{Borceux} for details.)

Remark that a bilimit is unique up to equivalence, while a (strict) 2-colimit is unique up to isomorphism. We use bicoproducts, mainly because of Corollary \ref{CorCoprodVari}.
\end{rem}

\begin{rem}\label{Rem2CoproductC}
If $A_1\ov{\iota_1}{\lra}A_1\am A_2\ov{\iota_2}{\lla}A_2$ is a bicoproduct in $\Cbb$, then its image $A_1\ov{\und{\iota}_1}{\lra}A_1\am A_2\ov{\und{\iota}_2}{\lla}A_2$ gives a coproduct in $\Cbb/\text{{\it 2-cells}}$.
\end{rem}

\begin{rem}
The {\it biproduct} is defined dually, by reversing the directions of 1-cells. Remark that the directions of 2-cells do not matter, since they are invertible.
\end{rem}

\begin{dfn}\label{Def2Pullback}
Let $\Cbb$ be a 2-category with invertible 2-cells.
For any $A_1,A_2,B\in\Cbb^0$ and $f_i\in\Cbb^1(A_i,B)\ (i=1,2)$, {\it bipullback} of $f_1$ and $f_2$ is defined to be a quartet $(A_1\times_BA_2,\pi_1,\pi_2,\kappa)$ 
as in the diagram
\[
\xy
(-9,6)*+{A_1\times_BA_2}="0";
(9,6)*+{A_2}="2";
(-9,-6)*+{A_1}="4";
(9,-6)*+{B}="6";
{\ar^(0.6){\pi_2} "0";"2"};
{\ar_{\pi_1} "0";"4"};
{\ar^{f_2} "2";"6"};
{\ar_{f_1} "4";"6"};
{\ar@{=>}_{\kappa} (-2,0);(2,0)};
\endxy
,
\]
which satisfies the following conditions.
\begin{itemize}
\item[{\rm (i)}]
For any diagram in $\Cbb$
\[
\xy
(-8,6)*+{X}="0";
(8,6)*+{A_2}="2";
(-8,-6)*+{A_1}="4";
(8,-6)*+{B}="6";
{\ar^{g_2} "0";"2"};
{\ar_{g_1} "0";"4"};
{\ar^{f_2} "2";"6"};
{\ar_{f_1} "4";"6"};
{\ar@{=>}^{\ep} (-2,0);(2,0)};
\endxy
,
\]
there exist $g,\xi_1,\xi_2$ as in the diagram
\[
\xy
(-22,16)*+{X}="-2";
(-8,6)*+{A_1\times_BA_2}="0";
(8,6)*+{A_2}="2";
(-8,-7)*+{A_1}="4";
(8,-7)*+{B}="6";
{\ar^{g} "-2";"0"};
{\ar@/^1.24pc/^{g_2} "-2";"2"};
{\ar@/_1.24pc/_(0.68){g_1} "-2";"4"};
{\ar_(0.66){\pi_2} "0";"2"};
{\ar^{\pi_1} "0";"4"};
{\ar^{f_2} "2";"6"};
{\ar_{f_1} "4";"6"};
{\ar@{=>}_{\kappa} (-1.5,-1);(2.5,-1)};
{\ar@{=>}^{\xi_1} (-12,3);(-16,-1)};
{\ar@{=>}_{\xi_2} (-8,9);(-6,14)};
\endxy
,
\]
satisfying $\ep\cdot(f_1\ci\xi_1)=(f_2\ci\xi_2)\cdot(\kappa\ci g)$, namely making the following diagram of 2-cells commutative.
\[
\xy
(-12,6)*+{f_1\ci\pi_1\ci g}="0";
(12,6)*+{f_2\ci\pi_2\ci g}="2";
(-12,-6)*+{f_1\ci g_1}="4";
(12,-6)*+{f_2\ci g_2}="6";
{\ar@{=>}^{\kappa\ci g} "0";"2"};
{\ar@{=>}_{f_1\ci\xi_1} "0";"4"};
{\ar@{=>}^{f_2\ci\xi_2} "2";"6"};
{\ar@{=>}_{\ep} "4";"6"};
{\ar@{}|\circlearrowright"0";"6"};
\endxy
\]
\item[{\rm (ii)}]
Given $(X,g_1,g_2,\ep)$, for any triplets $(g,\xi_1,\xi_2)$ and $(g\ppr,\xi_1\ppr,\xi_2\ppr)$ as in {\rm (i)}, there exists a unique 2-cell $\zeta\in\Cbb^2(g,g^{\prime})$ which satisfies $\xi_i\ppr\cdot(\pi_i\ci\zeta)=\xi_i\ (i=1,2)$.
\end{itemize}
\end{dfn}
\begin{rem}\label{Rem2Pullback}
The bipullback of $A_1\ov{f_1}{\lra}B\ov{f_2}{\lla}A_2$ is uniquely determined up to equivalence.
Similar properties as in Remark \ref{Rem2Coproduct} are also satisfied by bipullbacks.
\end{rem}

\begin{rem}\label{Rem2PullbackB}
Let $\Cbb$ and $f_i\in\Cbb^1(A_i,B)\ (i=1,2)$ be as in Definition \ref{Def2Pullback}. For any $X\in\Cbb^0$, the functors $\Cbb(X,f_i)\co\Cbb(X,A_i)\to\Cbb(X,B)$ induced by the composition with $f_i$ give a functor
\begin{equation}\label{EFunct}
E\co\Cbb(X,A_1\times_BA_2)\to\Cbb(X,f_1)/\Cbb(X,f_2)
\end{equation}
in a natural way, where $\Cbb(X,f_1)/\Cbb(X,f_2)$ is the comma category (\cite[Definition 1.6.1]{Borceux}).
Definition \ref{Def2Pullback} is saying that $E$ is an equivalence for any $X\in\Cbb^0$, which is the same universal property as that of a bipullback in \cite[P.155]{JS}, since 2-cells are invertible.
If one requires $E$ to be an isomorphism for each $X$, then $A_1\times_BA_2$ is called a {\it pullback} in \cite{Hoffnung}.

\end{rem}

\medskip

$\Sbb$ admits bicoproducts, as follows.
\begin{prop}\label{Prop2CoprodEqui}
Let $G$ be any finite group.
For any $X,Y\in\Ob(\Gs)$, let $X\am Y\in\Ob(\Gs)$ be the usual coproduct of $G$-sets.
If we denote the inclusions by
\[ \ups_X\co X\hookrightarrow X\am Y,\ \ \ups_Y\co Y\hookrightarrow X\am Y, \]
then
\[ \xg\ov{\frac{\ups_X}{G}}{\lra}\frac{X\am Y}{G}\ov{\frac{\ups_Y}{G}}{\lla}\frac{Y}{G} \]
gives a bicoproduct of $\xg$ and $\frac{Y}{G}$ in $\Sbb$.
\end{prop}
\begin{proof}
We confirm conditions {\rm (i)} and {\rm (ii)} in Definition \ref{Def2Coproduct}.

{\rm (i)} Suppose we are given 1-cells
\[ \al\co\xg\to\wl\ \ \text{and}\ \ \be\co \frac{Y}{G}\to \wl \]
to some 0-cell $\wl$. If we take the usual union of maps
\[ \al\cup\be \co X\am Y\to W \]
and the disjoint union of families
\[ \thh_{\al\cup\be}=\thh_{\al}\am\thh_{\be}=\{\thax\}_{x\in X}\am \{\thby \}_{y\in Y}, \]
then it can be easily shown that $\frac{\al\cup\be}{\thh_{\al\cup\be}}\co \frac{X\am Y}{G}\to\wl$ becomes a 1-cell which makes the following diagram commutative.
\[
\xy
(-20,8)*+{\xg}="-2";
(-13,-5)*+{}="-1";
(0,8)*+{\frac{X\amalg Y}{G}}="0";
(13,-5)*+{}="1";
(20,8)*+{\frac{Y}{G}}="2";
(0,-10)*+{\wl}="4";
{\ar^(0.4){\ups_X} "-2";"0"};
{\ar_(0.4){\ups_Y} "2";"0"};
{\ar_{\al} "-2";"4"};
{\ar^{\be} "2";"4"};
{\ar|*+{_{\al\cup\be}} "0";"4"};
{\ar@{}|\circlearrowright "0";"-1"};
{\ar@{}|\circlearrowright "0";"1"};
\endxy
\]

{\rm (ii)} Suppose there also exist a 1-cell $\gamma\co \frac{X\am Y}{G}\to \wl$ and 2-cells $\lam\co\gamma\ci\ups_X\tc\al$, $\rho\co\gamma\ci\ups_Y\tc\be$ as in 
\[
\xy
(-20,9)*+{\xg}="-2";
(-11,-4)*+{}="-1";
(0,9)*+{\frac{X\amalg Y}{G}}="0";
(11,-4)*+{}="1";
(20,9)*+{\frac{Y}{G}}="2";
(0,-9)*+{\wl}="4";
{\ar^(0.4){\ups_X} "-2";"0"};
{\ar_(0.4){\ups_Y} "2";"0"};
{\ar_{\al} "-2";"4"};
{\ar^{\be} "2";"4"};
{\ar^{\gamma} "0";"4"};
{\ar@{=>}^{\lam} (-4.5,4.5);(-8,1)};
{\ar@{=>}_{\rho} (4.5,4.5);(8,1)};
\endxy.
\]
Then the family of maps
\[ \lam\am\rho=\{\lam_x\}_{x\in X}\am\{\rho_y\}_{y\in Y} \]
gives a 2-cell $\lam\am\rho\co\gamma\tc\al\cup\be$, which makes the following diagrams of 2-cells commutative.
\[
\xy
(-16,6)*+{\gamma\ci\ups_X}="0";
(16,6)*+{(\al\cup\be)\ci\ups_X}="2";
(0,-6)*+{\al}="4";
(0,9)*+{}="5";
{\ar@{=>}^(0.44){(\lam\am\rho)\ci\ups_X} "0";"2"};
{\ar@{=>}_{\lam} "0";"4"};
{\ar@{=} "2";"4"};
{\ar@{}|\circlearrowright "4";"5"};
\endxy
\quad
\xy
(-16,6)*+{\gamma\ci\ups_Y}="0";
(16,6)*+{(\al\cup\be)\ci\ups_Y}="2";
(0,-6)*+{\be}="4";
(0,9)*+{}="5";
{\ar@{=>}^(0.44){(\lam\am\rho)\ci\ups_Y} "0";"2"};
{\ar@{=>}_{\rho} "0";"4"};
{\ar@{=} "2";"4"};
{\ar@{}|\circlearrowright "4";"5"};
\endxy
\]
Uniqueness of such a 2-cell can be checked immediately.
\end{proof}

\begin{cor}\label{CorCoprodEqui}
Under the same assumption as in Proposition \ref{Prop2CoprodEqui},
\[ \xg\ov{\und{\ups_X}}{\lra}\frac{X\am Y}{G}\ov{\und{\ups_Y}}{\lla}\frac{Y}{G} \]
gives a coproduct of $\xg$ and $\frac{Y}{G}$ in $\Csc$.
\end{cor}

\begin{prop}\label{Prop2CoprodVari}
Let $\xg$ and $\yh$ be any pair of 0-cells in $\Sbb$. Denote the monomorphisms\begin{eqnarray*}
&G\to G\times H\ ; \ g\mapsto (g,e)&\\
&H\to G\times H\ ; \ h\mapsto (e,h)&
\end{eqnarray*}
by $\iog$ and $\ioh$ respectively, and denote the natural maps
\begin{eqnarray*}
&X\to\Ind_{\iog}X\am\Ind_{\ioh}Y\ ;\ x\mapsto [e,x]\in\Ind_{\iog}X&\\
&Y\to\Ind_{\iog}X\am\Ind_{\ioh}Y\ ;\ y\mapsto [e,y]\in\Ind_{\ioh}Y&
\end{eqnarray*}
by $\ups_X$ and $\ups_Y$.
Then
\[ \xg\ov{\frac{\ups_X}{\iog}}{\lra}\frac{\Ind_{\iog}X\am\Ind_{\ioh}Y}{G\times H}\ov{\frac{\ups_Y}{\ioh}}{\lla}\yh \]
gives a bicoproduct of $\xg$ and $\yh$ in $\Sbb$.
\end{prop}
\begin{proof}
This immediately follows from Propositions \ref{PropIndEquiv}, \ref{Prop2CoprodEqui} and Remark \ref{Rem2Coproduct}.
\end{proof}

\begin{cor}\label{CorCoprodVari}
Under the same assumption as in Proposition \ref{Prop2CoprodVari},
\[ \xg\ov{\und{\ups_X}}{\lra}\frac{\Ind_{\iog}X\am\Ind_{\ioh}Y}{G\times H}\ov{\und{\ups_Y}}{\lla}\yh \]gives a coproduct of $\xg$ and $\yh$ in $\Csc$.
\end{cor}




\medskip

$\Sbb$ admits bipullbacks, as follows.
\begin{prop}\label{Prop2Pullback}
Let $\al\co\xg\to\zk$ and $\be\co\yh\to\zk$ be any pair of 1-cells in $\Sbb$.
Denote the natural projection homomorphisms by
\[ \prg\co G\times H\to G,\ \ \prh\co G\times H\to H. \]
If we
\begin{itemize}
\item[-] put $F=\{(x,y,k)\in X\times Y\times K\mid \be(y)=k\ax \}$, and put
\begin{eqnarray*}
&\wp_X\co F\to X\ ;\ (x,y,k)\mapsto x,&\\
&\wp_Y\co F\to Y\ ;\ (x,y,k)\mapsto y,&
\end{eqnarray*}
\item[-] equip $F$ with a $G\times H$-action
\begin{eqnarray*}
&(g,h)(x,y,k)=(gx,hy,\thby(h) k\thax(g)\iv)&\\
&(\fa (g,h)\in G\times H,\ \ \fa (x,y,k)\in F),&
\end{eqnarray*}
\item[-] define a 2-cell $\kappa\co\al\ci\wp_X\tc\be\ci\wp_Y$ by
\[ \kappa_{(x,y,k)}=k, \]
\end{itemize}
then the diagram
\begin{equation}\label{Diag18_0}
\xy
(-12,7)*+{\frac{F}{G\times H}}="0";
(12,7)*+{\yh}="2";
(-12,-6)*+{\xg}="4";
(12,-6)*+{\zk}="6";
{\ar^(0.52){\frac{\wp_Y}{\prh}} "0";"2"};
{\ar_{\frac{\wp_X}{\prg}} "0";"4"};
{\ar^{\frac{\be}{\thh_{\be}}} "2";"6"};
{\ar_{\frac{\al}{\thh_{\al}}} "4";"6"};
{\ar@{=>}^{\kappa} (-2,0.5);(2,0.5)};
\endxy
\end{equation}
gives a bipullback in $\Sbb$.
\end{prop}
\begin{proof}
For any $f=(x,y,k)\in F$ and $(g,h)\in G\times H$, we have
\begin{eqnarray*}
\wp_X((g,h)f)&=&\wp_X(gx,hy,\thby(h) k\thax(g)\iv)\\
&=&gx\ =\ g \wp_X(f),\\
\wp_Y((g,h)f)&=&hy\ =\ h \wp_Y(f),
\end{eqnarray*}
\[ \be\ci\wp_Y(f)=\be(y)=k\ax=\kappa_f\cdot(\al\ci\wp_X(f)), \]
\begin{eqnarray*}
\kappa_{(g,h)f}\cdot(\thh_{\al}\ci\prg)_f((g,h)))\cdot\kappa_f\iv%
&=&\thby(h) k\thax(g)\iv\cdot\thax(g)\cdot k\iv\\
&=&\thby(h)\ =\ (\thh_{\be}\ci\prh)_f((g,h)),
\end{eqnarray*}
which mean that $\wp_X,\wp_Y$ are 1-cells, and $\kappa$ is a 2-cell.

We confirm conditions {\rm (i), (ii)} in Definition \ref{Def2Pullback}.

{\rm (i)} Suppose we are given a diagram 
\[
\xy
(-8,6)*+{\wl}="0";
(8,6)*+{\yh}="2";
(-8,-6)*+{\xg}="4";
(8,-6)*+{\zk}="6";
{\ar^{\delta} "0";"2"};
{\ar_{\gamma} "0";"4"};
{\ar^{\be} "2";"6"};
{\ar_{\al} "4";"6"};
{\ar@{=>}^{\ep} (-2,0);(2,0)};
\endxy.
\]

If we define
\[ \gad\co W\to F\ \ \text{and}\ \ \ \thh_{\gad}=\{\thh_{\gad,w}\co L\to G\times H\}_{w\in W} \]
by
\begin{eqnarray}
&(\gad)(w)=(\gamma(w),\delta(w),\ep_w),&\label{Eq19_0}\\
&\thh_{\gad,w}=(\thh_{\gamma,w},\thh_{\delta,w})\co L\to G\times H&\label{Eq19_1}
\end{eqnarray}
for any $w\in W$, then it can be confirmed that $\frac{\gad}{\thh_{\gad}}\co\wl\tc \frac{F}{G\times H}$ becomes a 1-cell.
Moreover, we have
\begin{eqnarray}
&(\frac{\wp_X}{\prg})\ci(\frac{\gad}{\thh_{\gad}})=\frac{\gamma}{\thh_{\gamma}},&\label{Eq19_2}\\
&(\frac{\wp_Y}{\prh})\ci(\frac{\gad}{\thh_{\gad}})=\frac{\delta}{\thh_{\delta}},\label{Eq19_3}&
\end{eqnarray}
and the diagram
\[
\xy
(-24,18)*+{\wl}="-2";
(-10,6)*+{\frac{F}{G\times H}}="0";
(10,6)*+{\yh}="2";
(-10,-7)*+{\xg}="4";
(10,-7)*+{\zk}="6";
{\ar^(0.56){\gad} "-2";"0"};
{\ar@/^1.10pc/^{\delta} "-2";"2"};
{\ar@/_1.20pc/_{\gamma} "-2";"4"};
{\ar_(0.56){\wp_Y} "0";"2"};
{\ar^{\wp_X} "0";"4"};
{\ar^{\be} "2";"6"};
{\ar_{\al} "4";"6"};
{\ar@{=>}_{\kappa} (-1.5,-1);(2.5,-1)};
{\ar@{}^{_{\circlearrowright}} (-10,3);(-20,8)};
{\ar@{}_{_{\circlearrowright}} (-8,9);(-10,14)};
\endxy
\]
satisfies
\[ (\kappa\circ(\gad))_w=\kappa_{(\gamma(w),\delta(w),\ep_w)}=\ep_w\quad(\fa w\in W), \]
which means the commutativity of the following diagram.
\begin{equation}\label{Eq19_4}
\xy
(-18,6)*+{\al\ci\wp_X\ci(\gad)}="0";
(18,6)*+{\be\ci\wp_Y\ci(\gad)}="2";
(-18,-6)*+{\al\ci\gamma}="4";
(18,-6)*+{\be\ci\delta}="6";
{\ar@{=>}^{\kappa\ci (\gad)} "0";"2"};
{\ar@{=} "0";"4"};
{\ar@{=} "2";"6"};
{\ar@{=>}_{\ep} "4";"6"};
{\ar@{}|\circlearrowright"0";"6"};
\endxy
\end{equation}

{\rm (ii)} Suppose that the diagram
\[
\xy
(-24,18)*+{\wl}="-2";
(-10,6)*+{\frac{F}{G\times H}}="0";
(10,6)*+{\yh}="2";
(-10,-8)*+{\xg}="4";
(10,-8)*+{\zk}="6";
{\ar^(0.56){\phi} "-2";"0"};
{\ar@/^1.10pc/^(0.48){\delta} "-2";"2"};
{\ar@/_1.32pc/_(0.62){\gamma} "-2";"4"};
{\ar_(0.56){\wp_Y} "0";"2"};
{\ar^{\wp_X} "0";"4"};
{\ar^{\be} "2";"6"};
{\ar_{\al} "4";"6"};
{\ar@{=>}_{\kappa} (-1.5,-1);(2.5,-1)};
{\ar@{=>}^{\mu} (-14.5,3);(-18.5,1)};
{\ar@{=>}_{\nu} (-8,10);(-6.5,15)};
\endxy
\]
also makes the following diagram commutative.
\[
\xy
(-14,6)*+{\al\ci\wp_X\ci\phi}="0";
(14,6)*+{\be\ci\wp_Y\ci\phi}="2";
(-14,-6)*+{\al\ci\gamma}="4";
(14,-6)*+{\be\ci\delta}="6";
{\ar@{=>}^(0.49){\kappa\ci\phi} "0";"2"};
{\ar@{=>}_{\al\ci\mu} "0";"4"};
{\ar@{=>}^{\be\ci\nu} "2";"6"};
{\ar@{=>}_{\ep} "4";"6"};
{\ar@{}|\circlearrowright"0";"6"};
\endxy
\]

Express $\phi$ with its components by $\phi(w)=(x_w,y_w,k_w)\in F$.
Remark that a 2-cell $\zeta\co\phi\tc\gad$, if it exists, makes
\[
\xy
(0,8)*+{\wp_X\ci\phi}="0";
(-12,-6)*+{\wp_X\ci(\gad)}="2";
(12,-6)*+{\gamma}="4";
(0,-9)*+{}="5";
{\ar@{=>}_{\wp_X\ci\zeta} "0";"2"};
{\ar@{=>}^{\mu} "0";"4"};
{\ar@{=} "2";"4"};
{\ar@{}|\circlearrowright "0";"5"};
\endxy
\quad \text{and}\ 
\xy
(0,8)*+{\wp_Y\ci\phi}="0";
(-12,-6)*+{\wp_Y\ci(\gad)}="2";
(12,-6)*+{\delta}="4";
(0,-9)*+{}="5";
{\ar@{=>}_{\wp_Y\ci\zeta} "0";"2"};
{\ar@{=>}^{\nu} "0";"4"};
{\ar@{=} "2";"4"};
{\ar@{}|\circlearrowright "0";"5"};
\endxy
\]
commutative if and only if
\[
\xy
(0,7)*+{L}="0";
(-10,-5)*+{G\times H}="2";
(10,-5)*+{G}="4";
(0,-7)*+{}="5";
{\ar_{\zeta_w} "0";"2"};
{\ar^{\mu_w} "0";"4"};
{\ar_(0.56){\prg} "2";"4"};
{\ar@{}|\circlearrowright "0";"5"};
\endxy
\quad \text{and}\ 
\xy
(0,7)*+{L}="0";
(-10,-5)*+{G\times H}="2";
(10,-5)*+{H}="4";
(0,-7)*+{}="5";
{\ar_{\zeta_w} "0";"2"};
{\ar^{\nu_w} "0";"4"};
{\ar_(0.56){\prh} "2";"4"};
{\ar@{}|\circlearrowright "0";"5"};
\endxy
\]
are commutative for each $w\in W$. Thus there is no other choice than
\[ \zeta_w=(\mu_w,\nu_w)\co L\to G\times H. \]
It is straightforward to show that this $\zeta=\{(\mu_w,\nu_w)\}_{w\in W}$ in fact forms a 2-cell $\zeta\co\phi\tc\gad$.
\end{proof}

\begin{rem}\label{RemBiandStrict}
In the proof of Proposition \ref{Prop2Pullback}, the 1-cell $\frac{\gad}{\thh_{\gad}}\co\wl\to\frac{F}{G\times H}$ satisfying $(\ref{Eq19_2}),(\ref{Eq19_3})$ and the commutativity of $(\ref{Eq19_4})$ is uniquely determined by $(\ref{Eq19_0})$ and $(\ref{Eq19_1})$.
In Remark \ref{Rem2PullbackB}, this shows that the functor $(\ref{EFunct})$
\[ \textstyle{E\co\Sbb(\wl,\frac{F}{G\times H})\to\Sbb(\wl,\al)/\Sbb(\wl,\be)} \]
induces a bijection on objects, and thus $E$ is an isomorphism for any $\wl\in\Sbb^0$. This means $(\ref{Diag18_0})$ is a pullback in the sense of \cite{Hoffnung}.
As a consequence, a pullback of any $\xg\ov{\al}{\lra}\zk\ov{\be}{\lla}\yh$ exists in $\Sbb$, to which a bipullback becomes equivalent.  
\end{rem}

\begin{cor}\label{Cor2Prod}
A biproduct of 0-cells $\xg,\yh$ in $\Sbb$ is given by
\[ \xg\ov{\frac{\wp_X}{\prg}}{\lla}\frac{X\times Y}{G\times H}\ov{\frac{\wp_Y}{\prh}}{\lra}\yh, \]
where $\wp_X\co X\times Y\to X,\ \wp_Y\co X\times Y\to Y$ are the projections.
\end{cor}
\begin{proof}
If we take $\zk=\frac{\pt}{e}$ in Proposition \ref{Prop2Pullback}, then we obtain a biproduct of $\xg$ and $\yh$. In this case, we have a natural identification of $G\times H$-sets $F= X\times Y$.
\end{proof}

\begin{cor}\label{Cor2PBEmpty}
In the notation of Proposition \ref{Prop2Pullback}, if the $K$-orbits generated by $\al(X)$ and $\be(Y)$ in $Z$ are disjoint, namely if
\[ K\al(X)\cap K\be(Y)=\emptyset \]
holds as a subset of $Z$, then the bipullback is given by
\[
\xy
(-8,6)*+{\emptyset}="0";
(8,6)*+{\yh}="2";
(-8,-6)*+{\xg}="4";
(8,-6)*+{\zk}="6";
{\ar_{} "0";"2"};
{\ar^{} "0";"4"};
{\ar^{\be} "2";"6"};
{\ar_{\al} "4";"6"};
{\ar@{}|\circlearrowright "0";"6"};
\endxy.
\]
\end{cor}
\begin{proof}
This immediately follows from Proposition \ref{Prop2Pullback}.
\end{proof}

\begin{caution}
Proposition \ref{Prop2Pullback} does not mean
\begin{equation}\label{DiagC319}
\xy
(-10,7)*+{\frac{F}{G\times H}}="0";
(10,7)*+{\yh}="2";
(-10,-6)*+{\xg}="4";
(10,-6)*+{\zk}="6";
{\ar^(0.52){\und{\wp_Y}} "0";"2"};
{\ar_{\und{\wp_X}} "0";"4"};
{\ar^{\und{\be}} "2";"6"};
{\ar_{\und{\al}} "4";"6"};
{\ar@{}|\circlearrowright "0";"6"};
\endxy
\end{equation}
is a fibered product in $\Csc$.
In fact, this is only a weak fibered product. Namely, the natural map of sets
\begin{equation}\label{Map40}
\textstyle{\Csc(\wl,\frac{F}{G\times H})\to \Csc(\wl,\xg)\underset{\Csc(\wl,\zk)}{\times}\Csc(\wl,\yh)\ ;\ \und{\omega}\mapsto (\und{\wp}_X\ci\und{\omega},\und{\wp}_Y\ci\und{\omega})}
\end{equation}
is surjective for any $\wl\in\Ob(\Csc)$, but not necessarily bijective. For example, let $\iota\co e\to K$ be the unique homomorphism to a finite group $K$. Then for $\al=\be=\frac{\pt}{\iota}\co\frac{\pt}{e}\to\frac{\pt}{K}$, the above diagram $(\ref{DiagC319})$ becomes isomorphic to
\[
\xy
(-10,7)*+{\frac{\pt}{e}\am\cdots\am\frac{\pt}{e}}="0";
(10,7)*+{\frac{\pt}{e}}="2";
(-10,-6)*+{\frac{\pt}{e}}="4";
(10,-6)*+{\frac{\pt}{K}}="6";
{\ar^(0.64){\nabla} "0";"2"};
{\ar_{\nabla} "0";"4"};
{\ar^{\und{\big(\frac{\pt}{\iota}\big)}} "2";"6"};
{\ar_{\und{\big(\frac{\pt}{\iota}\big)}} "4";"6"};
{\ar@{}|\circlearrowright "0";"6"};
\endxy
\]
where $\frac{\pt}{e}\am\cdots\am\frac{\pt}{e}$ is the coproduct of $|K|$-copies of $\frac{\pt}{e}$, and $\nabla$ is the folding morphism (i.e., the unique morphism which induces $\id_{\frac{\pt}{e}}$ on each copy). Then for any $\wl\in\Ob(\Csc)$, the map $(\ref{Map40})$ is obviously surjective. However, it is not injective unless $K=e$.

Nevertheless, by Remark \ref{Rem2Pullback}, these weak fibered products which come from bipullbacks are closed under isomorphisms in $\Csc$, and thus form a natural distinguished class in the whole weak fibered products.
\end{caution}

\begin{dfn}\label{DefNWP}
A weak fibered product in $\Csc$
\[
\xy
(-8,6)*+{\wl}="0";
(8,6)*+{\yh}="2";
(-8,-6)*+{\xg}="4";
(8,-6)*+{\zk}="6";
{\ar^{\und{\delta}} "0";"2"};
{\ar_{\und{\gamma}} "0";"4"};
{\ar^{\und{\be}} "2";"6"};
{\ar_{\und{\al}} "4";"6"};
{\ar@{}|\circlearrowright "0";"6"};
\endxy
\]
is called a {\it natural weak pullback} (of $\und{\al}$ and $\und{\be}$) if it comes from some bipullback in $\Sbb$.
We write as 
\[
\xy
(-8,6)*+{\wl}="0";
(8,6)*+{\yh}="2";
(-8,-6)*+{\xg}="4";
(8,-6)*+{\zk}="6";
(0,0)*+{\nwp}="10";
{\ar^{\und{\delta}} "0";"2"};
{\ar_{\und{\gamma}} "0";"4"};
{\ar^{\und{\be}} "2";"6"};
{\ar_{\und{\al}} "4";"6"};
\endxy
\]
to indicate it is a natural weak pullback.
\end{dfn}

%
%

\begin{prop}\label{PropPullbackAdjEquivEx}
Let $\Cbb$ be a 2-category with invertible 2-cells. If $\al\co X\to Y$ is an equivalence, then
\[
\xy
(-6,6)*+{X}="0";
(6,6)*+{X}="2";
(-6,-6)*+{X}="4";
(6,-6)*+{Y}="6";
{\ar^{\id} "0";"2"};
{\ar_{\id} "0";"4"};
{\ar^{\al} "2";"6"};
{\ar_{\al} "4";"6"};
{\ar@{}|\circlearrowright "0";"6"};
\endxy
\]
is a bipullback.
\end{prop}
\begin{proof}
By Remark \ref{AdjEqPlus}, we take an adjoint equivalence $(\al,\be,\rho\iv,\lambda)$ as in the following diagram.
\[
\xy
(-40,0)*+{X}="0";
(-20,0)*+{Y}="2";
(0,0)*+{X}="4";
(20,0)*+{Y}="6";
{\ar^{\al} "0";"2"};
{\ar^{\be} "2";"4"};
{\ar_{\al} "4";"6"};
{\ar@/_1.8pc/_{\id} "0";"4"};
{\ar@/^1.8pc/^{\id} "2";"6"};
{\ar@{=>}^{\rho} (-20,-3);(-20,-6)};
{\ar@{=>}^{\lam} (0,3);(0,6)};
\endxy
\]
By definition, it satisfies $\al\ci\rho=\lam\ci\al$ and $\be\ci\lam=\rho\ci\be$. We confirm conditions {\rm (i), (ii)} in Definition \ref{Def2Pullback}.

{\rm (i)} Suppose we are given a diagram
\[
\xy
(-7,6)*+{W}="0";
(7,6)*+{X}="2";
(-7,-6)*+{X}="4";
(7,-6)*+{Y}="6";
{\ar^{\delta} "0";"2"};
{\ar_{\gamma} "0";"4"};
{\ar^{\al} "2";"6"};
{\ar_{\al} "4";"6"};
{\ar@{=>}^{\ep} (-2,0);(2,0)};
\endxy
\]
in $\Sbb$.
Then we see that the diagram
\[
\xy
(-24,18)*+{W}="-2";
(-10,6)*+{X}="0";
(-22,1)*+{}="1";
(8,6)*+{X}="2";
(-10,-8)*+{X}="4";
(8,-8)*+{Y}="6";
{\ar^(0.56){\gamma} "-2";"0"};
{\ar@/^1.10pc/^(0.46){\delta} "-2";"2"};
{\ar@/_1.28pc/_{\gamma} "-2";"4"};
{\ar_{\id} "0";"2"};
{\ar^{\id} "0";"4"};
{\ar^{\al} "2";"6"};
{\ar_{\al} "4";"6"};
{\ar@{}|\circlearrowright "0";"6"};
{\ar@{}|\circlearrowright "0";"1"};
{\ar@{=>}_{\eta} (-9,10);(-7,15)};
\endxy
\]
with $\eta=(\rho\ci\delta)\cdot(\be\ci\ep)\cdot(\rho\iv\ci\gamma)$, satisfies
\begin{eqnarray}
\al\ci\eta&=&(\al\ci\rho\ci\delta)\cdot(\al\ci\be\ci\ep)\cdot(\al\ci\rho\iv\ci\gamma)\label{Eq_aleta}\\
&=&(\lam\ci\al\ci\delta)\cdot(\al\ci\be\ci\ep)\cdot(\lam\iv\ci\al\ci\gamma)\ =\ \ep.\nonumber
\end{eqnarray}

{\rm (ii)} Suppose that the diagram
\[
\xy
(-24,18)*+{W}="-2";
(-10,6)*+{X}="0";
(8,6)*+{X}="2";
(-10,-8)*+{X}="4";
(8,-8)*+{Y}="6";
{\ar^(0.56){\pi} "-2";"0"};
{\ar@/^1.10pc/^(0.46){\delta} "-2";"2"};
{\ar@/_1.28pc/_(0.62){\gamma} "-2";"4"};
{\ar_{\id} "0";"2"};
{\ar^{\id} "0";"4"};
{\ar^{\al} "2";"6"};
{\ar_{\al} "4";"6"};
{\ar@{}|\circlearrowright "0";"6"};
{\ar@{=>}_{\xi} (-13,4);(-18,0)};
{\ar@{=>}_{\zeta} (-9,10);(-7,15)};
\endxy
\]
also satisfies
\begin{equation}\label{Eq_epalzeta}
\ep\cdot(\al\ci\xi)=\al\ci\zeta.
\end{equation}
It suffices to show the existence and the uniqueness of $\varpi\co\pi\tc\gamma$ which satisfies
\[ \varpi=\xi\ \ \text{and}\ \ \eta\cdot\varpi=\zeta. \]
Since such $\varpi$ is trivially unique $(=\xi)$, it remains to show that $(\ref{Eq_epalzeta})$ implies $\zeta=\eta\cdot\xi$.

By $(\ref{Eq_aleta})$ and $(\ref{Eq_epalzeta})$, we have
\[ (\al\ci\zeta)=\ep\cdot\ep\iv\cdot(\al\ci\zeta)=(\al\ci\eta)\cdot(\al\ci\xi).     \]
It follows $(\be\ci\al)\ci\zeta=(\be\ci\al)\ci(\eta\cdot\xi)$, and thus
\begin{eqnarray*} \zeta&=&(\rho\ci\delta)\iv\cdot((\be\ci\al)\ci\zeta)\cdot(\rho\ci\pi)\\
&=&(\rho\ci\delta)\iv\cdot((\be\ci\al)\ci(\eta\cdot\xi))\cdot(\rho\ci\pi)\ =\ \eta\cdot\xi.
\end{eqnarray*}

\end{proof}

\begin{prop}\label{PropEqui2Pullback}
Let $G$ be a finite group. If
\[
\xy
(-8,6)*+{X\times_ZY}="0";
(8,6)*+{Y}="2";
(-8,-6)*+{X}="4";
(8,-6)*+{Z}="6";
{\ar^(0.6){\delta} "0";"2"};
{\ar_{\gamma} "0";"4"};
{\ar^{\be} "2";"6"};
{\ar_{\al} "4";"6"};
{\ar@{}|\circlearrowright "0";"6"};
\endxy
\]
is a fibered product in $\Gs$, then
\[
\xy
(-10,6)*+{\frac{X\times_ZY}{G}}="0";
(10,6)*+{\frac{Y}{G}}="2";
(-10,-6)*+{\frac{X}{G}}="4";
(10,-6)*+{\frac{Z}{G}}="6";
{\ar^(0.52){\frac{\delta}{G}} "0";"2"};
{\ar_{\frac{\gamma}{G}} "0";"4"};
{\ar^{\frac{\be}{G}} "2";"6"};
{\ar_{\frac{\al}{G}} "4";"6"};
{\ar@{}|\circlearrowright "0";"6"};
\endxy
\]
is a bipullback in $\Sbb$.
Thus the functor $\frac{\bullet}{G}\co\Gs\to\GrSet$ sends fibered products in $\Gs$ to bipullbacks in $\Sbb$.
\end{prop}
\begin{proof}
Let $\pro^{(1)}\co G\times G\to G$, $\pro^{(2)}\co G\times G\to G$ and $\Delta\co G\to G\times G$ 
be the projections onto 1st and 2nd components, and the diagonal homomorphism respectively.
By Proposition \ref{Prop2Pullback}, we have a bipullback of $\frac{\al}{G}$ and $\frac{\be}{G}$
\[
\xy
(-10,6)*+{\frac{F}{G\times G}}="0";
(10,6)*+{\frac{Y}{G}}="2";
(-10,-6)*+{\frac{X}{G}}="4";
(10,-6)*+{\frac{Z}{G}}="6";
{\ar^(0.52){\frac{\wp_Y}{\pro^{(2)}}} "0";"2"};
{\ar_{\frac{\wp_X}{\pro^{(1)}}} "0";"4"};
{\ar^{\frac{\be}{G}} "2";"6"};
{\ar_{\frac{\al}{G}} "4";"6"};
{\ar@{=>}^{\kappa} (-2,0);(2,0)};
\endxy
\]
as in the notation of Proposition \ref{Prop2Pullback}. Remark that $F$ is defined by
\[ F=\{ (x,y,g)\in X\times Y\times G\mid \be(y)=g\ax \}, \]
on which $G\times G$ acts by
\[ (g_1,g_2)\cdot(x,y,g)=(g_1x,g_2y,g_2gg_1\iv)\quad(\fa (g_1,g_2)\in G\times G,\ \fa (x,y,g)\in F). \]
If we define maps $\pi$ and $\chi$ by
\begin{eqnarray*}
\pi\co F\to X\times_ZY &;& (x,y,g)\mapsto (gx,y)\\
\chi\co X\times_ZY\to F &;& (x,y)\mapsto (x,y,e),
\end{eqnarray*}
then $\frac{\pi}{\pro^{(2)}}$ and $\frac{\chi}{\Delta}$ become 1-cells.

By Remark \ref{Rem2Pullback}, it suffices to show that $\pi$ and $\chi$ give an equivalence $\frac{F}{G\times G}\simeq \frac{X\times_ZY}{G}$. 
It can be easily checked that we have $\pi\ci\chi=\id_{\frac{X\times_ZY}{G}}$.
If we define as
\[ \lam_{f}=(g\iv,e)\quad(\fa f=(x,y,g)\in F), \]
then $\lam\co \chi\ci\pi\tc\id_{\frac{F}{G\times G}}$ gives a 2-cell. 
(We can also confirm that $(\chi,\pi,\id,\lam)$ is in fact an adjoint equivalence.)
\end{proof}

\begin{cor}\label{Cor2PBIncl}
Let $\xg$ be any 0-cell. Let $\iota_1\co X_1\hookrightarrow X$ and $\iota_2\co X_2\hookrightarrow X$ be inclusions of finite $G$-sets. If we denote the inclusions
\[ X_1\cap X_2\hookrightarrow X_1\ \ \text{and}\ \ X_1\cap X_2\hookrightarrow X_2 \]
by $\iota\ppr_1$ and $\iota\ppr_2$ respectively, then
\[
\xy
(-10,6)*+{\frac{X_1\cap X_2}{G}}="0";
(10,6)*+{\frac{X_2}{G}}="2";
(-10,-6)*+{\frac{X_1}{G}}="4";
(10,-6)*+{\frac{X}{G}}="6";
{\ar^(0.52){\frac{\iota_2\ppr}{G}} "0";"2"};
{\ar_{\frac{\iota_1\ppr}{G}} "0";"4"};
{\ar^{\frac{\iota_2}{G}} "2";"6"};
{\ar_{\frac{\iota_1}{G}} "4";"6"};
{\ar@{}|\circlearrowright "0";"6"};
\endxy
\]
gives a bipullback.
Especially, remark that we have the following.
\begin{enumerate}
\item If $X_1=X_2$, then $\iota_1\ppr$ and $\iota\ppr_2$ are identities.
\item If $X_1\cap X_2=\emptyset$, then $\frac{X_1\cap X_2}{G}=\emptyset$.
\end{enumerate}
\end{cor}
\begin{proof}
This immediately follows from Proposition \ref{PropEqui2Pullback}.
\end{proof}

\section{Stabilizerwise image}

As Remark \ref{RemGrSet} suggests, a 1-cell $\al\co\xg\to\yh$ can be thought as a parallel array of group homomorphisms on stabilizers $\thax|_{G_x}\co G_x\to H_{\ax}$.
With this view, we can consider analogs of images of group homomorphisms and factorizations through them, for 1-cells in $\Sbb$.

\subsection{Stab-surjective 1-cells}
The notion of stab-surjective 1-cells, which we now define, can be regarded as an analog of that of surjective group homomorphisms.

\begin{dfn}\label{DefStabsurj}
A 1-cell $\al\co\xg\to\yh$ is called {\it surjective on stabilizers} or shortly {\it stab-surjective}, if the following conditions are satisfied.
\begin{itemize}
\item[{\rm (i)}] For any $y\in Y$, there exist $x\in X$ and $h\in H$ satisfying $y=h\ax$.
\item[{\rm (ii)}] If $x,x\ppr\in X$ and $h,h\ppr\in H$ satisfy $h\ax=h\ppr\al(x\ppr)$, then there exists $g\in G$ which satisfies $x\ppr=gx$ and $h=h\ppr\thax(g)$.
\end{itemize}
\end{dfn}

\begin{rem}\label{RemStabsurj}
If $\al\co\xg\to\yh$ is stab-surjective, then for any $x\in X$, the restriction of $\thax$ onto $G_x$ gives a surjective homomorphism
\[ \thax|_{G_x}\co G_x\to H_{\ax}, \]
by condition {\rm (ii)} in Definition \ref{DefStabsurj}.
\end{rem}

\begin{ex}\label{ExaStabsurj}
Let $G$ be a finite group, and let $N\nm G$ be a normal subgroup. Let
\[ p\co G\to G/N\ ;\ g\mapsto \overline{g} \]
denote the quotient homomorphism.
Then for any $Z\in\Ob((G/N)\text{-}\sett)$, the 1-cell
\[ \frac{\id_Z}{p}\co \frac{\Inf^G_NZ}{G}\to \frac{Z}{(G/N)} \]
is stab-surjective.
Here $\Inf^G_NZ$ denotes the set $Z$, equipped with the $G$-action
\[ gz=\overline{g}z\qquad(\fa g\in G,\fa z\in Z). \]
\end{ex}
\begin{proof}
For any $z\in Z$, we have $e\cdot\id_Z(z)=z$.
Moreover if $z_1,z_2\in Z$ and $\overline{g}_1\overline{g}_2\in G/N$ satisfy $\overline{g}_1z_1=\overline{g}_2z_2$, then $g_2\iv g_1\in G$ satisfies
\[ (g_2\iv g_1)z_1=\overline{g}_2\iv \overline{g}_1z_1=\overline{g}_2\iv \overline{g}_2z_2=z_2,\ \ \text{and}\ \ \ \overline{g}_2\cdot p(g_2\iv g_1)=\overline{g}_1. \]
\end{proof}

\begin{prop}\label{PropStabsurj}
Let $\al\co\xg\to\yh$ be a 1-cell in $\Sbb$.
\begin{enumerate}
\item If there exists a 2-cell $\delta\co\al\ppr\tc\al$ from a stab-surjective 1-cell $\al\ppr\co\xg\to\yh$, then so is $\al$.
Namely the stab-surjectivity does not depend on representatives of the equivalence class $\und{\al}$. Thus we can speak of the stab-surjectivity of morphisms in $\Csc$.
\item If $\al$ is an equivalence, then $\al$ is stab-surjective.
\end{enumerate}
\end{prop}
\begin{proof}
{\rm (1)} By definition, $\al$ and $\al\ppr$ are related by
\begin{eqnarray*}
\ax=\delta_x\al\ppr(x)&&(\fa x\in X),\\
\delta_{gx}\cdot\thh_{\al\ppr,x}(g)\cdot\delta_x\iv=\thax(g)&&(\fa x\in X, \fa g\in G).
\end{eqnarray*}
We confirm conditions {\rm (i), (ii)} in Definition \ref{DefStabsurj}.

{\rm (i)} For any $y\in Y$, there exist $x\in X$ and $h\in H$ satisfying $y=h\al\ppr(x)$. Thus we have $y=(h\delta_x\iv)\ax$.

{\rm (ii)} If $x_1,x_2\in X$ and $h_1,h_2\in H$ satisfy $h_1\al(x_1)=h_2\al(x_2)$, 
then by the stab-surjectivity of $\al\ppr$, there exists $g\in G$ which satisfies 
\[ x_2=gx_1\ \  \text{and}\ \  h_1\delta_{x_1}=h_2\delta_{x_2}\thh_{\al\ppr,x_1}(g). \]
Since $\delta_{x_2}\thh_{\al\ppr,x_1}(g)\delta_{x_1}\iv=\thh_{\al,x_1}(g)$, we obtain $x_2=gx_1$ and $h_1=h_2\thh_{\al,x_1}(g)$.

{\rm (2)} Take a diagram
\[
\xy
(-40,0)*+{\xg}="0";
(-20,0)*+{\yh}="2";
(0,0)*+{\xg}="4";
(20,0)*+{\yh}="6";
{\ar^{\al} "0";"2"};
{\ar^{\be} "2";"4"};
{\ar_{\al} "4";"6"};
{\ar@/_1.8pc/_{\id} "0";"4"};
{\ar@/^1.8pc/^{\id} "2";"6"};
{\ar@{=>}^{\rho} (-20,-3);(-20,-6)};
{\ar@{=>}^{\lam} (0,3);(0,6)};
\endxy
\]
in $\Sbb$, satisfying $\al\ci\rho=\lam\ci\al$ and $\rho\ci\be=\be\ci\lam$.
We confirm conditions {\rm (i) and (ii)} in Definition \ref{DefStabsurj}.

{\rm (i)} For any $y\in Y$, we have $y=\lam_y\cdot\al(\be(y))$.

{\rm (ii)} Suppose $x_1,x_2\in X$ and $h_1,h_2\in H$ satisfy $h_1\al(x_1)=h_2\al(x_2)$. If we put $g_i=\thh_{\be,\al(x_i)}(h_i)\cdot\rho_{x_i}\iv\ (i=1,2)$, then we have
\begin{eqnarray*}
g_1x_1&=&\thh_{\be,\al(x_1)}(h_1)\cdot\rho_{x_1}\iv x_1\ =\ \thh_{\be,\al(x_1)}(h_1)\cdot(\be\ci\al(x_1))\\
&=&\be(h_1\al(x_1))\ =\ \be(h_2\al(x_2))\ =\ \thh_{\be,\al(x_2)}(h_2)\cdot(\be\ci\al(x_2))\\
&=&\thh_{\be,\al(x_2)}(h_2)\cdot\rho_{x_2}\iv x_2\ =\ g_2x_2
\end{eqnarray*}
and we can confirm
\begin{eqnarray*}
\thh_{\al,x_1}(g_2\iv g_1)&=&\thh_{\al,g_1x_1}(g_2\iv)\cdot\thh_{\al,x_1}(g_1)\ =\ \thh_{\al,g_2x_2}(g_2\iv)\cdot\thh_{\al,x_1}(g_1)\\
&=&h_2\iv\cdot\lam_{h_2\al(x_2)}\cdot\lam_{h_1\al(x_1)}\iv\cdot h_1\ =\ h_2\iv h_1.
\end{eqnarray*}
Thus $g=g_2\iv g_1\in G$ satisfies $x_2=gx_1$ and $h_1=h_2\cdot\thh_{\al,x_1}(g)$.
\end{proof}

\begin{prop}\label{PropStabsurjCompos}
Let $\xg\ov{\al}{\lra}\yh\ov{\be}{\lra}\zk$ be a sequence of 1-cells in $\Sbb$.
If $\al$ and $\be$ are stab-surjective, then so is $\be\ci\al$.
\end{prop}
\begin{proof}
We confirm conditions {\rm (i), (ii)} in Definition \ref{DefStabsurj}.

{\rm (i)} For any $z\in Z$, there exist $y\in Y$ and $k\in K$ satisfying $z=k\be(y)$ by the stab-surjectivity of $\be$. Then by the stab-surjectivity of $\al$, there exist $x\in X$ and $h\in H$ satisfying $y=h\al(x)$. Thus we obtain
\[ z=k\be(h\ax)=k\thh_{\be,\ax}(h)\cdot(\be\ci\ax). \]

{\rm (ii)} Suppose $x_1,x_2\in X$ and $k_1,k_2\in K$ satisfy $k_1\be(\al(x_1))=k_2\be(\al(x_2))$.
By the stab-surjectivity of $\be$, there exists $h\in H$ satisfying
\[ \al(x_2)=h\al(x_1)\ \ \ \text{and}\ \ \ k_1=k_2\cdot\thh_{\be,\al(x_1)}(h). \]
Then by the stab-surjectivity of $\al$, there exists $g\in G$ satisfying
\[ x_2=gx_1\ \ \ \text{and}\ \ \ h=\thh_{\al,x_1}(g). \]
Thus we have $k_1=k_2\cdot\thh_{\be,\al(x_1)}(\thh_{\al,x_1}(g))=k_2\cdot\thh_{\be\ci\al,x_1}(g)$.
\end{proof}

\medskip

Stab-surjective 1-cells are stable under bipullbacks, as follows.
\begin{prop}\label{Prop2PBSSurj}
Let
\[
\xy
(-8,6)*+{\wl}="0";
(8,6)*+{\yh}="2";
(-8,-6)*+{\xg}="4";
(8,-6)*+{\zk}="6";
{\ar^{\delta} "0";"2"};
{\ar_{\gamma} "0";"4"};
{\ar^{\be} "2";"6"};
{\ar_{\al} "4";"6"};
{\ar@{=>}^{\ep} (-2,0);(2,0)};
\endxy
\]
be a bipullback in $\Sbb$. If $\be$ is stab-surjective, then so is $\gamma$.
\end{prop}
\begin{proof}
We use the notation in Proposition \ref{Prop2Pullback}.
By Remark \ref{Rem2Pullback}, Propositions \ref{PropStabsurj} and \ref{PropStabsurjCompos}, it suffices to confirm conditions {\rm (i), (ii)} in Definition \ref{DefStabsurj} for $\wp_X$ in the bipullback
\[
\xy
(-10,6)*+{\frac{F}{G\times H}}="0";
(10,6)*+{\yh}="2";
(-10,-6)*+{\xg}="4";
(10,-6)*+{\zk}="6";
{\ar^(0.52){\frac{\wp_Y}{\prh}} "0";"2"};
{\ar_{\frac{\wp_X}{\prg}} "0";"4"};
{\ar^{\frac{\be}{\thh_{\be}}} "2";"6"};
{\ar_{\frac{\al}{\thh_{\al}}} "4";"6"};
{\ar@{=>}^{\kappa} (-2,0);(2,0)};
\endxy
\]
constructed in Proposition \ref{Prop2Pullback}.

\medskip

{\rm (i)} For any $x\in X$, since $\be$ is stab-surjective, there exist $y\in Y$ and $k\in K$ satisfying $\ax=k\be(y)$. Thus we obtain an element $(x,y,k\iv)\in F$, which satisfies $x=\wp_X(x,y,k\iv)$.

{\rm (ii)} Suppose $f_1=(x_1,y_1,k_1), f_2=(x_1,y_2,k_2)\in F$ and $g_1,g_2\in G$ satisfy $g_1\wp_X(f_1)=g_2\wp_X(f_2)$.
Then we have $\thh_{\al,x_1}(g_1)\al(x_1)=\thh_{\al,x_2}(g_2)\al(x_2)$. 
Since $\be$ is stab-surjective, there exists $h\in H$ which satisfies
\[ y_2=hy_1\ \ \ \text{and}\ \ \ \thh_{\al,x_1}(g_1)k_1\iv=\thh_{\al,x_2}(g_2)k_2\iv\thh_{\be,y_1}(h). \]
Then $a=(g_2\iv g_1,h)\in G\times H$ satisfies
\begin{eqnarray*}
a\cdot(x_1,y_1,k_1)&=&(g_2\iv g_1x_1,hy_1,\thh_{\be,y_1}(h)\cdot k_1\cdot\thh_{\al,x_1}(g_2\iv g_1)\iv )\\
&=&(x_2,y_2,\thh_{\be,y_1}(h)k_1\thh_{\al,x_1}(g_1)\iv\thh_{\al,x_2}(g_2))\ =\ (x_2,y_2,k_2)
\end{eqnarray*}
and $g_1=g_2\cdot\prg(a)$.
\end{proof}

\begin{prop}\label{PropStabsurjDecomp}
Let $\al\co\xg\to\yh$ be a stab-surjective 1-cell in $\Sbb$. Let $X=X_1\am\cdots\am X_s$ be the decomposition of $X$ into $G$-orbits.
If we put
\begin{eqnarray*}
Y_i&=&H\al(X_i)\\
&=&\{ h\ax\in Y\mid h\in H, x\in X_i \},
\end{eqnarray*}
then we have the following.
\begin{enumerate}
\item For $i\ne j$, we have $Y_i\cap Y_j=\emptyset$.
\item Each $Y_i$ is $H$-transitive.
\item $Y=Y_1\am\cdots\am Y_s$ gives the decomposition of $Y$ into $H$-orbits. In particular, $X$ and $Y$ have the same number of orbits.
\item For any $1\le i\le s$, the restriction of $\al$
\[ \al_i=\al|_{X_i}\co X_i\to Y_i \]
is stab-surjective.
\item $Y_i$ does not depend on the choice of representatives of $\und{\al}$.
\end{enumerate}
\end{prop}
\begin{proof}
{\rm (1)} If there is an element $y\in Y_i\cap Y_j$ for $i\ne j$, then there exist $x\in X_i, x\ppr\in X_j$ and $h,h\ppr\in H$ satisfying
\[ y=h\al(x)=h\ppr\al(x\ppr). \]
Then by the stab-surjectivity of $\al$, there should be $g\in G$ which satisfies $gx=x\ppr$, which contradicts to the fact that $X_i$ and $X_j$ are distinct $G$-orbits.

{\rm (2)} For any $y,y\ppr\in Y_i$, there exist $x,x\ppr\in X_i$ and $h,h\ppr\in H$ satisfying
\[ y=h\al(x),\ \ y\ppr=h\ppr\al(x\ppr) \]
by definition of $Y_i=H\al(X_i)$. Since $X_i$ is $G$-transitive, there is $g\in G$ satisfying $x\ppr=gx$. Thus we obtain
\[ y\ppr=h\ppr\al(gx)=h\ppr\thax(g)\ax=(h\ppr\thax(g) h\iv)\cdot y. \]

{\rm (3)} By {\rm (1)} and {\rm (2)}, it remains to show $Y=Y_1\cup\cdots\cup Y_s$. However, this is obvious from the stab-surjectivity of $\al$.

{\rm (4)} This is trivial. {\rm (5)} This follows from Remark \ref{Rem2cell}.

\end{proof}

\begin{cor}\label{CorAddRev}
For a $G$-equivariant 1-cell $\frac{\al}{G}\co\xg\to \frac{Y}{G}$, the following are equivalent.
\begin{enumerate}
\item $\al$ is a $G$-equivariant isomorphism.
\item $\al$ is an equivalence.
\item $\al$ is stab-surjective.
\item $\und{\al}$ is an isomorphism in $\Csc$.
\end{enumerate}
\end{cor}
\begin{proof}
{\rm (1)}$\Rightarrow${\rm (2)}$\Rightarrow${\rm (3)} and {\rm (2)}$\Leftrightarrow${\rm (4)} are already shown.
It remains to show {\rm (3)}$\Rightarrow${\rm (1)}.
Suppose $\frac{\al}{G}\co\xg\to \frac{Y}{G}$ is stab-surjective.
By Proposition \ref{PropStabsurjDecomp}, we may assume $X$ and $Y$ are transitive. By Remark \ref{RemStabsurj}, $\al$ induces a surjection on stabilizers. This means $\al$ is isomorphism.
\end{proof}

\begin{prop}\label{PropSSDef}
Let $\al\co\xg\to\yh$ be a 1-cell, where $X$ is $G$-transitive. Then the following are equivalent.
\begin{enumerate}
\item $\al$ is stab-surjective.
\item There exist a section $N\nm G_0\le G$ and $Z\in \Ob((G_0/N)\text{-}\sett)$ and a diagram
\[
\xy
(-10,6)*+{\xg}="0";
(10,6)*+{\yh}="2";
(-10,-6)*+{\frac{\Inf^{G_0}_NZ}{G}}="4";
(10,-6)*+{\frac{Z}{(G_0/N)}}="6";
{\ar^{\al} "0";"2"};
{\ar^{\xi} "4";"0"};
{\ar_{\eta} "6";"2"};
{\ar_{\frac{\id_Z}{p}} "4";"6"};
{\ar@{=>}^{\ep} (-2,0);(2,0)};
\endxy
\]
where
\begin{itemize}
\item[{\rm (i)}] $\xi$ and $\eta$ are equivalences.
\item[{\rm (ii)}] $p\co G_0\to G_0/N$ is the quotient homomorphism.
\end{itemize}

\end{enumerate}

Moreover, $Z$ in {\rm (2)} can be taken as $Z=\pt$.
\end{prop}
\begin{proof}
$(2)\Rightarrow(1)$ follows from Example \ref{ExaStabsurj} and Propositions \ref{PropStabsurj}, \ref{PropStabsurjCompos}. It suffices to show $(1)\Rightarrow(2)$.

Suppose $\al$ is stab-surjective. Remark that $Y$ becomes transitive by Proposition \ref{PropStabsurjDecomp}. Take $x_0\in X$, and put
\[ G_0=G_{x_0},\ \ y_0=\al(x_0),\ \ H_0=H_{y_0}. \]
Then by Remark \ref{RemGrSet2}, we have a commutative diagram
\[
\xy
(-10,6)*+{\xg}="0";
(10,6)*+{\yh}="2";
(-10,-6)*+{\frac{G/G_0}{G}}="4";
(10,-6)*+{\frac{H/H_0}{H}}="6";
{\ar^{\al} "0";"2"};
{\ar^{\cong} "4";"0"};
{\ar_{\cong} "6";"2"};
{\ar_(0.52){\Theta_{\al,x_0}} "4";"6"};
{\ar@{}|\circlearrowright "0";"6"};
\endxy
\]
where the vertical arrows are (equivariant) isomorphisms.
Since we have $G/G_0=\Ind^G_{G_0}(G_0/G_0)$ and $H/H_0=\Ind^H_{H_0}(H_0/H_0)$, there are equivalences
\[ \frac{\pt}{G_0}\ov{\simeq}{\lra}\frac{G/G_0}{G},\ \ \text{and}\ \ \frac{\pt}{H_0}\ov{\simeq}{\lra}\frac{H/H_0}{H} \]
as in Proposition \ref{PropIndEquiv}. Moreover, since $\al$ is stab-surjective, it induces surjective group homomorphism $\thh_{\al,x_0}|_{G_0}\co G_0\to H_0$, which induces the group isomorphism $\eta\co G_0/\Ker (\thh_{\al,x_0}|_{G_0})\ov{\cong}{\lra}H_0$. Thus if we put $N=\Ker (\thh_{\al,x_0}|_{G_0})\nm G_0$, we obtain the following commutative diagram.
\[
\xy
(-12,15)*+{\xg}="0";
(12,15)*+{\yh}="2";
(-12,0)*+{\frac{G/G_0}{G}}="4";
(12,6)*+{\frac{H/H_0}{H}}="5";
(12,-3.5)*+{\frac{\pt}{H_0}}="6";
(-12,-15)*+{\frac{\pt}{G_0}}="8";
(12,-15)*+{\frac{\pt}{(G_0/N)}}="10";
{\ar^{\al} "0";"2"};
{\ar^{\cong} "4";"0"};
{\ar^{\simeq} "8";"4"};
{\ar_{\cong} "5";"2"};
{\ar_{\simeq} "6";"5"};
{\ar_{\cong}^{\frac{\pt}{\eta}} "10";"6"};
{\ar^{\al} "0";"2"};
{\ar_(0.46){\frac{\pt}{p}} "8";"10"};
{\ar@{}|\circlearrowright "0";"10"};
\endxy
\]
\end{proof}

\subsection{Factorization through $\SIm$}

We introduce the notion of the stabilizerwise image, which plays a role analogous to the image of a group homomorphism. This provides the decomposition of any 1-cell in $\Sbb$ into an equivariant 1-cell and a stab-surjective 1-cell.

\begin{dfn}\label{DefSIm}
Let $\al\co\xg\to\yh$ be any 1-cell in $\Sbb$.
\begin{enumerate}
\item Define $\SIm\al=\SIm(\frac{\al}{\thh_{\al}})\in\Ob(\Hs)$ by
\[ \SIm\al=(H\times X)/\sim, \]
where the relation $\sim$ is defined as follows.
\begin{itemize}
\item[-] $(\eta,x),(\eta\ppr,x\ppr)\in H\times X$ are equivalent if there exists $g\in G$ satisfying
\[ x\ppr=gx\ \ \ \text{and}\ \ \ \eta=\eta\ppr\thax(g). \]
\end{itemize}
We denote the equivalence class of $(\eta,x)$ by $[\eta,x]$. The $H$-action on $\SIm\al$ is given by $h[\eta, x]=[h\eta,x]$.
We call $\SIm\al$ the {\it stabilizerwise image} of $\al=\frac{\al}{\thh_{\al}}$.
\item Define a map $\ups_{\al}\co X\to\SIm\al$ by
\[ \ups_{\al}(x)=[e,x]\quad(\fa x\in X) \]
and put $\thh_{\ups_{\al}}=\thh_{\al}$. Then $\ups_{\al}=(\frac{\ups_{\al}}{\thh_{\al}})\co\xg\to \frac{\SIm\al}{H}$ becomes a 1-cell.
\end{enumerate}
\end{dfn}

\begin{prop}\label{PropSImSurj}
For any 1-cell $\al\co\xg\to\yh$, the induced 1-cell $\ups_{\al}\co X\to \SIm\al$ is stab-surjective.
\end{prop}
\begin{proof}
Conditions {\rm (i), (ii)} in Definition \ref{DefStabsurj} are confirmed as follows.

{\rm (i)} For any $[\eta,x]\in\SIm\al$, we have $[\eta,x]=\eta [e,x]=\eta\ups_{\al}(x)$.

{\rm (ii)} If $x,x\ppr\in X$ and $h,h\ppr\in H$ satisfy $h\ups_{\al}(x)=h\ppr\ups_{\al}(x\ppr)$, i.e., $[h,x]=[h\ppr,x\ppr]$, then by definition of $\SIm\al$, there exists $g\in G$ which satisfies $x\ppr=gx$ and $h=h\ppr\thax(g)$.
\end{proof}

\begin{rem}
$\SIm(\frac{\al}{\thh_{\al}})$ essentially depends only on the acting part $\thh_{\al}$.
\end{rem}

\begin{rem}
If $\al$ is $\iota$-equivariant for some monomorphism $\iota\co G\to H$, then $\SIm\al=\SIm(\frac{\al}{\iota})$ is nothing but $\Ind_{\iota}X$. In this case, $\ups_{\al}\co\xg\to \frac{\Ind_{\iota}X}{H}$ is an equivalence, as shown in Proposition \ref{PropIndEquiv}.
\end{rem}

\begin{lem}\label{LemSIm}
Let $\al\co\xg\to\yh$ and $\be\co\xg\to \frac{Z}{H}$ be 1-cells satisfying $\thh_{\al}=\thh_{\be}=\thh$. If we define a map $\wt{\be}\co\SIm\al\to Z$ by
\[ \wt{\be}([\eta,x])=\eta\be(x)\quad(\fa [\eta,x]\in\SIm\al), \]
then we obtain the following commutative diagram of 1-cells.
\[
\xy
(-10,0)*+{\xg}="0";
(12,0)*+{}="1";
(6,8)*+{\frac{\SIm\al}{H}}="2";
(6,-8)*+{\frac{Z}{H}}="4";
{\ar^{\frac{\ups_{\al}}{\thh}} "0";"2"};
{\ar_{\frac{\be}{\thh}} "0";"4"};
{\ar^{\frac{\wt{\be}}{H}} "2";"4"};
{\ar@{}|\circlearrowright "0";"1"};
\endxy
\]

\end{lem}
\begin{proof}
Well-definedness of $\wt{\be}$ follows from the equation
\[ \eta\be(gx)=\eta\thh_x(g)\be(x)\quad(\fa (\eta,x)\in H\times X,\fa g\in G). \]
Commutativity of the diagram can be checked immediately.
\end{proof}

\begin{prop}\label{PropSIm}
For any 1-cell $\al\co\xg\to\yh$, we have a commutative diagram of 1-cells
\[
\xy
(-20,0)*+{\xg}="0";
(0,8)*+{\frac{\SIm\al}{H}}="2";
(0,-6)*+{}="3";
(20,0)*+{\frac{Y}{H}}="4";
{\ar^(0.46){\frac{\ups_{\al}}{\thh_{\al}}} "0";"2"};
{\ar^(0.54){\frac{\wt{\al}}{H}} "2";"4"};
{\ar@/_0.8pc/_{\frac{\al}{\thh_{\al}}} "0";"4"};
{\ar@{}|\circlearrowright "2";"3"};
\endxy
\qquad(\thh_{\ups_{\al}}=\thh_{\al}).
\]
We call this the {\it $\mathit{SIm}$-factorization} of $\al$.
\end{prop}
\begin{proof}
This immediately follows from Lemma \ref{LemSIm}.
\end{proof}

\begin{prop}\label{PropSImFactorSystem}
In $\Sbb$, let $\Scal$ be the class of stab-surjective 1-cells, and let $\Ecal$ be the class of equivariant 1-cells. Then the pair $(\Scal,\Ecal)$ satisfies the following properties\footnote{This is a bit weaker than the notion of a {\it factorization system} in \cite[Definition 1.6.]{Dupont_Vitale}, since $\Ecal$ is not closed under equivalences by 2-cells, nor under compositions with equivalences.}.
\begin{enumerate}
\item[{\rm (0)}] Each of $\Scal$ and $\Ecal$ is closed under compositions.
\item[{\rm (1)}] For any 1-cell $\al\co\xg\to\yh$, there exist 1-cells $s$ and $u$, with a 2-cell $u\ci s\tc\al$ as in the diagram
\[
\xy
(-10,4)*+{\xg}="0";
(10,4)*+{\yh}="2";
(0,-8)*+{\frac{Z}{H}}="4";
{\ar^{\al} "0";"2"};
{\ar_{s} "0";"4"};
{\ar_{u} "4";"2"};
{\ar@{=>} (0,-2.3);(0,1.7)};
\endxy,
\]
where $s$ is stab-surjective and $u$ is $H$-equivariant.
\item[{\rm (2)}] If in the diagram
\begin{equation}\label{Diag_albetep}
\xy
(-8,6)*+{\xg}="0";
(8,6)*+{\yh\ppr}="2";
(-8,-6)*+{\yh}="4";
(8,-6)*+{\frac{Z}{H}}="6";
{\ar^{\be} "0";"2"};
{\ar_{\al} "0";"4"};
{\ar^{\frac{\delta}{H}} "2";"6"};
{\ar_{\frac{\gamma}{H}} "4";"6"};
{\ar@{=>}^{\ep} (-2,0);(2,0)};
\endxy,
\end{equation}
$\al$ and $\be$ are stab-surjective, then the following holds.
\begin{itemize}
\item[{\rm (i)}] There exists a triplet $(\omega,\mu,\nu)$
\[
\xy
(-11,7)*+{\xg}="0";
(13,7)*+{\yh\ppr}="2";
(-11,-7)*+{\yh}="4";
(13,-7)*+{\frac{Z}{H}}="6";
{\ar^{\be} "0";"2"};
{\ar_{\al} "0";"4"};
{\ar^{\frac{\delta}{H}} "2";"6"};
{\ar_{\frac{\gamma}{H}} "4";"6"};
{\ar_{\omega} "4";"2"};
{\ar@{=>}^(0.3){\mu} (-6.5,-0.5);(-3.5,2.5)};
{\ar@{=>}^{\nu} (9,0);(6,-3)};
\endxy
\]
satisfying $\delta\ci\mu=\ep\cdot(\nu\ci\al)$.
\item[{\rm (ii)}] For any other triplet $(\omega\ppr,\mu\ppr,\nu\ppr)$ as in {\rm (i)}, there exists a unique 2-cell $\zeta\co\omega\tc\omega\ppr$ which satisfies
\[ \mu\ppr\cdot(\zeta\ci\al)=\mu\ \ \ \text{and}\ \ \ \nu\ppr\cdot(\delta\ci\zeta)=\nu. \]
\end{itemize}
\item[{\rm (3)}] $\omega$ in {\rm (2)} is an equivalence. More precisely, this $\omega$ can be taken as an $H$-isomorphism $\omega\co Y\ov{\cong}{\lra}Y\ppr$.
\end{enumerate}
\end{prop}
\begin{proof}
{\rm (0)} This follows from Proposition \ref{PropStabsurjCompos}.

{\rm (1)} This follows from Proposition \ref{PropSIm}.

{\rm (2)} Suppose diagram $(\ref{Diag_albetep})$ is given. We confirm conditions {\rm (i), (ii)}.

{\rm (i)} For any $y\in Y$, take $x\in X$ and $\eta\in H$ satisfying $y=\eta\ax$. If we define $\omega(y)$ by $\omega(y)=\eta\ep_x\iv\be(x)$, then this gives a well-defined $H$-equivariant map $\omega\co Y\to Y\ppr$.
Indeed if $x_1,x_2\in X$ and $\eta_1,\eta_2\in H$ satisfy $y=\eta_1\al(x_1)=\eta_2\al(x_2)$, then, since there is $g\in G$ satisfying
\[ x_2=gx_1\ \ \ \text{and}\ \ \ \eta_1=\eta_2\thh_{\al,x_1}(g), \]
we obtain
\begin{eqnarray*}
\eta_2\ep_{x_2}\iv\be(x_2)&=&\eta_2\thh_{\al,x_1}(g)\ep_{x_1}\iv\thh_{\be,x_1}(g)\iv\be(x_2)\\
&=&\eta_2\thh_{\al,x_1}(g)\ep_{x_1}\iv\be(x_1)\ =\ \eta_1\ep_{x_1}\iv\be(x_1).
\end{eqnarray*}
$H$-equivariance is obvious.

Moreover, for the 2-cell $\mu\co\omega\ci\al\tc\be$ defined by
\[ \mu_x=\ep_x\quad(\fa x\in X), \]
the triplet $(\omega,\mu,\id)$ 
satisfies the desired property.

{\rm (ii)} Suppose there is another triplet $(\omega\ppr,\mu\ppr,\nu\ppr)$.
By assumption, we have
\begin{equation}\label{Eq_muxepnu}
\mu\ppr_x=\ep_x\cdot\nu\ppr_{\ax}
\end{equation}
for any $x\in X$. It suffices to show the existence and the uniqueness of a 2-cell $\zeta\co\omega\tc\omega\ppr$ satisfying
\[ \nu\ppr\cdot(\delta\ci\zeta)=\id\ \ \ \text{and}\ \ \ \mu\ppr\cdot(\zeta\ci\al)=\ep. \]
By $(\ref{Eq_muxepnu})$, we can rephrase this condition as
\begin{eqnarray*}
&& \nu\ppr\cdot(\delta\ci\zeta)=\id\ \ \ \text{and}\ \ \ \mu\ppr\cdot(\zeta\ci\al)=\ep\\
&\Leftrightarrow& \nu\ppr_y\cdot\zeta_y=\id\ \ \ \text{and}\ \ \ \mu\ppr_x\cdot\zeta_{\ax}=\ep_x\quad(\fa x\in X,\fa y\in Y)\\
&\Leftrightarrow& \zeta_y=\nu^{\prime-1}_y\ \ \ \text{and}\ \ \ \mu\ppr_x\cdot\nu^{\prime-1}_{\ax}=\ep_x\quad(\fa x\in X,\fa y\in Y)\\
&\Leftrightarrow& \zeta_y=\nu^{\prime-1}_y \quad(\fa y\in Y).
\end{eqnarray*}
This last condition is satisfied only by $\zeta=\{\zeta_y=\nu^{\prime-1}_{y}\}_{y\in Y}$. This in fact becomes a 2-cell, since we have
\begin{eqnarray*}
\nu_y^{\prime-1}\omega(y)&=&\nu_y^{\prime-1}\eta\ep_x\iv\be(x)\ =\ \nu_y^{\prime-1}\eta\ep_x\iv\mu_x\ppr\omega\ppr(\al(x))\\
&=&\nu_y^{\prime-1}\eta\nu_{\al(x)}\ppr\omega\ppr(\al(x))\ =\ \thh_{\omega\ppr,\al(x)}(\eta)\omega\ppr(\al(x))\\
&=&\omega\ppr(\eta\al(x))\ =\ \omega\ppr(y)
\end{eqnarray*}
for any $y=\eta\al(x)\in Y$.

{\rm (3)} This is shown by a canonical argument, by applying {\rm (2)} twice. A closer look at the construction of $\omega$ in the proof of {\rm (2)} shows it can be taken as an $H$-isomorphism. (cf. Corollary \ref{CorAddRev}.)
\end{proof}

\begin{cor}\label{CorSImFactorSystem}
For any $\al\co\xg\to\yh$, its stabilizerwise image $\SIm\al$ is characterized up to $H$-isomorphism, by the factorization in Proposition \ref{PropSIm}.
\end{cor}
\begin{proof}
This immediately follows from Proposition \ref{PropSImFactorSystem}.
\end{proof}

\begin{cor}\label{CorAdd1}
Let $\xg\ov{\al}{\lra}\yh\ov{\be}{\lra}\zk$ be a sequence of 1-cells in $\Sbb$. 
\begin{enumerate}
\item If $\al$ is stab-surjective, then we have an isomorphism of $K$-sets $\SIm(\be\ci\al)\cong\SIm\be$. In particular, we have $\SIm(\be\ci\al)\cong\SIm\be$ if $\al$ is an equivalence $($for example, $\Ind$-equivalence$)$.
\item If $H=K$ and $\be$ is $H$-equivariant, then we have an isomorphism of $H$-sets $\SIm(\be\ci\al)\cong\SIm\al$. Thus in particular we have an isomorphism of finite $G$-sets $\SIm\al\cong X$ for any $G$-equivariant 1-cell $\al\co\xg\to \frac{Y}{G}$.
\end{enumerate}
\end{cor}
\begin{proof}
This immediately follows from Proposition \ref{PropSImFactorSystem}.
\end{proof}

\begin{prop}\label{PropAdd2}
Let $G, K$ be finite groups, and let $\al\co \xg\to \zk$ and $\be\co \frac{Y}{G}\to \zk$ be any pair of 1-cells. Then for the union map $\al\cup\be\co X\am Y\to Z$, which gives a 1-cell $\frac{X\am Y}{G}\to \zk$ $($Proposition \ref{Prop2CoprodEqui}$)$, we have an isomorphism of $K$-sets
\[ \SIm(\al\cup\be)\cong\SIm\al\am\SIm\be. \]
\end{prop}
\begin{proof}
This follows from the definition of $\SIm$.
\end{proof}

\begin{cor}\label{CorAdd2}
For any pair of 1-cells $\al\co\xg\to\zk$ and $\be\co\yh\to\zk$ in $\Sbb$, if we take the 1-cell
\[ \al\cup\be\co \xg\am\yh\to\zk \]
obtained by the universal property of the bicoproduct, then we have an isomorphism of $K$-sets
\[ \SIm(\al\cup\be)\cong\SIm\al\am\SIm\be. \]
\end{cor}
\begin{proof}
This follows from Proposition \ref{Prop2CoprodVari}, Corollary \ref{CorAdd1} and Proposition \ref{PropAdd2}.
\end{proof}

\begin{prop}\label{PropSImAndPB}
Let 
\[
\xy
(-8,6)*+{\wl}="0";
(8,6)*+{\yh}="2";
(-8,-6)*+{\xg}="4";
(8,-6)*+{\zk}="6";
{\ar^{\delta} "0";"2"};
{\ar_{\gamma} "0";"4"};
{\ar^{\be} "2";"6"};
{\ar_{\al} "4";"6"};
{\ar@{=>}^{\ep} (-2,0);(2,0)};
\endxy
\]
be a bipullback in $\Sbb$. If we factorize $\al$ and $\be$ as
\[
\xy
(-12,6)*+{\xg}="0";
(12,6)*+{\zk}="2";
(0,-6)*+{\frac{\SIm\al}{K}}="4";
(0,9)*+{}="5";
{\ar^{\al} "0";"2"};
{\ar_(0.4){\ups_{\al}} "0";"4"};
{\ar_(0.6){\frac{\wt{\al}}{K}} "4";"2"};
{\ar@{}|\circlearrowright "4";"5"};
\endxy
,
\quad
\xy
(-12,6)*+{\yh}="0";
(12,6)*+{\zk}="2";
(0,-6)*+{\frac{\SIm\be}{K}}="4";
(0,9)*+{}="5";
{\ar^{\be} "0";"2"};
{\ar_(0.4){\ups_{\be}} "0";"4"};
{\ar_(0.6){\frac{\wt{\be}}{K}} "4";"2"};
{\ar@{}|\circlearrowright "4";"5"};
\endxy,
\]
and if we take the fibered product of $\wt{\al}$ and $\wt{\be}$
\[
\xy
(-14,6)*+{S=\SIm\al\times_Z\SIm\be}="0";
(14,6)*+{\SIm\be}="2";
(-14,-6)*+{\SIm\al}="4";
(14,-6)*+{Z}="6";
{\ar^(0.7){p_{\be}} "0";"2"};
{\ar_{p_{\al}} "0";"4"};
{\ar^{\wt{\be}} "2";"6"};
{\ar_{\wt{\al}} "4";"6"};
{\ar@{}|\circlearrowright "0";"6"};
\endxy
\]
in $\Ks$, then there is an isomorphism of $K$-sets
\[ \phi\co\SIm(\al\ci\gamma)\ov{\cong}{\lra}\SIm\al\times_Z\SIm\be. \]
\end{prop}
\begin{proof}
By Proposition \ref{PropEqui2Pullback},
\[
\xy
(-12,6)*+{\frac{S}{K}}="0";
(12,6)*+{\frac{\SIm\be}{K}}="2";
(-12,-6)*+{\frac{\SIm\al}{K}}="4";
(12,-6)*+{\frac{Z}{K}}="6";
{\ar^(0.46){\frac{p_{\be}}{K}} "0";"2"};
{\ar_{\frac{p_{\al}}{K}} "0";"4"};
{\ar^{\frac{\wt{\be}}{K}} "2";"6"};
{\ar_{\frac{\wt{\al}}{K}} "4";"6"};
{\ar@{}|\circlearrowright "0";"6"};
\endxy
\]
becomes a bipullback in $\Sbb$.
By taking bipullbacks $F_1$ and $F_2$, we obtain the following diagram.
\[
\xy
(0,14)*+{\frac{F_2}{K\times H}}="2";
(20,14)*+{\yh}="4";
(-20,0)*+{\frac{F_1}{G\times K}}="10";
(0,0)*+{\frac{S}{K}}="12";
(20,0)*+{\frac{\SIm\be}{K}}="14";
(-20,-14)*+{\xg}="20";
(0,-14)*+{\frac{\SIm\al}{K}}="22";
(20,-14)*+{\frac{Z}{K}}="24";
{\ar^(0.52){p_{\be}\ppr} "2";"4"};
{\ar_{\ups_{\be}\ppr} "2";"12"};
{\ar^{\ups_{\be}} "4";"14"};
{\ar^(0.52){\ups_{\al}\ppr} "10";"12"};
{\ar^(0.48){\frac{p_{\be}}{K}} "12";"14"};
{\ar_{p_{\al}\ppr} "10";"20"};
{\ar^{\frac{p_{\al}}{K}} "12";"22"};
{\ar^{\frac{\wt{\be}}{K}} "14";"24"};
{\ar_(0.48){\ups_{\al}} "20";"22"};
{\ar_{\frac{\wt{\al}}{K}} "22";"24"};
{\ar@{}|\circlearrowright "12";"24"};
{\ar@{=>}^{\ep_2} (8,7);(12,7)};
{\ar@{=>}^{\ep_1} (-12,-7);(-8,-7)};
\endxy
\]
By the universal property of the bipullback, the bipullback of $\ups_{\al}\ppr$ and $\ups_{\be}\ppr$ should be equivalent to $\wl$. Thus we obtain a bipullback
\[
\xy
(-10,6)*+{\wl}="0";
(10,6)*+{\frac{F_2}{K\times H}}="2";
(-10,-6)*+{\frac{F_1}{G\times K}}="4";
(10,-6)*+{\frac{S}{K}}="6";
{\ar^(0.4){\ups_{\al}\pprr} "0";"2"};
{\ar_{\ups_{\be}\pprr} "0";"4"};
{\ar^{\ups_{\be}\ppr} "2";"6"};
{\ar_(0.6){\ups_{\al}\ppr} "4";"6"};
{\ar@{=>}^{\ep_3} (-2,0);(2,0)};
\endxy,
\]
together with 2-cells $p_{\al}\ppr\ci\ups_{\be}\pprr\tc\gamma$ and $p_{\be}\ppr\ci\ups_{\al}\pprr\tc\delta$.

Remark that $\ups_{\al}\ppr$ and $\ups_{\be}\ppr$ are stab-surjective by Proposition \ref{Prop2PBSSurj}, and thus so is $\ups_{\al}\ppr\ci\ups_{\be}\pprr$ by Proposition \ref{PropStabsurjCompos}.
Thus we obtain an $(\Scal,\Ecal)$-factorization
\[
\xy
(-12,4)*+{\wl}="0";
(12,4)*+{\zk}="2";
(0,-8)*+{\frac{S}{K}}="4";
{\ar^{\al\ci\gamma} "0";"2"};
{\ar_(0.4){\ups_{\al}\ppr\ci\ups_{\be}\pprr} "0";"4"};
{\ar_(0.6){\frac{\wt{\al}\ci p_{\al}}{K}} "4";"2"};
{\ar@{=>} (0,-3);(0,1)};
\endxy
\]
which implies that $\SIm(\al\ci\gamma)$ and $S$ are $K$-isomorphic, by Corollary \ref{CorSImFactorSystem}.
\end{proof}

\section{Mackey functors on $\Sbb$}

We define Mackey functors on $\Sbb$ in an analogous way as the ordinary ones for finite groups, using bicoproducts and bipullbacks.

\subsection{Definition}
We define the notions of a (semi-)Mackey functor on $\Sbb$ and on $\Csc$, which turn out to be the same.

\begin{dfn}\label{DefSemiMackC}\label{DefSemiMackS}
A {\it semi-Mackey functor} $M=(M^{\ast},M_{\ast})$ on $\Csc$ is a pair of a contravariant functor $M^{\ast}\co\Csc\to\Sett$ and a covariant functor $M_{\ast}\co\Csc\to\Sett$ which satisfies the following.
\begin{enumerate}
\item[{\rm (0)}] $M^{\ast}(\xg)=M_{\ast}(\xg)$ for any object $\xg\in\Ob(\Csc)$. We denote this simply by $M(\xg)$.
\item[{\rm (1)}] [Additivity] For any pair of objects $\xg$ and $\yh$ in $\Csc$, if we take their coproduct
\[ \xg\ov{\und{\ups_X}}{\lra}\xg\am\yh\ov{\und{\ups_Y}}{\lla}\yh \]
in $\Csc$, then the natural map
\[ (M^{\ast}(\und{\ups_X}),M^{\ast}(\und{\ups_Y}))\co M(\xg\am\yh)\to M(\xg)\times M(\yh) \]
is bijective. Also, $M(\emptyset)$ is a singleton.
\item[{\rm (2)}] [Mackey condition] For any natural weak pullback
\[
\xy
(-10,7)*+{\wl}="0";
(10,7)*+{\yh}="2";
(-10,-7)*+{\xg}="4";
(10,-7)*+{\zk}="6";
(0,0)*+{\nwp}="10";
{\ar^{\und\delta} "0";"2"};
{\ar_{\und\gamma} "0";"4"};
{\ar^{\und\be} "2";"6"};
{\ar_{\und\al} "4";"6"};
\endxy
\]
in $\Csc$, the following diagram in $\Sett$ becomes commutative.
\[
\xy
(-12,7)*+{M(\wl)}="0";
(12,7)*+{M(\yh)}="2";
(-12,-7)*+{M(\xg)}="4";
(12,-7)*+{M(\zk)}="6";
{\ar_{M^{\ast}(\und\delta)} "2";"0"};
{\ar_{M_{\ast}(\und\gamma)} "0";"4"};
{\ar^{M_{\ast}(\und\be)} "2";"6"};
{\ar^{M^{\ast}(\und\al)} "6";"4"};
{\ar@{}|\circlearrowright "0";"6"};
\endxy
\]
\end{enumerate}

We can alternatively define a semi-Mackey functor by using $\Sbb$. In the following, when we speak of a 2-functor from $\Sbb$ to $\Sett$, we regard $\Sett$ as a 2-category equipped only with identity 2-cells. Thus a 2-functor $\Sbb\to\Sett$ is nothing but a functor $\Csc\to\Sett$.

\medskip

A {\it semi-Mackey functor} $M=(M^{\ast},M_{\ast})$ on $\Sbb$ is a pair of a contravariant 2-functor $M^{\ast}\co\Sbb\to\Sett$ and a covariant 2-functor $M_{\ast}\co\Sbb\to\Sett$ which satisfies the following.
\begin{enumerate}
\item[{\rm (0)}] $M^{\ast}(\xg)=M_{\ast}(\xg)$ for any 0-cell $\xg\in\Sbb^0$. We denote this simply by $M(\xg)$.
\item[{\rm (1)}] [Additivity] For any pair of 0-cells $\xg$ and $\yh$ in $\Sbb$, if we take their bicoproduct
\[ \xg\ov{\ups_X}{\lra}\xg\am\yh\ov{\ups_Y}{\lla}\yh \]
in $\Sbb$, then the natural map
\begin{equation}\label{RCoeffAdd1}
(M^{\ast}(\ups_X),M^{\ast}(\ups_Y))\co M(\xg\am\yh)\to M(\xg)\times M(\yh)
\end{equation}
is bijective. Also, $M(\emptyset)$ is a singleton.
\item[{\rm (2)}] [Mackey condition] For any bipullback
\begin{equation}\label{RCoeffAdd2}
\xy
(-10,7)*+{\wl}="0";
(10,7)*+{\yh}="2";
(-10,-7)*+{\xg}="4";
(10,-7)*+{\zk}="6";
{\ar^{\delta} "0";"2"};
{\ar_{\gamma} "0";"4"};
{\ar^{\be} "2";"6"};
{\ar_{\al} "4";"6"};
{\ar@{=>}^{\kappa} (-2,0);(2,0)};
\endxy
\end{equation}
in $\Sbb$, the following diagram in $\Sett$ becomes commutative.
\begin{equation}\label{RCoeffAdd3}
\xy
(-12,7)*+{M(\wl)}="0";
(12,7)*+{M(\yh)}="2";
(-12,-7)*+{M(\xg)}="4";
(12,-7)*+{M(\zk)}="6";
{\ar_{M^{\ast}(\delta)} "2";"0"};
{\ar_{M_{\ast}(\gamma)} "0";"4"};
{\ar^{M_{\ast}(\be)} "2";"6"};
{\ar^{M^{\ast}(\al)} "6";"4"};
{\ar@{}|\circlearrowright "0";"6"};
\endxy
\end{equation}
\end{enumerate}
This is just a paraphrase of the definition using $\Csc$. 
With this view, for any morphism $\und{\al}$ in $\Csc$, we write $M^{\ast}(\und{\al})$ and $M_{\ast}(\und{\al})$ simply as $M^{\ast}(\al)$ and $M_{\ast}(\al)$.

\end{dfn}

\begin{prop}\label{RemMackInv}
Let $M$ be a semi-Mackey functor on $\Sbb$ $($= semi-Mackey functor on $\Csc$$)$. If $\al\co\xg\to\yh$ is an equivalence, then $M^{\ast}(\al)$ and $M_{\ast}(\al)$ are bijections, mutually inverse to each other.
\end{prop}
\begin{proof}
For a quasi-inverse $\be$ of $\al$, we have
\begin{eqnarray*}
&M^{\ast}(\be)\ci M^{\ast}(\al)=M^{\ast}(\be\ci\al)=M^{\ast}(\id)=\id,&\\
&M^{\ast}(\al)\ci M^{\ast}(\be)=M^{\ast}(\al\ci\be)=M^{\ast}(\id)=\id,&
\end{eqnarray*}
and thus $M^{\ast}(\al)$ is a bijection. Similarly for $M_{\ast}(\al)$. Moreover, by Proposition \ref{PropPullbackAdjEquivEx}, we have
\[ M^{\ast}(\al)\ci M_{\ast}(\al)=M_{\ast}(\id)\ci M^{\ast}(\id)=\id. \]
This means $M_{\ast}(\al)=M^{\ast}(\al)\iv$.
\end{proof}

\begin{dfn}\label{DefSemiMackMorph}
Let $M$ and $N$ be semi-Mackey functors on $\Sbb$. A {\it morphism} $\varphi\co M\to N$ of semi-Mackey functors is a family of maps
\[ \varphi=\{ \varphi_{\xg}\co M(\xg)\to N(\xg) \}_{\xg\in\Sbb^0} \]
compatible with contravariant and covariant parts. Namely, it gives natural transformations
\[ \varphi\co M^{\ast}\tc N^{\ast}\ \ \ \text{and}\ \ \ \varphi\co M_{\ast}\tc N_{\ast}. \]
With the usual composition of natural transformations, we obtain the category of semi-Mackey functors denoted by $\SMackS$.
\end{dfn}

\begin{prop}\label{RemSemiMackMon}
$\ \ $
\begin{enumerate}
\item Let $M$ be a semi-Mackey functor on $\Sbb$. Let $\xg$ be any 0-cell in $\Sbb$. If we denote the coproduct by
\[ \xg\ov{\ups_1}{\lra}\frac{X\am X}{G}\ov{\ups_2}{\lla}\xg \]
and the folding map by
\[ \nabla\co \frac{X\am X}{G}\to\xg, \]
then the composition of
\[ M(\xg)\times M(\xg)\ov{(M^{\ast}(\ups_1),M^{\ast}(\ups_2))\iv}{\lra}M(\frac{X\am X}{G})\ov{M_{\ast}(\nabla)}{\lra}M(\xg) \]
gives an addition on $M(\xg)$. With this addition and the unit given by
\[ M(\emptyset)\ov{M_{\ast}(\iota_X)}{\lra}M(\xg) \]
where $\iota_X\co\emptyset\to\xg$ is the unique 1-cell, $M(\xg)$ becomes a monoid. 
\item Let $\varphi\co M\to N$ be a morphism of semi-Mackey functors on $\Sbb$. For any 0-cell $\xg$ in $\Sbb$,
\[ \varphi_{\xg}\co M(\xg)\to N(\xg) \]
becomes a monoid homomorphism.
\end{enumerate}
Thus $M^{\ast}$ and $M_{\ast}$ can be regarded as functors to $\Mon$, and $\varphi$ becomes a natural transformation between such functors.
\end{prop}
\begin{proof}
For any pair of 0-cells $\xg,\yh$ in $\Sbb$, let us abbreviate the isomorphism induced from the bicoproduct by $\mu\co M(\xg)\times M(\yh)\ov{\cong}{\lra}M(\xg\am\yh)$ regardless of $\xg,\yh$. For a 0-cell $\xg$ in $\Sbb$, let $m\co M(\xg)\times M(\xg)\to M(\xg)$ denote the composition of
\[ \textstyle{M(\xg)\times M(\xg)\underset{\mu}{\ov{\cong}{\lra}}M(\xg\am\xg)\ov{M_{\ast}(\nabla)}{\lra}M_{\ast}(\xg)}. \]

\smallskip

{\rm (1)} Let $\nabla_3\co\xg\am\xg\am\xg\to\xg$ be the folding morphism, and let $m_3\co M(\xg)\times M(\xg)\times M(\xg)\to M(\xg)$ be the composition of
\[ \textstyle{M(\xg)\times M(\xg)\times M(\xg)\underset{\nu}{\ov{\cong}{\lra}} M(\xg\am\xg\am\xg)\ov{M_{\ast}(\nabla_3)}{\lra}M(\xg)}. \]
Here the first isomorphism $\nu$ is given by the injections for the bicoproduct, and the following diagram becomes commutative.
\[
\xy
(-24,8)*+{M(\xg)\times M(\xg)\times M(\xg)}="0";
(24,8)*+{M(\xg\am\xg)\times M(\xg)}="2";
(2,-4)*+{}="3";
(-24,-8)*+{M(\xg)\times M(\xg\am\xg)}="4";
(-2,4)*+{}="5";
(24,-8)*+{M(\xg\am\xg\am\xg)}="6";
{\ar^(0.54){\mu\times \id}_(0.54){\cong} "0";"2"};
{\ar_{\id\times\mu}^{\cong} "0";"4"};
{\ar^{\mu}_{\cong} "2";"6"};
{\ar_(0.54){\mu}^(0.54){\cong} "4";"6"};
{\ar^{\nu} "0";"6"};
{\ar@{}|\circlearrowright "2";"3"};
{\ar@{}|\circlearrowright "4";"5"};
\endxy
\]
With this, the commutativity of 
\[
\xy
(-16,8)*+{\xg\am\xg\am\xg}="0";
(14,8)*+{\xg\am\xg}="2";
(-2,-4)*+{}="3";
(14,-8)*+{\xg}="4";
{\ar^(0.56){\nabla\am\id} "0";"2"};
{\ar_{\nabla_3} "0";"4"};
{\ar^{\nabla} "2";"4"};
{\ar@{}|\circlearrowright "2";"3"};
\endxy
\]
yields a commutative diagram 
\[
\xy
(-28,12)*+{M(\xg)\times M(\xg)\times M(\xg)}="0";
(28,12)*+{M(\xg)\times M(\xg)}="2";
(-28,-12)*+{M(\xg\am\xg\am\xg)}="4";
(28,-12)*+{M(\xg\am\xg)}="6";
(0,0)*+{M(\xg\am\xg)\times M(\xg)}="8";
(10,-28)*+{M(\xg)}="10";
(-42,0)*+{}="11";
(0,14)*+{}="12";
(32,-10)*+{}="13";
(6,-12)*+{}="14";
{\ar^{m\times \id} "0";"2"};
{\ar_{\nu}^{\cong} "0";"4"};
{\ar_{\mu\times\id} "0";"8"};
{\ar^{\mu}_{\cong} "2";"6"};
{\ar_{M_{\ast}(\nabla\am\id)} "4";"6"};
{\ar_{M_{\ast}(\nabla_3)} "4";"10"};
{\ar^{M_{\ast}(\nabla)} "6";"10"};
{\ar_{M_{\ast}(\nabla)\times\id} "8";"2"};
{\ar_{\mu} "8";"4"};
{\ar@{}|\circlearrowright "8";"11"};
{\ar@{}|\circlearrowright "8";"12"};
{\ar@{}|\circlearrowright "8";"13"};
{\ar@{}|\circlearrowright "10";"14"};
\endxy
\]
which shows $m\ci(m\times\id)=m_3$. By symmetry, we also have $m\ci(\id\times m)=m_3$. Thus the associativity $m\ci(m\times\id)=m\ci(\id\times m)$ follows. Commutativity of this binary operation $m$ is also easily verified.

Let $0\in M(\xg)$ denote the image of the unique element of $M(\emptyset)$ by $M_{\ast}(\iota_X)$, where $\iota_X\co\emptyset\to\xg$ is the unique 1-cell. Then the commutativity of
\[
\xy
(-23,8)*+{\xg}="0";
(-14,8)*+{\cong\xg\am\emptyset}="2";
(10,8)*+{\xg\am\xg}="4";
(-8,10)*+{}="5";
(0,-8)*+{\xg}="6";
{\ar^{\id\am\iota_X} "2";"4"};
{\ar_{\id} "0";"6"};
{\ar^{\nabla} "4";"6"};
{\ar@{}|\circlearrowright "5";"6"};
\endxy
\]
yields a commutative diagram
\[
\xy
(-36.5,8)*+{M(\xg)}="0";
(-20,8)*+{\cong\, M(\xg\am\emptyset)}="2";
(-30,17)*+{}="3";
(20,8)*+{M(\xg\am\xg)}="4";
(-6,8)*+{}="5";
(0,-8)*+{M(\xg)}="6";
(-20,22)*+{M(\xg)\times M(\emptyset)}="8";
(20,22)*+{M(\xg)\times M(\xg)}="10";
{\ar_{M_{\ast}(\id\am\iota_X)} "2";"4"};
{\ar_{\id} "0";"6"};
{\ar^{M_{\ast}(\nabla)} "4";"6"};
{\ar_{\mu}^{\cong} "8";"2"};
{\ar^{\id\times M_{\ast}(\iota_X)} "8";"10"};
{\ar^{\mu}_{\cong} "10";"4"};
{\ar_(0.58){\text{1st projection}} "8";"0"};
{\ar@{}|\circlearrowright "5";"6"};
{\ar@{}|\circlearrowright "4";"8"};
{\ar@{}|\circlearrowright "2";"3"};
\endxy
\]
which shows $m((x,0))=x\ \ (\fa x\in M(\xg))$.

{\rm (2)} This immediately follows from the naturality of $\varphi$.
\end{proof}

\begin{dfn}\label{DefMack}
A semi-Mackey functor $M$ on $\Sbb$ is a {\it Mackey functor} if the monoid $M(\xg)$ is an additive group for any $\xg\in\Sbb^0$.
The full subcategory of Mackey functors in $\SMackS$ is denoted by $\MackS$.
\end{dfn}

\begin{rem}
$M\in\Ob(\SMackS)$ belongs to $\MackS$ if and only if both $M^{\ast}$ and $M_{\ast}$ are functors to $\Ab$.
\end{rem}

This allows us the following definition. Compare with Definition \ref{DefSemiMackS}. In this definition, $\RMod$ denotes the category of $R$-modules. A 2-functor from $\Sbb$ to $\RMod$ is nothing but a functor from $\Csc$ to $\RMod$.

\begin{dfn}\label{DefRLinearMack}
Let $R$ be a commutative ring. An $R$-{\it linear Mackey functor} $M=(M^{\ast},M_{\ast})$ on $\Sbb$ is a pair of a contravariant 2-functor $M^{\ast}\co\Sbb\to\RMod$ and a covariant 2-functor $M_{\ast}\co\Sbb\to\RMod$, which satisfies the following.
\begin{enumerate}
\item[{\rm (0)}] $M^{\ast}(\xg)=M_{\ast}(\xg)\, (=M(\xg))$ for any 0-cell $\xg\in\Sbb^0$. 
\item[{\rm (1)}] [Additivity] For any pair of 0-cells $\xg$ and $\yh$ in $\Sbb$, the natural map $(\ref{RCoeffAdd1})$ is an isomorphism. $M(\emptyset)=0$ is the zero module.
\item[{\rm (2)}] [Mackey condition] For any bipullback $(\ref{RCoeffAdd2})$ in $\Sbb$, the diagram $(\ref{RCoeffAdd3})$ is a commutative diagram in $\RMod$.
\end{enumerate}
A {\it morphism} $\varphi\co M\to N$ of $R$-linear Mackey functors is a family $\varphi=\{ \varphi_{\xg}\}_{\xg\in\Sbb^0}$ of $R$-homomorphisms compatible with contravariant and covariant parts.
We denote the category of $R$-linear Mackey functors by $\MackSR$, or by $\MackCR$.
\end{dfn}

\begin{rem}
Remark that the additive completion of monoids gives a functor $K_0\co \Mon\to\Ab$. From any semi-Mackey functor $M=(M^{\ast}, M_{\ast})$, by composing $K_0$ we obtain a Mackey functor $K_0M=(K_0\ci M^{\ast},K_0\ci M_{\ast})$ on $\Sbb$.
This gives a functor $K_0\co\SMackS\to\MackS$, which is left adjoint to the inclusion functor $\MackS\hookrightarrow\SMackS$.

Furthermore, since tensoring with $R$ gives an additive functor $-\otimes_{\mathbb{Z}} R\co \Ab\to\RMod$, from any semi-Mackey functor $M=(M^{\ast}, M_{\ast})$, by composing $-\otimes_{\mathbb{Z}} R$ and $K_0$, we obtain an $R$-linear Mackey functor $M^R=((-\otimes_{\mathbb{Z}} R)\ci K_0\ci M^{\ast},(-\otimes_{\mathbb{Z}} R)\ci K_0\ci M_{\ast})$ on $\Sbb$.
This gives a functor $(-)^R\co\SMackS\to\MackSR$, which is left adjoint to the forgetful functor $\MackSR\rightarrow\SMackS$.
\end{rem}

\begin{lem}\label{RemMackSMackG}
For a fixed finite group $G$, the functor $\frac{\bullet}{G}\co \Gs\to\GrSet$ in Proposition \ref{PropFunctEqui} induces\footnote{This question is raised by Professor Fumihito Oda.} a functor
\[ \MackS\to \Mack(G)\ ;\ M=(M^{\ast},M_{\ast})\mapsto\ \Mbf=(M^{\ast}(\frac{\bullet}{G}),M_{\ast}(\frac{\bullet}{G})), \]
where $\Mack(G)$ denotes the category of $($ordinary$)$ Mackey functors on $G$.
\end{lem}
\begin{proof}
This follows from Propositions \ref{Prop2CoprodEqui} and \ref{PropEqui2Pullback}.
\end{proof}

\begin{prop}
Let $G$ be a fixed finite group. Mackey functors $\Mbf$ on $G$ obtained in Lemma \ref{RemMackSMackG} form a special class in $\Mack(G)$, since $\Mbf$ satisfies
\[ \Mbf^{\ast}(\al)=\Mbf^{\ast}(\al\ppr),\ \ \ \Mbf_{\ast}(\al)=\Mbf_{\ast}(\al\ppr) \]
for any $\al,\al\ppr\in\Gs(X,Y)$ satisfying $\und{\al}=\und{\al\ppr}$ in $\Csc$.
This can be explained more precisely as follows.
\begin{enumerate}
\item Let $\GrSet|_G$ denote the subcategory of $\GrSet$, whose objects are $\xg$ for some $X\in\Ob(\Gs)$, and morphisms are $G$-equivariant maps. Then obviously we have $\GrSet|_G=\Gs$.
\item Let $\Csc|_G$ denote the subcategory of $\Csc$ obtained as the quotient image of $\GrSet|_G$ under the functor $\GrSet\to\Csc$ in Remark \ref{RemSC}. Then we have $\Csc|_G\simeq G\text{-}\und{\mathit{set}}$, where the right hand side denotes the category of finite {\it fused $G$-sets}. The category of finite fused $G$-sets, defined in \cite[section 3]{Bouc_fused}, is the quotient of $\Gs$ defined by the following.
\begin{itemize}
\item[-] $\Ob(G\text{-}\und{\mathit{set}})=\Ob(\Gs)$.
\item[-] For any $X,Y\in\Ob(G\text{-}\und{\mathit{set}})$, the morphism set is
\[ G\text{-}\und{\mathit{set}}(X,Y)=\Gs(X,Y)/\sim, \]
where two morphisms $f,f\ppr\in\Gs(X,Y)$ are defined to be equivalent $f\sim f\ppr$ when there exists a $G$-map $w\co X\to G^c$ satisfying
\[ f\ppr(x)=w(x)f(x)\quad(\fa x\in X, g\in G). \]
Here, $G^c$ is the $G$-set $G$ on which $G$ acts by the conjugation
\[ G\times G^c\to G^c\ ;\ (g,x)\mapsto gxg\iv. \]
\end{itemize}

\smallskip

 Thus the functor $\frac{\bullet}{G}\co\Gs\to\Csc$ factors through $G\text{-}\und{\mathit{set}}$.
\[
\xy
(-14,6)*+{\Gs}="0";
(14,6)*+{\Csc}="2";
(0,8)*+{}="7";
(0,-6)*+{G\text{-}\und{\mathit{set}}\simeq\Csc|_G}="8";
{\ar^{\frac{\bullet}{G}} "0";"2"};
{\ar_{} "0";"8"};
{\ar@{^(->} "8";"2"};
{\ar@{}|\circlearrowright "7";"8"};
\endxy
\]
\item Since $\Csc|_G$ is closed under coproducts and natural weak pullbacks by Corollary \ref{CorCoprodEqui} and Proposition \ref{PropEqui2Pullback}, any Mackey functor $M$ on $\Csc$ can be restricted to give a Mackey functor $\Mbf$ on $G\text{-}\und{\mathit{set}}$, which is called {\it fused Mackey functor} on $G$ $($\cite[Definition 4.2]{Bouc_fused}$)$.
\end{enumerate}
\end{prop}
\begin{proof}
Since {\rm (1)} is obvious and {\rm (3)} follows from {\rm (2)}, we only show {\rm (2)}.

Let $\al,\al\ppr\co\xg\to \frac{Y}{G}$ be $G$-equivariant 1-cells. Then a 2-cell $\ep\co\al\tc\al\ppr$ is, by definition, a map $\ep\co X\to G$ satisfying
\begin{equation}\label{Eq_Equi2_1}
\al\ppr(x)=\ep_x\al(x)\quad(\fa x\in X)
\end{equation}
and
\begin{equation}\label{Eq_Equi2_2}
\ep_{gx}g\ep_x\iv=g\quad(\fa g\in G,\fa x\in X).
\end{equation}
Remark that $(\ref{Eq_Equi2_2})$ is equivalent to that $\ep$ is an element of $\Gs(X,G^c)$. This condition does not depend on the 1-cells $\al, \al\ppr$.

Also remark that the vertical composition of 2-cells gives a group structure on $\Gs(X,G^c)$. Condition $(\ref{Eq_Equi2_1})$ means that this group $\Gs(X,G^c)$ acts on the set of morphisms $\GrSet|_G(\xg,\frac{Y}{G})=\Gs(X,Y)$.

Since $\Csc|_G(\xg,\frac{Y}{G})$ is the quotient of $\GrSet|_G(\xg,\frac{Y}{G})$ by 2-cells, it agrees with the quotient of $\Gs(X,Y)$ by this group action. Namely, we have
\[ \Csc|_G(\xg,\frac{Y}{G})\cong \Gs(X,G^c)\backslash \Gs(X,Y). \]
This gives an equivalence $\Csc|_G\simeq G\text{-}\und{\mathit{set}}$.
\end{proof}

\subsection{Functors on span category}
Before the comparison of Mackey functors and biset functors, intermediately we show that a Mackey functor can be realized as a single functor on the span category of $\Sbb$. This is an analog of Lindner's result (\cite[Theorem 4]{Lindner}). We also remark that spans in 2-categories are studied in detail in \cite{Hoffnung}. In the following, $\Cbb$ denotes a 2-category with invertible 2-cells.

\begin{dfn}\label{Def2Span}(\cite[Definitions 3.1.1, 3.3.1]{Hoffnung})
Let $X$ and $Y$ be 0-cells in $\Cbb$. A {\it span} $S$ to $X$ from $Y$ in $\Cbb$ is a pair of 1-cells from some 0-cell $W_S$
\[ S=(\spSa) \]
in $\Cbb$. We sometimes simply write this as $\gsha$. The span $(X\ov{\id_{X}}{\lla}X\ov{\id_{X}}{\lra}X)$ is denoted by $\Id={}_{X}\! \Id_{X}$, and called the {\it identity span}.
\end{dfn}

\begin{dfn}\label{Def2SpanXY}(\cite[section 3]{Hoffnung}) 
Let $X$ and $Y$ be any pair of 0-cells in $\Cbb$. Then a 2-category $\Spana$ is defined as follows.
\begin{enumerate}
\item[{\rm (0)}] A 0-cell in $\Spana$ is a span $S$ to $X$ from $Y$.
\item[{\rm (1)}] A 1-cell in $\Spana$ from $S=(\spSa)$ to $T=(\spTa)$ is a triplet $(\varphi,\mu_X,\mu_Y)$ of a 1-cell $\varphi$ and 2-cells $\mu_X,\mu_Y$ in $\Cbb$ as in the following diagram.
\[
\xy
(-20,0)*+{X}="0";
(0,10)*+{W_S}="2";
(0,-10)*+{W_T}="4";
(20,0)*+{Y}="6";
{\ar_{\al_S} "2";"0"};
{\ar^{\be_S} "2";"6"};
{\ar^{\varphi} "2";"4"};
{\ar^{\al_T} "4";"0"};
{\ar_{\be_T} "4";"6"};
{\ar@{=>}^{\mu_X} (-4,-2);(-8,2)};
{\ar@{=>}_{\mu_Y} (4,-2);(8,2)};
\endxy
\]
\item[{\rm (2)}] If $(\varphi,\mu_X,\mu_Y)\co S\to T$ and $(\varphi\ppr,\mu_X\ppr,\mu_Y\ppr)\co S\to T$ are 1-cells in $\Spana$, then a 2-cell $\ep\co (\varphi,\mu_X,\mu_Y)\tc(\varphi\ppr,\mu_X\ppr,\mu_Y\ppr)$ in $\Spana$ is a 2-cell $\ep\co \varphi\tc \varphi\ppr$ in $\Cbb$, which makes the following diagrams commutative.
\[
\xy
(-10,6)*+{\al_T\ci \varphi}="0";
(10,6)*+{\al_T\ci \varphi\ppr}="2";
(0,-8)*+{\al_S}="4";
(0,10)*+{}="5";
{\ar@{=>}^{\al_T\ci\ep} "0";"2"};
{\ar@{=>}_(0.4){\mu_X} "0";"4"};
{\ar@{=>}^(0.4){\mu_X\ppr} "2";"4"};
{\ar@{}|\circlearrowright "4";"5"};
\endxy,
\qquad
\xy
(-10,6)*+{\be_T\ci \varphi}="0";
(10,6)*+{\be_T\ci \varphi\ppr}="2";
(0,-8)*+{\be_S}="4";
(0,10)*+{}="5";
{\ar@{=>}^{\be_T\ci\ep} "0";"2"};
{\ar@{=>}_(0.4){\mu_Y} "0";"4"};
{\ar@{=>}^(0.4){\mu_Y\ppr} "2";"4"};
{\ar@{}|\circlearrowright "4";"5"};
\endxy
\]
\end{enumerate}

Composition of 1-cells
\[ (\varphi,\mu_X,\mu_Y)\co (\spSa)\to (\spTa) \]
and 
\[ (\psi,\nu_X,\nu_Y)\co (\spTa)\to (\spPa) \]
is defined to be
\[ (\psi\ci \varphi,\mu_X\cdot(\nu_X\ci \varphi),\mu_Y\cdot(\nu_Y\ci \varphi)). \]

Vertical composition of 2-cells
\[
\xy
(-16,0)*+{S}="0";
(16,0)*+{T}="2";
{\ar@/^2.0pc/^{(\varphi,\mu_X,\mu_Y)} "0";"2"};
{\ar|*+{_{(\varphi\ppr,\mu_X\ppr,\mu_Y\ppr)}} "0";"2"};
{\ar@/_2.0pc/_{(\varphi\pprr,\mu_X\pprr,\mu_Y\pprr)} "0";"2"};
{\ar@{=>}^{\ep} (0,6);(0,3)};
{\ar@{=>}^{\ep\ppr} (0,-3);(0,-6)};
\endxy
\]
is defined to be $\ep\ppr\cdot\ep$, using the vertical composition in $\Cbb$.

Horizontal composition of 2-cells
\[
\xy
(-32,0)*+{S}="0";
(0,0)*+{T}="2";
(32,0)*+{P}="4";
{\ar@/^1.3pc/^{(\varphi,\mu_X,\mu_Y)} "0";"2"};
{\ar@/_1.3pc/_{(\varphi\ppr,\mu_X\ppr,\mu_Y\ppr)} "0";"2"};
{\ar@/^1.3pc/^{(\psi,\nu_X,\nu_Y)} "2";"4"};
{\ar@/_1.3pc/_{(\psi\ppr,\nu_X\ppr,\nu_Y\ppr)} "2";"4"};
{\ar@{=>}^{\ep} (-16,2);(-16,-2)};
{\ar@{=>}^{\delta} (16,2);(16,-2)};
\endxy
\]
is defined to be $\delta\ci\ep$, using the horizontal composition in $\Cbb$.

\end{dfn}

The following is shown in \cite{Hoffnung}.
\begin{fact}(\cite[Proposition 3.4.1]{Hoffnung})\label{FactHoff}
$\Spana$ is in fact a 2-category, for each pair $X,Y$.
\end{fact}

\begin{dfn}\label{Def_0519_1}
Let $X$ and $Y$ be 0-cells in $\Cbb$.
Two spans
\[ S=(X\ov{\al_S}{\lla}W_S\ov{\be_S}{\lra}Y),\ \ \text{and}\ \ T=(X\ov{\al_T}{\lla}W_T\ov{\be_T}{\lra}Y) \]
are {\it equivalent} if there exists an equivalence
\[
\xy
(-18,0)*+{X}="0";
(0,9)*+{W_S}="2";
(0,-9)*+{W_T}="4";
(18,0)*+{Y}="6";
{\ar_{\al_S} "2";"0"};
{\ar^{\be_S} "2";"6"};
{\ar^{\varphi} "2";"4"};
{\ar^{\al_T} "4";"0"};
{\ar_{\be_T} "4";"6"};
{\ar@{=>}^{\nu_X} (-4,-2);(-8,2)};
{\ar@{=>}_{\nu_Y} (4,-2);(8,2)};
\endxy
\]
in $\Spana$.
Remark that this implies in particular $\varphi$ is an equivalence in $\Cbb$.
We denote the equivalence class of $S$ by $[S]$. 
\end{dfn}

\begin{dfn}\label{DefSpan2}
For any 1-cell $\al\co X\to Y$ in $\Cbb$, we define the equivalence classes $\Rbf_{\al}$ and $\Tbf_{\al}$ by
\begin{eqnarray*}
\Rbf_{\al}&=&[X\ov{\id}{\lla}X\ov{\al}{\lra}Y]\quad(\text{in}\ \Spana),\\
\Tbf_{\al}&=&[Y\ov{\al}{\lla}X\ov{\id}{\lra}X]\quad(\text{in}\ \SpanHGa).
\end{eqnarray*}
\end{dfn}

\begin{prop}\label{Prop2CoSpan}
Suppose $\Cbb$ admits bicoproducts.
Then $\Spana$ also admits bicoproducts induced from those in $\Cbb$.
\end{prop}
\begin{proof}
For any pair of 0-cells $S=(\spSa)$ and $T=(\spTa)$ in $\Spana$, if we take the bicoproduct of $W_S$ and $W_T$
\[ W_S\ov{\ups_{W_S}}{\lra}W_S\am W_T\ov{\ups_{W_T}}{\lla}W_T \]
in $\Cbb$, then by its universal property, we obtain a diagram
\[
\xy
(0,-16)*+{W_T}="0";
(-35,0)*+{X}="2";
(0,0)*+{W_S\am W_T}="4";
(35,0)*+{Y}="6";
(0,16)*+{W_S}="8";
{\ar^{\al_T} "0";"2"};
{\ar_{\ups_{W_T}} "0";"4"};
{\ar_{\be_T} "0";"6"};
{\ar^(0.52){\al_S\cup\al_T} "4";"2"};
{\ar_(0.52){\be_S\cup\be_T} "4";"6"};
{\ar_{\al_S} "8";"2"};
{\ar^{\ups_{W_S}} "8";"4"};
{\ar^{\be_S} "8";"6"};
{\ar@{=>}^{\lam_X} (-8,-4);(-12,-7)};
{\ar@{=>}^{\lam_Y} (8,-4);(12,-7)};
{\ar@{=>}_{\kappa_X} (-8,4);(-12,7)};
{\ar@{=>}_{\kappa_Y} (8,4);(12,7)};
\endxy.
\]
This gives a bicoproduct
\[ S\ov{(\ups_{W_S},\kappa_X,\kappa_Y)}{\lra}(X\ov{\al_S\cup\al_T}{\lla}W_S\am W_T\ov{\be_S\cup\be_T}{\lra}Y)\ov{(\ups_{W_T},\lam_X,\lam_Y)}{\lla}T \]
in $\Spana$.
\end{proof}

\begin{dfn}\label{DefSpanSum}
Assume $\Cbb$ admits bicoproducts.
Let $X$ and $Y$ be 0-cells in $\Cbb$. For spans in $\Cbb$
\[ S=(X\ov{\al_S}{\lla}W_S\ov{\be_S}{\lra}Y)\ \ \text{and}\ \ T=(X\ov{\al_T}{\lla}W_T\ov{\be_T}{\lra}Y), \]
 their {\it sum} is defined to be the bicoproduct
\[ S+T=(X\ov{\al_S\cup\al_T}{\lla}W_S\am W_T\ov{\be_S\cup\be_T}{\lra}Y). \]
\end{dfn}

\begin{rem}\label{PropSpanSumEquiv}
Sum of the spans does not depend on the representatives of the equivalence classes in $\Spana$. Thus $[S]+[T]=[S+T]$ is well-defined.
\end{rem}

\begin{dfn}\label{Def2SpanCompos}
Assume $\Cbb$ admits bipullbacks.
Let
\begin{eqnarray*}
S&=&(\spSca)\ \ \in(\SpanHGa)^0\\
T&=&(\spTzca)\ \ \in(\mathrm{Span}^{Z}_{Y})^0
\end{eqnarray*}
be two consecutive spans in $\Cbb$. We define their {\it composition}
\[ T\ci S=(Z\ov{\al_{T\ci S}}{\lla}W_{T\ci S}\ov{\be_{T\ci S}}{\lra}X) \]
as follows.
\begin{itemize}
\item[-] Take a bipullback
\[
\xy
(-8,6)*+{F}="0";
(8,6)*+{W_S}="2";
(-8,-6)*+{W_T}="4";
(8,-6)*+{Y}="6";
{\ar^(0.52){\wp_{W_S}} "0";"2"};
{\ar_{\wp_{W_T}} "0";"4"};
{\ar^{\al_S} "2";"6"};
{\ar_{\be_T} "4";"6"};
{\ar@{=>}^{\chi} (-2,0);(2,0)};
\endxy
\]
and put
\[ W_{T\ci S}=F,\ \ \al_{T\ci S}=\al_T\ci\wp_{W_T},\ \ \be_{T\ci S}=\be_S\ci\wp_{W_S}, \]
as in the following diagram.
\[
\xy
(0,10)*+{W_{T\ci S}}="0";
(-12,0)*+{W_T}="2";
(12,0)*+{W_S}="4";
(-25,-10)*+{Z}="6";
(0,-10)*+{Y}="8";
(25,-10)*+{X}="10";
{\ar_(0.6){\wp_{W_T}} "0";"2"};
{\ar^(0.6){\wp_{W_S}} "0";"4"};
{\ar_(0.6){\al_T} "2";"6"};
{\ar_(0.4){\be_T} "2";"8"};
{\ar^(0.4){\al_S} "4";"8"};
{\ar^(0.6){\be_S} "4";"10"};
{\ar@{=>}^{\chi} (-2,0);(2,0)};
\endxy
\]
\end{itemize}
\end{dfn}
The equivalence class $[T\ci S]$ does not depend on representatives of equivalence classes of spans $[S]$ and $[T]$. Consequently, we obtain the following category.
\begin{dfn}\label{DefSpanCat}
The {\it span category} $\Sp(\Cbb)$ of $\Cbb$ is defined as follows\footnote{In \cite{Hoffnung}, a tricategory is constructed by using pullbacks as compositions of spans (\cite[Theorem 3.0.3]{Hoffnung}).}.
\begin{enumerate}
\item $\Ob(\Sp(\Cbb))=\Cbb^0=\Ob(\Cbb/\text{{\it 2-cells}})$.
\item For any pair of objects $X$ and $Y$, a morphism from $X$ to $Y$ is a equivalence class $[S]$ of a span $_{Y}\! S_X\in(\SpanHGa)^0$. When we want to emphasize it is a morphism in $\Sp(\Cbb)$, we will denote it by $[S]\co X\rta Y$.
\end{enumerate}
The composition of morphisms is defined by the composition of spans, and the identity span gives the identity morphism.
\end{dfn}

\begin{rem}
For any pair of objects $X$ and $Y$ in $\Sp(\Cbb)$, the set of morphisms $\Sp(\Cbb)(X,Y)$ has a structure of monoid with the addition obtained in Definition \ref{DefSpanSum}. Unit for this addition is given by $0=[Y\lla\emptyset\lra X]$.
\end{rem}


Now we return to the case $\Cbb=\Sbb$. In the rest, we simply denote $\Sp(\Sbb)$ by $\Sp$.
The following result is shown in the same way as in \cite[Lemma 3]{Lindner} and \cite[section 3]{PS}.

\begin{prop}
Let $\xg,\yh$ be any pair of objects in $\Sp$.
If we take their bicoproduct
\[ \xg\ov{\ups_X}{\lra}\xg\am\yh\ov{\ups_Y}{\lla}\yh \]
in $\Sbb$, then
\[ \xg\ov{\Rbf_{\ups_X}}{\lta}\xg\am\yh\ov{\Rbf_{\ups_Y}}{\rta}\yh \]
is a product of $\xg$ and $\yh$ in $\Sp$.
\end{prop}

\begin{dfn}
The category $\Tcal$ is defined as follows.
\begin{enumerate}
\item $\Ob(\Tcal)=\Ob(\Sp)$.
\item For any objects $\xg,\yh$ in $\Tcal$,
\[ \Tcal(\xg,\yh)=K_0(\Sp(\xg,\yh)). \]
\end{enumerate}
Thus a morphism $\xg\rta\yh$ in $\Tcal$ is written as a difference
\[ [S]-[T]\co\xg\rta\yh \]
of $[S],[T]\in\Sp(\xg,\yh)$. Composition of morphisms is defined by extending the composition in $\Sp$ by linearity.
Also in $\Tcal$,
\[ \xg\ov{\Rbf_{\ups_X}}{\lta}\xg\am\yh\ov{\Rbf_{\ups_Y}}{\rta}\yh \]
gives a product of $\xg$ and $\yh$.
\end{dfn}

Since equivalences in $\Sbb$ preserve the number of orbits by Proposition \ref{PropStabsurjDecomp}, it can be easily shown that the natural map $\Sp(\xg,\yh)\to\Tcal(\xg,\yh)$ is a monomorphism. These form a faithful functor $c\colon \Sp\to\Tcal$.

\begin{dfn}\label{DefAddFtr}
$\ \ $
\begin{enumerate}
\item Denote the category of functors $E\co\Sp\to\Sett$ preserving finite products by $\Add(\Sp,\Sett)$. Morphisms are natural transformations.
\item Similarly, denote the category of functors $F\co\Tcal\to\Sett$ preserving finite products by $\Add(\Tcal,\Sett)$. Morphisms are natural transformations.
\end{enumerate}
\end{dfn}
\begin{rem}
$\ \ $
\begin{enumerate}
\item For any $E\in\Ob(\Add(\Sp,\Sett))$ and for any $\xg\in\Ob(\Sp)$, the set $E(\xg)$ becomes a monoid with respect to the addition
\[ E(\xg)\times E(\xg)\cong E(\frac{X\am X}{G})\ov{E(\Tbf_{\nabla})}{\lra}E(\xg), \]
where $\nabla\co\frac{X\am X}{G}\to \xg$ is the folding map. Similarly, $F(\xg)$ becomes an abelian group for any $F\in\Ob(\Add(\Tcal,\Sett))$ and any $\xg\in\Ob(\Tcal)$.
\item Composition of the natural functor $c\co\Sp\to\Tcal$ yields a functor
\[ \Add(\Tcal,\Sett)\to \Add(\Sp,\Sett)\ ;\  F\mapsto F\ci c. \]
This is a fully faithful functor, and $E\in\Ob(\Add(\Sp,\Sett))$ comes from some $F\in\Ob(\Add(\Tcal,\Sett))$ if and only if $E(\xg)$ is an abelian group for any $\xg\in\Sbb^0$.
\end{enumerate}
\end{rem}

\begin{ex}\label{ExRep}
For any 0-cell $\xg$ in $\Sbb$, the representable functor
\[ \Tcal(\xg,-)\co\Tcal\to\Sett \]
preserves finite products, and thus becomes an object in $\Add(\Tcal,\Sett)$. Similarly for representable functors on $\Sp$.
\end{ex}

If $F$ is an object in $\Add(\Tcal,\Sett)$, it can be regarded as a functor to $\Ab$. 
Similarly as in Definition \ref{DefRLinearMack}, we can also define $R$-linear case as follows.
\begin{dfn}\label{DefRLinearAdd}
We denote the category of functors $F\co\Tcal\to\RMod$ preserving finite products by $\Add(\Tcal,\RMod)$. Morphisms are natural transformations. Since $\Tcal$ is an additive category, this is nothing but the category of additive functors, in the usual sense.
\end{dfn}

\begin{prop}\label{PropAddFtrMack}
To give a $($resp. semi-$)$Mackey functor $M$ on $\Sbb$ is equivalent to give a functor $F\co\Tcal\to\Sett$ $($resp. $\Sp\to\Sett$$)$ preserving finite products.
More precisely, there are equivalences of categories
\begin{eqnarray*}
\SMackS&\ov{\simeq}{\lra}&\Add(\Sp,\Sett),\\
\MackS&\ov{\simeq}{\lra}&\Add(\Tcal,\Sett),
\end{eqnarray*}
which make the following diagram commutative.
\[
\xy
(-14,7)*+{\MackS}="0";
(14,7)*+{\Add(\Tcal,\Sett)}="2";
(-14,-7)*+{\SMackS}="4";
(14,-7)*+{\Add(\Sp,\Sett)}="6";
{\ar^(0.46){\simeq} "0";"2"};
{\ar@{^(->} "0";"4"};
{\ar^{-\ci c} "2";"6"};
{\ar_(0.46){\simeq} "4";"6"};
{\ar@{}|\circlearrowright "0";"6"};
\endxy
\]
\end{prop}
\begin{proof}
This is shown in the same way as in \cite[Theorem 4]{Lindner} and \cite[Proposition 4.1]{PS}. The only different point is that we are using natural weak pullbacks instead of fibered products.
We only state the correspondence of $M$ and $F$.
To each $F$, we associate $M$ by
\begin{itemize}
\item[-] $M(\xg)=F(\xg)$ for any 0-cell $\xg$ in $\Sbb$.
\item[-] $M_{\ast}(\al)=F(\Tbf_{\al}),\ M^{\ast}(\al)=F(\Rbf_{\al})$ for any 1-cell $\al\co\xg\to\yh$ in $\Sbb$.
\end{itemize}
Conversely, to each $M$, we associate $F$ satisfying
\begin{itemize}
\item[-] $F(\xg)=M(\xg)$ for any 0-cell $\xg$ in $\Sbb$.
\item[-] $F([S])=M_{\ast}(\al_S)\circ M^{\ast}(\be_S)$ for any span $S=(\yh\ov{\al_S}{\lla}\lsws\ov{\be_S}{\lra}\xg)$.
\end{itemize}
This $F([S])$ only depends on the equivalence class $[S]$, since for an equivalence
\[
\xy
(-20,0)*+{\yh}="0";
(0,10)*+{\lsws}="2";
(0,-10)*+{\ltwt}="4";
(20,0)*+{\xg}="6";
{\ar_{\al_S} "2";"0"};
{\ar^{\be_S} "2";"6"};
{\ar^{\varphi} "2";"4"};
{\ar^{\al_T} "4";"0"};
{\ar_{\be_T} "4";"6"};
{\ar@{=>}^{\nu_X} (-4,-2);(-8,2)};
{\ar@{=>}_{\nu_Y} (4,-2);(8,2)};
\endxy
\]
of spans, we have $M_{\ast}(\al_S)\ci M^{\ast}(\be_S)=M_{\ast}(\al_T)\ci M_{\ast}(\varphi)\ci M^{\ast}(\varphi)\ci M^{\ast}(\be_T)=M_{\ast}(\al_T)\ci M^{\ast}(\be_T)$ 
by Proposition \ref{RemMackInv}.
\end{proof}

The same correspondence gives the following equivalence.
\begin{prop}\label{PropAddFtrMackRCoeff}
There is an equivalences of categories
\[ \MackSR\ov{\simeq}{\lra}\Add(\Tcal,\RMod). \]
\end{prop}

\subsection{Deflative Mackey functors}

We define a special class of Mackey functors, called {\it deflative} Mackey functors, which will be shown to correspond to biset functors in the next section. The deflativity condition corresponds to the equation $(\ref{Eq1.3})$ for biset functors.
\begin{dfn}\label{DefDeflMack}
A semi-Mackey functor $M$ on $\Sbb$ is called {\it deflative} if for any stab-surjective 1-cell $\al\co\xg\to\yh$ in $\Sbb$, the equality
\[ M_{\ast}(\al)\ci M^{\ast}(\al)=\id_{M(\yh)} \]
is satisfied. An ($R$-linear) Mackey functor is called deflative if it is deflative as a semi-Mackey functor. 

The full subcategory of deflative semi-Mackey functors is denoted by $\SMack_{\dfl}(\Sbb)\subseteq\SMackS$. 
Similarly, the full subcategory of deflative Mackey functors is denoted by $\Mack_{\dfl}(\Sbb)\subseteq\MackS$. In the $R$-linear case, similarly we denote as $\Mack_{\dfl}^R(\Sbb)\subseteq\MackSR$.
\end{dfn}

\begin{prop}\label{PropDeflMack}
For an $R$-linear $($resp. semi-$)$Mackey functor $M$ on $\Sbb$, the following are equivalent.
\begin{enumerate}
\item $M$ is deflative.
\item For any finite group $G$ and its normal subgroup $N\nm G$, if we denote the quotient homomorphism by $p\co G\to G/N$, then 
the equality
\[ M_{\ast}(\frac{\pt}{p})\ci M^{\ast}(\frac{\pt}{p})=\id \]
is satisfied for the 1-cell $\frac{\pt}{p}\co \frac{\pt}{G}\to \frac{\pt}{(G/N)}$.
\end{enumerate}
\end{prop}
\begin{proof}
This follows from Propositions \ref{PropStabsurjDecomp}, \ref{PropSSDef} and \ref{RemMackInv}.
\end{proof}

\begin{cor}\label{CorFDef}
Let $M$ be an $R$-linear $($resp. semi-$)$Mackey functor on $\Sbb$, and let $F$ be the corresponding object in $\Add(\Tcal,\RMod)$ $($resp. $\Add(\Sp,\Sett)$$)$. Then the following are equivalent.
\begin{enumerate}
\item $M$ is deflative.
\item For any stab-surjective 1-cell $\al\co\xg\to\yh$ in $\Sbb$, we have
\[ F([\yh\ov{\al}{\lla}\xg\ov{\al}{\lra}\yh])=\id_{F(\yh)}. \]
\item For any finite group $G$ and its normal subgroup $N\nm G$, 
\[ F([\frac{\pt}{(G/N)}\ov{\frac{\pt}{p}}{\lla}\frac{\pt}{G}\ov{\frac{\pt}{p}}{\lra}\frac{\pt}{(G/N)}])=\id \]
holds for the quotient homomorphism $p\co G\to G/N$.
\end{enumerate}
\end{cor}
\begin{proof}
This follows from the fact that for any 1-cell $\al\co\xg\to\yh$, we have
\begin{eqnarray*}
M_{\ast}(\al)\ci M^{\ast}(\al)&=&F(\Tbf_{\al})\ci F(\Rbf_{\al})\ =\ F(\Tbf_{\al}\ci \Rbf_{\al})\\
&=&F([\yh\ov{\al}{\lla}\xg\ov{\al}{\lra}\yh]).
\end{eqnarray*}
\end{proof}

By this corollary, we define as follows.
\begin{dfn}\label{DefFDef}
An object $F$ in $\Add(\Tcal,\RMod)$ $($resp. $\Add(\Tcal,\Sett)$, $\Add(\Sp,\Sett))$ is called {\it deflative} if for any stab-surjective 1-cell $\al\co\xg\to\yh$, 
\[ F([\yh\ov{\al}{\lla}\xg\ov{\al}{\lra}\yh])=\id_{F(\yh)} \]
holds. We denote the full subcategory of deflative objects by $\Add_{\dfl}(\Tcal,\RMod)$ $($resp. $\Add_{\dfl}(\Tcal,\Sett)$, $\Add_{\dfl}(\Sp,\Sett))$.
\end{dfn}

\subsection{Bigger Burnside rings}

We introduce an example of Mackey functor, the {\it bigger Burnside functor} $\Obig$, which is not deflative.
This plays a similar role to the ordinary Burnside functor for a fixed finite group $G$. (For example, in \cite{N_Tensor}, the category $\MackS$ is shown to be symmetric monoidal with unit $\Obig$.)

\medskip

By Example \ref{ExRep}, especially we have the following.
\begin{ex}\label{ExBigBurn}
We have an object $\Tcal(\frac{\pt}{e},-)$ in $\Add(\Tcal,\Sett)$. We call the corresponding Mackey functor the {\it bigger Burnside functor}, and denote it by $\Obig\in\Ob(\MackS)$. 
\end{ex}

\begin{rem}
By Proposition \ref{PropAddFtrMack} and Yoneda's lemma, there is a natural isomorphism of abelian groups
\[ \MackS(\Obig,M)\cong M(\frac{\pt}{e}) \]
for any $M\in\Ob(\MackS)$. When $M=\Obig$, this gives an isomorphism for the endomorphism ring of $\Obig$
\[ \MackS(\Obig,\Obig)\cong \Obig(\frac{\pt}{e}). \]
\end{rem}

It can be easily shown that $\Obig$ is not deflative.
In fact, for any $\xg\in\Sbb^0$ with $X\ne\emptyset$, the representable functor $\Tcal(\xg,-)\in\Ob(Add(\Tcal,\Sett))$ becomes non-deflative. For simplicity, we only show in the following case.
\begin{claim}
For $\xg\in\Sbb^0$, if $X$ is $G$-transitive, then $F=\Tcal(\xg,-)$ is non-deflative.
\end{claim}
\begin{proof}
By Corollary \ref{CorAdded1}, replacing $G$ if necessary, we may assume $\xg$ is of the form $\frac{\pt}{G}$ from the beginning.
Take a finite group $G\ppr$ and a surjective group homomorphism $p\co G\ppr\rightarrow G$ satisfying $|G\ppr| >|G|$, and let $\Scal_p$ be the span $\Scal_p=(\frac{\pt}{G}\ov{\frac{\pt}{p}}{\lla}\frac{\pt}{G\ppr}\ov{\frac{\pt}{p}}{\lra}\frac{\pt}{G})$. 
Since $\frac{\pt}{G\ppr}$ and $\frac{\pt}{G}$ are never equivalent in $\Sbb$ by Proposition \ref{Rem0519_2}, we have $[\Scal_p]\ne\id_{\frac{\pt}{G}}$.
This means that the endomorphism on $F(\frac{\pt}{G})=\Tcal(\frac{\pt}{G},\frac{\pt}{G})$
\[ F([\Scal_p])=[\Scal_p]\ci-\co \Tcal(\frac{\pt}{G},\frac{\pt}{G})\to \Tcal(\frac{\pt}{G},\frac{\pt}{G}) \]
satisfies
\[ F([\Scal_p])(\id_{\frac{\pt}{G}})=[\Scal_p]\ne\id_{\frac{\pt}{G}}=\id_{F(\frac{\pt}{G})}(\id_{\frac{\pt}{G}}), \]
and thus $F([\Scal_p])\ne\id_{F(\frac{\pt}{G})}$.
By Corollary \ref{CorFDef}, this means $F$ is non-deflative.
\end{proof}

By definition, for any 0-cell $\xg\in\Sbb^0$ we have
\begin{equation}\label{Eq_Obig}
\Obig(\xg)=\Tcal(\frac{\pt}{e},\xg)=K_0(\Sp(\frac{\pt}{e},\xg)).
\end{equation}
A closer look at this shows that $\Obig(\xg)$ has a structure of a commutative ring related to the ordinary Burnside ring (Proposition \ref{PropBurntoBig}). With this view, we call $\Obig(\xg)$ the {\it bigger Burnside ring} over $\xg$.

\medskip

By $(\ref{Eq_Obig})$, we see that $\Obig(\xg)$ arises from the following {\it slice 2-category} $\SoverX$, which is an instance of a {\it lax comma category}. (See \cite[section 4, 4.1]{Kelly} for a general definition. Our case is realized there, if we let $\Gamma$ to be the identity on 0-cells and 1-cells, and put $\Delta$ to be the constant functor.)
By definition, this 2-category $\SoverX$ can be identified with $\mathrm{Span}^{\xg}_{\frac{\pt}{e}}$, i.e., \lq the left half' of the 2-category defined in Definition \ref{Def2SpanXY}.

In the following, $\Cbb$ denotes a 2-category with invertible 2-cells, as before.
\begin{dfn}\label{Def2Slice}
Let $X$ be any 0-cell in $\Cbb$. Then a 2-category $\SoverXa$ is defined as follows.
\begin{enumerate}
\item[{\rm (0)}] A 0-cell in $\SoverXa$ is a 1-cell $(\aaxa)$ in $\Cbb$, from some $A\in\Cbb^0$.
\item[{\rm (1)}] A 1-cell in $\SoverXa$ from $(\aaxa)$ to $(\bbxa)$ is a pair $(\varphi,\mu)$ of a 1-cell $\varphi$ and a 2-cell $\mu$ in $\Cbb$ as in the following diagram.
\[
\xy
(-10,6)*+{A}="0";
(10,6)*+{B}="2";
(0,-8)*+{X}="4";
{\ar^{\varphi} "0";"2"};
{\ar_(0.4){\al} "0";"4"};
{\ar^(0.38){\be} "2";"4"};
{\ar@{=>}^{\mu} (2,3);(-2,0)};
\endxy
\]
\item[{\rm (2)}] If $(\varphi,\mu)\co (\aaxa)\to(\bbxa)$ and $(\varphi\ppr,\mu\ppr)\co(\aaxa)\to(\bbxa)$ are 1-cells in $\SoverXa$, then a 2-cell $\ep\co (\varphi,\mu)\tc(\varphi\ppr,\mu\ppr)$ in $\SoverXa$ is a 2-cell $\ep\co \varphi\tc \varphi\ppr$ in $\Cbb$, which makes the following diagram commutative.
\[
\xy
(-10,6)*+{\be\ci \varphi}="0";
(10,6)*+{\be\ci \varphi\ppr}="2";
(0,-8)*+{\al}="4";
(0,10)*+{}="5";
{\ar@{=>}^{\be\ci\ep} "0";"2"};
{\ar@{=>}_(0.4){\mu} "0";"4"};
{\ar@{=>}^(0.4){\mu\ppr} "2";"4"};
{\ar@{}|\circlearrowright "4";"5"};
\endxy
\]
\end{enumerate}

Composition of 1-cells
\[ (\aaxa)\ov{(\varphi,\mu)}{\lra}(\bbxa)\ov{(\psi,\nu)}{\lra}(\ccxa) \]
is defined to be
\[ (\psi\ci \varphi,\mu\cdot(\nu\ci \varphi))\co (\aaxa)\to(\ccxa). \]
Vertical composition of 2-cells
\[
\xy
(-22,0)*+{(\aaxa)}="0";
(22,0)*+{(\bbxa)}="2";
{\ar@/^2.0pc/^{(\varphi,\mu)} "0";"2"};
{\ar|*+{_{(\varphi\ppr,\mu\ppr)}} "0";"2"};
{\ar@/_2.0pc/_{(\varphi\pprr,\mu\pprr)} "0";"2"};
{\ar@{=>}^{\ep} (0,6);(0,3)};
{\ar@{=>}^{\ep\ppr} (0,-3);(0,-6)};
\endxy
\]
is defined to be $\ep\ppr\cdot\ep$, using the vertical composition in $\Cbb$.

Horizontal composition of 2-cells
\[
\xy
(-40,0)*+{(\aaxa)}="0";
(0,0)*+{(\bbxa)}="2";
(40,0)*+{(\ccxa)}="4";
{\ar@/^1.6pc/^{(\varphi,\mu)} "0";"2"};
{\ar@/_1.6pc/_{(\varphi\ppr,\mu\ppr)} "0";"2"};
{\ar@/^1.6pc/^{(\psi,\nu)} "2";"4"};
{\ar@/_1.6pc/_{(\psi\ppr,\nu\ppr)} "2";"4"};
{\ar@{=>}^{\ep} (-20,2.5);(-20,-2.5)};
{\ar@{=>}^{\delta} (20,2.5);(20,-2.5)};
\endxy
\]
is defined to be $\delta\ci\ep$, using the horizontal composition in $\Cbb$.

\end{dfn}
Then $\SoverXa$ becomes in fact a 2-category. This can be shown in a similar way as Fact \ref{FactHoff}. (If $\Cbb=\Sbb$, this is indeed a particular case of Fact \ref{FactHoff}, since  $\SoverX$ can be identified with $\mathrm{Span}^{\xg}_{\frac{\pt}{e}}$.)
Moreover, the following is also shown by a general argument on 2-categories.
\begin{prop}\label{Prop2Slice}
$\ \ $
\begin{enumerate}
\item If $\Cbb$ admits bicoproducts, then $\SoverXa$ admits bicoproducts.
\item If $\Cbb$ admits bipullbacks, then  $\SoverXa$ admits biproducts.
\end{enumerate}
\end{prop}
\begin{proof}
{\rm (1)} This is a special case of Proposition \ref{Prop2CoSpan}.
For any pair of 0-cells $(\aaxa)$ and $(\bbxa)$ in $\SoverXa$, if we take their bicoproduct $A\ov{\ups_A}{\lra}A\am B\ov{\ups_B}{\lla}B$ in $\Cbb$, then by its universal property,
 we obtain a diagram
\[
\xy
(0,-7)*+{X}="0";
(-24,10)*+{A}="2";
(0,10)*+{A\am B}="4";
(24,10)*+{B}="6";
{\ar_{\al} "2";"0"};
{\ar|*+{_{\al\cup \be}} "4";"0"};
{\ar^{\be} "6";"0"};
{\ar^(0.36){\ups_A} "2";"4"};
{\ar_(0.36){\ups_B} "6";"4"};
{\ar@{=>}_{\lam_A} (-5,5);(-8.5,2)};
{\ar@{=>}^{\lam_B} (5,5);(8.5,2)};
\endxy.
\]
This gives a bicoproduct $(\aaxa)\ov{(\ups_A,\lam_A)}{\lra}(A\am B\ov{\al\cup \be}{\lra}X)\ov{(\ups_B,\lam_B)}{\lla}(\bbxa)$ in $\SoverXa$.

{\rm (2)} For any pair of 0-cells $(\aaxa)$ and $(\bbxa)$ in $\SoverXa$, if we take their bipullback
\[
\xy
(-8,6)*+{F}="0";
(8,6)*+{B}="2";
(-8,-6)*+{A}="4";
(8,-6)*+{X}="6";
{\ar^(0.52){\wp_B} "0";"2"};
{\ar_{\wp_A} "0";"4"};
{\ar^{\be} "2";"6"};
{\ar_{\al} "4";"6"};
{\ar@{=>}^{\chi} (-2,0);(2,0)};
\endxy
\]
in $\Cbb$, then
\[ (\aaxa)\ov{(\wp_A,\chi)}{\lla}(F\ov{\be\ci\wp_B}{\to}X)\ov{(\wp_B,\id)}{\lra}(\bbxa) \]
gives a biproduct of $(\aaxa)$ and $(\bbxa)$ in $\SoverXa$.
Also remark that $(\id,\chi)\co(F\ov{\be\ci\wp_B}{\to}X)\ov{\simeq}{\lra}(F\ov{\al\ci\wp_A}{\to}X)$ is an equivalence in $\SoverXa$.
\end{proof}

\begin{cor}
We say two 0-cells $(\aaxa)$ and $(\bbxa)$ in $\SoverXa$ are {\it equivalent} if there exists an equivalence $(\varphi,\mu)\co(\aaxa)\to(\bbxa)$. Then the set of equivalence classes 
\[ (\SoverXa)^0/\text{equivalence}\quad \big(=\Sp(\Cbb)(\pt,X)\big) \]
forms a semi-ring with the addition and the multiplication induced from bicoproducts and biproducts. In the same notation as in Definition \ref{Def_0519_1}, we denote the equivalence class of $(\aaxa)$ by $[\aaxa]$.
\end{cor}

Now we return to the case $\Cbb=\Sbb$.
\begin{rem}
By {\rm (\ref{Eq_Obig})}, $\Obig(\xg)$ is nothing but the additive completion of the above semi-ring $(\SoverX)^0/\text{{\it equivalence}}$.
\end{rem}




\begin{prop}\label{PropFunctSIm}
Let $G$ be any finite group, and let $X$ be any finite $G$-set.
To any 1-cell in $\Sbb/_{\xg}$
\[ (\varphi,\mu)\co(\kagx)\to(\hbgx), \]
we associate a map
\[ \mathfrak{s}_{(\varphi,\mu)}=f\co \SIm\al\to\SIm\be \]
defined by
\[ f([\xi,a])=[\xi\mu_a,\varphi(a)]\quad(\fa [\xi,a]\in\SIm\al=(G\times A)/\!\sim\ ). \]
Then we have the following.
\begin{enumerate}
\item $f$ is well-defined. Moreover, if there is a 2-cell $\omega\co (\varphi,\mu)\tc(\varphi\ppr,\mu\ppr)$, then the associated maps $f=\mathfrak{s}_{(\varphi,\mu)}$ and $f\ppr=\mathfrak{s}_{(\varphi\ppr,\mu\ppr)}$ are equal.
\item $f$ is a $G$-map, which makes the following diagram in $\Gs$ commutative, where $\wt{\al}$ and $\wt{\be}$ are those obtained in Proposition \ref{PropSIm}.
\[
\xy
(-8,8)*+{\SIm\al}="0";
(-8,-8)*+{\SIm\be}="2";
(10,0)*+{X}="4";
(-16,0)*+{}="5";
{\ar^{\wt{\al}} "0";"4"};
{\ar_{\wt{\be}} "2";"4"};
{\ar_{f} "0";"2"};
{\ar@{}|\circlearrowright "4";"5"};
\endxy
\]
\item $\mu\co(\frac{\ups_{\be}}{\thh_{\be}})\ci(\frac{\varphi}{\thh_{\varphi}})\tc(\frac{f}{G})\ci(\frac{\ups_{\al}}{\thh_{\al}})$ is a 2-cell in $\Sbb$. Thus we have the following diagram in $\Sbb$.
\[
\xy
(-22,20)*+{\frac{A}{K}}="0";
(-22,-20)*+{\frac{B}{H}}="2";
(-2,8)*+{\frac{\SIm\al}{G}}="4";
(4,18)*+{}="5";
(-2,-8)*+{\frac{\SIm\be}{G}}="6";
(4,-18)*+{}="7";
(26,0)*+{\xg}="8";
(-8,0)*+{}="9";
{\ar^{\ups_{\al}} "0";"4"};
{\ar_{\ups_{\be}} "2";"6"};
{\ar_{\varphi} "0";"2"};
{\ar^{\frac{\wt{\al}}{G}} "4";"8"};
{\ar_{\frac{\wt{\be}}{G}} "6";"8"};
%
{\ar|*+{_{\frac{f}{G}}} "4";"6"};
{\ar@/^1.8pc/^{\al} "0";"8"};
{\ar@/_1.8pc/_{\be} "2";"8"};
{\ar@{=>}^{\mu} (-17,0);(-13,0)};
{\ar@{}|\circlearrowright "4";"5"};
{\ar@{}|\circlearrowright "6";"7"};
{\ar@{}|\circlearrowright "8";"9"};
\endxy
\]
\item If $(\varphi,\mu)$ is an equivalence, then $f$ is an isomorphism of $G$-sets. 
\end{enumerate}
\end{prop}
\begin{proof}
{\rm (1)} Suppose $[\xi,a]=[\xi\ppr,a\ppr]$ holds as elements in $\SIm\al$. This means there exists $k\in K$ satisfying $a\ppr=ka$ and $\xi=\xi\ppr\thh_{\al,a}(k)$.
Then we obtain
\begin{eqnarray*}
f([\xi,a])&=&[\xi\mu_a,\varphi(a)]\ =\ [\xi\ppr\thh_{\al,a}(k)\mu_a,\varphi(a)]\\
&=&[\xi\ppr\mu_{ka},\thh_{\varphi,a}(k)\varphi(a)]\ =\ [\xi\ppr\mu\ppr_{a\ppr},\varphi\ppr(a\ppr)]\ =\ f\ppr([\xi\ppr,a\ppr]).
\end{eqnarray*}
This shows the well-definedness of $f$, and the equation $f=f\ppr$.

\noindent {\rm (2)}
$f$ is a $G$-map, since we have
\[ f(g[\xi,a])=f([g\xi,a])=[g\xi\mu_a,\varphi(a)]=g\cdot (f([\xi,a])) \]
for any $[\xi,a]\in\SIm\al$ and $g\in G$.
The commutativity follows from
\[ \wt{\be}\circ f([\xi,a])=\wt{\be}([\xi\mu_a,\varphi(a)])=\xi\mu_a\cdot(\be\ci\varphi(a))=\xi\al(a)=\wt{\al}([\xi,a]). \]

\noindent {\rm (3)} For any $a\in A$ and $k\in K$, we have
\[ \mu_a\cdot(\ups_{\be}\ci\varphi)(a)=\mu_a\cdot[e,\varphi(a)]=f([e,a])=f\ci\ups_{\al}(a), \]
\[ \mu_{ka}\cdot(\thh_{\be}\ci\thh_{\varphi})_a(k)\cdot\mu_a\iv=\thh_{\al,a}(k). \]

\noindent {\rm (4)} For the identity morphism
\[ (\id,\id)\co [\kagx]\to [\kagx], \]
we have
\[ \mathfrak{s}_{(\id,\id)}([\xi,a])=[\xi,a] \]
for any $[\xi,a]\in\SIm\al$, which means $\mathfrak{s}_{(\id,\id)}=\id_{\SIm\al}$.

Suppose $(\varphi,\mu)\co(\kagx)\to(\hbgx)$ is an equivalence with a quasi-inverse $(\psi,\nu)\co(\hbgx)\to(\kagx)$. For any $[\xi,a]\in\SIm\al$, we have
\[ \mathfrak{s}_{(\psi,\nu)\ci (\varphi,\mu)}([\xi,a])=\mathfrak{s}_{(\psi\ci\varphi,\mu\cdot(\nu\ci\varphi))}([\xi,a])=[\xi\mu_a\nu_{\varphi(a)}, \psi\ci\varphi(a)]=\mathfrak{s}_{(\psi,\nu)}\ci\mathfrak{s}_{(\varphi,\mu)}([\xi,a]). \]
By {\rm (1)}, it follows $\mathfrak{s}_{(\psi,\nu)}\ci\mathfrak{s}_{(\varphi,\mu)}=\mathfrak{s}_{(\id,\id)}=\id$. Similarly we have $\mathfrak{s}_{(\varphi,\mu)}\ci\mathfrak{s}_{(\psi,\nu)}=\id$.
\end{proof}


\begin{prop}\label{PropBurntoBig}
Let $\xg$ be any 0-cell in $\Sbb$.
\begin{enumerate}
\item The correspondence
\[ \Omega_G(X)\to\Obig(\xg)\ ;\ [A\ov{p}{\lra}X]\mapsto [\frac{A}{G}\ov{\frac{p}{G}}{\lra}\xg] \]
preserves additions and multiplications, and thus induces a ring homomorphism.
The left hand side is the ordinary Burnside ring over $X$, namely the Grothendieck ring of the slice category $\Gs/X$.

\item The correspondence
\[ \Obig(\xg)\to\Omega_G(X)\ ;\ [\aax]\to [\SIm \al\ov{\wt{\al}}{\lra}X] \]
obtained in Proposition \ref{PropFunctSIm} preserves additions and multiplications, and thus induces a ring homomorphism.

\item The composition of homomorphisms in {\rm (1)} and {\rm (2)}
\[ \Omega_G(X)\to\Obig(\xg)\to\Omega_G(X) \]
is identity.
\end{enumerate}
\end{prop}
\begin{proof}
Remark that the ring structure on $\Obig(\xg)$ is given by bicoproducts and biproducts obtained in Proposition \ref{Prop2Slice}.

{\rm (1)} This follows from Proposition \ref{Prop2CoprodEqui} and \ref{PropEqui2Pullback}.

{\rm (2)} This follows from Corollary \ref{CorAdd2} and Proposition \ref{PropSImAndPB}.

{\rm (3)} This follows from Corollary \ref{CorAdd1} {\rm (2)}.
\end{proof}

\begin{rem}
In Proposition \ref{PropBurntoBig}, the homomorphisms obtained in {\rm (2)} are shown to form a morphism of Mackey functors, while the homomorphisms in {\rm (1)} do not.
\end{rem}

\section{Interpretation of biset functors}

In this section, we give an interpretation of biset functors as Mackey functors on $\Sbb$. Indeed, we will show that the category of biset functors $\BisetFtr$ is equivalent to that of deflative Mackey functors $\Mack_{\dfl}^R(\Sbb)$.

\subsection{Range of a span}

The idea to relate bisets and spans in $\Sbb$ is the {\it flipping up} as in the introduction, which is more directly performed as in the following Definition \ref{DefSU}. To fill the gap between spans and bisets, we will introduce the {\it range} biset, using stabilizerwise image.


We can associate a span to any biset as follows.
\begin{dfn}\label{DefSU}
Let $G,H$ be finite groups. For any $G$-$H$-biset $U$, we associate a span $S_U$ to $\frac{\pt}{G}$ from $\frac{\pt}{H}$ by
\[ S_U=(\frac{\pt}{G}\ov{\frac{\pt}{\prg}}{\lla}\frac{U}{G\times H}\ov{\frac{\pt}{\prh}}{\lra}\frac{\pt}{H}), \]
where $\prg\co G\times H\to G$ and $\prh\co G\times H\to H$ are the projections. Here $U$ is regarded as a $G\times H$-set, as in the introduction.
\end{dfn}

\begin{claim}\label{RemSU}
Let $U,U\ppr$ be two $G$-$H$-bisets. Then, we have the following.
\begin{enumerate}
\item $[S_{U\am U\ppr}]=[S_U]+[S_{U\ppr}]$.
\item If there is an isomorphism of $G$-$H$-bisets $U\cong U\ppr$, then $[S_U]=[S_{U\ppr}]$.
\end{enumerate}
\end{claim}
\begin{proof}
{\rm (1)} This is straightforward.

{\rm (2)} Assume there is an isomorphism of $G$-$H$-bisets $f\co U\ov{\cong}{\lra}U\ppr$. This yields an equivariant isomorphism $\frac{f}{G\times H}\co \frac{U}{G\times H}\ov{\cong}{\lra}\frac{U\ppr}{G\times H}$, which makes the following diagram commutative.
\[
\xy
(-24,0)*+{\frac{\pt}{G}}="0";
(6,0)*+{}="1";
(0,12)*+{\frac{U}{G\times H}}="2";
(0,-12)*+{\frac{U\ppr}{G\times H}}="4";
(24,0)*+{\frac{\pt}{H}}="6";
(-6,0)*+{}="7";
{\ar_{\frac{\pt}{\prg}} "2";"0"};
{\ar^{\frac{\pt}{\prh}} "2";"6"};
{\ar|*+{_{\frac{f}{G\times H}}} "2";"4"};
{\ar^{\frac{\pt}{\prg}} "4";"0"};
{\ar_{\frac{\pt}{\prh}} "4";"6"};
{\ar@{}|\circlearrowright "0";"1"};
{\ar@{}|\circlearrowright "6";"7"};
\endxy
\]
\end{proof}

Conversely, we associate a biset to any span, which we call the {\it range} of the span.
\begin{dfn}\label{DefRange}
Let $S=(\xg\ov{\al_S}{\lla}\lsws\ov{\be_S}{\lra}\yh)$ be any span in $\Sbb$. By Corollary \ref{Cor2Prod}, we obtain a 1-cell
\[ \gamma_S\co\lsws\to \frac{X\times Y}{G\times H} \]
defined by
\begin{eqnarray*}
\gamma_S(w)=(\al_S(w),\be_S(w))&&(\fa w\in W_S),\\
\thh_{\gamma_S,w}(\ell)=(\thh_{\al_S,w}(\ell),\thh_{\be_S,w}(\ell))&&(\fa w\in W_S,\fa \ell\in L_S).
\end{eqnarray*}
Then the {\it range} of $S$ is defined to be
\[ \Rg(S)=\SIm\gamma_S =(G\times H\times W_S)/\sim, \]
regarded as a $G$-$H$-biset by
\begin{eqnarray*}
g[\xi,\eta,w]h&=&(g,h\iv)[\xi,\eta,w]\\
&=&[g\xi,h\iv\eta,w]\qquad(\fa g\in G,\fa h\in H,\fa [\xi,\eta,w]\in \Rg(S)).
\end{eqnarray*}
\end{dfn}

\begin{rem}
By definition, the range $\Rg(S)$ is
\[ \Rg(S)=(G\times H\times W_S)/\sim, \]
where $(\xi,\eta,w),(\xi\ppr,\eta\ppr,w\ppr)\in G\times H\times W_S$ are equivalent $($ i.e. $[\xi,\eta,w]=[\xi\ppr,\eta\ppr,w\ppr])$ if and only if there exists $\ell\in L_S$ satisfying
\[ w\ppr=\ell w,\ \xi=\xi\ppr\thh_{\al_S,w}(\ell),\ \eta=\eta\ppr\thh_{\be_S,w}(\ell). \]
\end{rem}

\begin{rem}\label{RemRange1}
For any span $S=(\spS)$, we have the following commutative diagram.
\begin{equation}\label{Diag_RangeS}
\xy
(0,18)*+{\lsws}="0";
(0,-2)*+{\frac{\Rg(S)}{G\times H}}="2";
(-26,0)*+{}="3";
(0,-14)*+{\frac{X\times Y}{G\times H}}="4";
(26,0)*+{}="5";
(-32,-14)*+{\xg}="6";
(32,-14)*+{\yh}="8";
{\ar_{\al_S} "0";"6"};
{\ar^{\be_S} "0";"8"};
{\ar_(0.6){\frac{\ups_{\gamma_S}}{\thh_{\gamma_S}}} "0";"2"};
{\ar_{\wt{\gamma_S}} "2";"4"};
{\ar^(0.6){\frac{\wp_X}{\prg}} "4";"6"};
{\ar_(0.6){\frac{\wp_Y}{\prh}} "4";"8"};
{\ar@{}|\circlearrowright "3";"4"};
{\ar@{}|\circlearrowright "4";"5"};
\endxy
\end{equation}
\end{rem}

\begin{prop}\label{PropRangeIsom}
If two spans in $\Sbb$ to $\xg$ from $\yh$
\[ S=(\spS),\ \ \text{and}\ \ T=(\spT) \]
are equivalent, then their ranges $\Rg(S)$ and $\Rg(T)$ are isomorphic as $G$-$H$-bisets.

\end{prop}
\begin{proof}
Since $S$ is equivalent to $T$, there exists a diagram
\[
\xy
(-20,0)*+{\xg}="0";
(0,10)*+{\lsws}="2";
(0,-10)*+{\ltwt}="4";
(20,0)*+{\yh}="6";
{\ar_{\al_S} "2";"0"};
{\ar^{\be_S} "2";"6"};
{\ar^{\lam} "2";"4"};
{\ar^{\al_T} "4";"0"};
{\ar_{\be_T} "4";"6"};
{\ar@{=>}^{\nu_X} (-4,-2);(-8,2)};
{\ar@{=>}_{\nu_Y} (4,-2);(8,2)};
\endxy
\]
and a diagram
\[
\xy
(-20,0)*+{\xg}="0";
(0,10)*+{\ltwt}="2";
(0,-10)*+{\lsws}="4";
(20,0)*+{\yh}="6";
{\ar_{\al_T} "2";"0"};
{\ar^{\be_T} "2";"6"};
{\ar^{\kappa} "2";"4"};
{\ar^{\al_S} "4";"0"};
{\ar_{\be_S} "4";"6"};
{\ar@{=>}^{\mu_X} (-4,-2);(-8,2)};
{\ar@{=>}_{\mu_Y} (4,-2);(8,2)};
\endxy
\]
with
\[
\xy
(-40,0)*+{\lsws}="0";
(-20,0)*+{\ltwt}="2";
(0,0)*+{\lsws}="4";
(20,0)*+{\ltwt}="6";
{\ar^{\lam} "0";"2"};
{\ar^{\kappa} "2";"4"};
{\ar_{\lam} "4";"6"};
{\ar@/_1.8pc/_{\id} "0";"4"};
{\ar@/^1.8pc/^{\id} "2";"6"};
{\ar@{=>}^{\rho_S} (-20,-3);(-20,-6)};
{\ar@{=>}^{\rho_T} (0,3);(0,6)};
\endxy,
\]
which satisfy
\begin{eqnarray*}
&\rho_{S,w}\cdot(\kappa\ci\lam(w))=w,&\\
&\thh_{\al_S,w}(\rho_{S,w})=\al_S\ci\rho_{S,w}=\nu_{X,w}\cdot\mu_{X,\lam(w)},&\\&\thh_{\be_S,w}(\rho_{S,w})=\be_S\ci\rho_{S,w}=\nu_{Y,w}\cdot\mu_{Y,\lam(w)}&
\end{eqnarray*}
for any $w\in W_S$. Similar equations hold for $\rho_T$.

If we define a map $f_{\lam}\co \Rg(S)\to\Rg(T)$ by
\[ f_{\lam}([\xi,\eta,w])=[\xi\nu_{X,w},\eta\nu_{Y,w},\lam(w)] \]
for any $[\xi,\eta,w]\in\Rg(S)$, then the following holds.
\begin{claim}\label{ClaimRangeIsom}
$f_{\lam}$ is a well-defined map of $G$-$H$-bisets.
\end{claim}
Suppose Claim \ref{ClaimRangeIsom} is shown.
By symmetry, we also have a $G$-$H$-map $f_{\kappa}\co \Rg(T)\to \Rg(S)$ defined by
\[ f_{\kappa}([x,y,v])=[x\mu_{X,v},y\mu_{Y,v},\kappa(v)] \]
for any $[x,y,v]\in\Rg(T)$.
Then we obtain
\begin{eqnarray*}
f_{\kappa}\ci f_{\lam}([\xi,\eta,w])&=&f_{\kappa}([\xi\nu_{X,w},\eta\nu_{Y,w},\lam(w)])\\
&=&[\xi\nu_{X,w}\mu_{X,\lam(w)},\eta\nu_{Y,w}\mu_{Y,\lam(w)},\kappa(\lam(w))]\\
&=&[\xi\thh_{\al_S,w}(\rho_{S,w}),\eta\thh_{\be_S,w}(\rho_{S,w}),\rho_{S,w}\iv w]\ =\ [\xi,\eta,w]
\end{eqnarray*}
for any $[\xi,\eta,w]\in\Rg(S)$, namely, $f_{\kappa}\ci f_{\lam}=\id_{\Rg(S)}$.
Similarly we have $f_{\lam}\ci f_{\kappa}=\id_{\Rg(T)}$, and thus $f_{\lam}$ is an isomorphism of $G$-$H$-bisets.

Thus it remains to show Claim \ref{ClaimRangeIsom}. To show the well-definedness of $f_{\lam}$, suppose $(\xi,\eta,w),(\xi\ppr,\eta\ppr,w\ppr)\in G\times H\times W_S$ satisfy $[\xi,\eta,w]=[\xi\ppr,\eta\ppr,w\ppr]$ in $\Rg(S)$. By definition there exists $\ell\in L_S$ which satisfies
\[ w\ppr=\ell w,\ \xi=\xi\ppr\thh_{\al_S,w}(\ell),\ \eta=\eta\ppr\thh_{\be_S,w}(\ell). \]
Then we have
\[ \xi\ppr\nu_{X,w\ppr}=\xi\thh_{\al_S,w}(\ell)\iv\nu_{X,\ell w}=\xi\nu_{X,w}\thh_{\al_T,\lam(w)}(\thh_{\lam,w}(\ell))\iv \]
and similarly $\eta\ppr\nu_{Y,w\ppr}=\eta\nu_{Y,w}\thh_{\be_T,\lam(w)}(\thh_{\lam,w}(\ell))\iv$.
Thus we obtain
\begin{eqnarray*}
&\ &[\xi\ppr\nu_{X,w\ppr},\eta\ppr\nu_{Y,w\ppr},\lam(w\ppr)]\\%
&=&[\xi\nu_{X,w}\thh_{\al_T,\lam(w)}(\thh_{\lam,w}(\ell))\iv,\eta\nu_{Y,w}\thh_{\be_T,\lam(w)}(\thh_{\lam,w}(\ell))\iv,\thh_{\lam,w}(\ell)\lam(w)]\\
&=&[\xi\nu_{X,w},\eta\nu_{Y,w},\lam(w)]
\end{eqnarray*}
in $\Rg(T)$, which shows the well-definedness of $f_{\lam}$.
The $G$-$H$-equivariance is obvious.
\end{proof}

\begin{prop}\label{PropRangeDeflative}
For any object $F$ in $\Add(\Tcal,\RMod)$, the following are equivalent.
\begin{enumerate}
\item $F$ is deflative.
\item For any pair of finite groups $G,H$ and any span $S=(\frac{\pt}{G}\ov{\al_S}{\lla}\frac{W_S}{L_S}\ov{\be_S}{\lra}\frac{\pt}{H})$ to $\frac{\pt}{G}$ from $\frac{\pt}{H}$, the equality $F([S])=F([S_{\Rg(S)}])$ holds.
\end{enumerate}
\end{prop}
\begin{proof}
Let $S=(\frac{\pt}{G}\ov{\al_S}{\lla}\frac{W_S}{L_S}\ov{\be_S}{\lra}\frac{\pt}{H})$ be any span to $\frac{\pt}{G}$ from $\frac{\pt}{H}$. If $F$ is deflative, then since $\ups_{\gamma_S}$ is stab-surjective in diagram $(\ref{Diag_RangeS})$, we have
\begin{eqnarray*}
F([S])&=&F(\Tbf_{\al_S})\ci F(\Rbf_{\be_S})\\
&=&F(\Tbf_{\wp_X\ci\wt{\gamma_S}})\ci F(\Tbf_{\ups_{\gamma_S}})\ci F(\Rbf_{\ups_{\gamma_S}})\ci F(\Rbf_{\wp_Y\ci\wt{\gamma_S}})\\
&=&F(\Tbf_{\wp_X\ci\wt{\gamma_S}})\ci F(\Rbf_{\wp_Y\ci\wt{\gamma_S}})\ =\ F([S_{\Rg(S)}]).
\end{eqnarray*}

Conversely, assume $F$ satisfies {\rm (2)}. 
We confirm the condition {\rm (3)} in Corollary \ref{CorFDef}.
For any surjective group homomorphism $p\co G\to H$, take a span
\[ \Scal_p=(\frac{\pt}{H}\ov{\frac{\pt}{p}}{\lla}\frac{\pt}{G}\ov{\frac{\pt}{p}}{\lra}\frac{\pt}{H}) \]
and let $\gamma=(\frac{\pt}{p},\frac{\pt}{p})$ be the 1-cell obtained by the universal property of the biproduct. The following diagram is commutative.
\[
\xy
(0,6)*+{\frac{\pt}{G}}="0";
(-15,1)*+{}="3";
(0,-10)*+{\frac{\pt}{H\times H}}="4";
(15,1)*+{}="5";
(-22,-10)*+{\frac{\pt}{H}}="6";
(22,-10)*+{\frac{\pt}{H}}="8";
{\ar_{\frac{\pt}{p}} "0";"6"};
{\ar^{\frac{\pt}{p}} "0";"8"};
{\ar_{\gamma} "0";"4"};
{\ar^(0.6){\frac{\pt}{\pro_1}} "4";"6"};
{\ar_(0.6){\frac{\pt}{\pro_2}} "4";"8"};
{\ar@{}|\circlearrowright "3";"4"};
{\ar@{}|\circlearrowright "4";"5"};
\endxy
\]
It suffices to show $F([\Scal_p])=\id$. 
Let $\Delta\co H\to H\times H$ denote the diagonal morphism, and let
\[ \ups\co\pt=H/H\to (H\times H)/\Delta(H) \]
be the map defined by $\ups(eH)=(e,e)\Delta(H)$. Then
\[ \frac{\ups}{\Delta}\co\frac{\pt}{H}\ov{\simeq}{\lra}\frac{(H\times H)/\Delta(H)}{H\times H} \]
gives an equivalence ($\Ind$-equivalence), and the composition of 1-cells
\[ \frac{\pt}{G}\ov{\frac{\pt}{p}}{\lra}\frac{\pt}{H}\ov{\frac{\ups}{\Delta}}{\lra}\frac{(H\times H)/\Delta(H)}{H\times H}\ov{\frac{\pt}{H\times H}}{\lra}\frac{\pt}{H\times H} \]
becomes equal to $\gamma$. Since $\frac{\ups}{\Delta}\ci\frac{\pt}{p}$ is stab-surjective, this means $\SIm\gamma\cong\frac{(H\times H)/\Delta(H)}{H\times H}$, and thus $\Rg(\Scal_p)\cong(H\times H)/\Delta(H)$ as $H$-$H$-bisets. By the assumption we obtain
\[ F([\Scal_p])=F([S_{\Rg(\Scal_p)}])=F([\frac{\pt}{H}\ov{\frac{\pt}{\pro_1}}{\lla}\frac{(H\times H)/\Delta(H)}{H\times H}\ov{\frac{\pt}{\pro_2}}{\lra}\frac{\pt}{H}]). \]
Moreover, since 
\[
\xy
(0,6)*+{\frac{\pt}{H}}="0";
(-15,1)*+{}="3";
(0,-10)*+{\frac{(H\times H)/\Delta(H)}{H\times H}}="4";
(15,1)*+{}="5";
(-22,-10)*+{\frac{\pt}{H}}="6";
(22,-10)*+{\frac{\pt}{H}}="8";
{\ar_{\id} "0";"6"};
{\ar^{\id} "0";"8"};
{\ar_{\frac{\ups}{\Delta}}^{\simeq} "0";"4"};
{\ar^(0.6){\frac{\pt}{\pro_1}} "4";"6"};
{\ar_(0.6){\frac{\pt}{\pro_2}} "4";"8"};
{\ar@{}|\circlearrowright "3";"4"};
{\ar@{}|\circlearrowright "4";"5"};
\endxy
\]
is commutative, we have $[\frac{\pt}{H}\ov{\frac{\pt}{\pro_1}}{\lla}\frac{(H\times H)/\Delta(H)}{H\times H}\ov{\frac{\pt}{\pro_2}}{\lra}\frac{\pt}{H}]=[\frac{\pt}{H}\ov{\id}{\lla}\frac{\pt}{H}\ov{\id}{\lra}\frac{\pt}{H}]=\id$, and thus
\[ F([\frac{\pt}{H}\ov{\frac{\pt}{\pro_1}}{\lla}\frac{(H\times H)/\Delta(H)}{H\times H}\ov{\frac{\pt}{\pro_2}}{\lra}\frac{\pt}{H}])=F(\id)=\id. \]
\end{proof}

\begin{prop}\label{PropComposBSSP}
For any consecutive spans
\[ S=(\spSc),\ \ \text{and}\ \ T=(\spTzc), \]
there exists a natural isomorphism of $K$-$G$-bisets
\[ \Rg(T)\times_H\Rg(S)\cong \Rg(T\ci S). \]
\end{prop}
\begin{proof}
By definition, composition $T\ci S$ is defined by using a bipullback as
\[
\xy
(0,12)*+{\frac{F}{L_T\times L_S}}="0";
(-15,0)*+{\frac{W_T}{L_T}}="2";
(15,0)*+{\frac{W_S}{L_S}}="4";
(-30,-12)*+{\zk}="6";
(0,-12)*+{\yh}="8";
(30,-12)*+{\xg}="10";
(36,-6)*+{.}="11";
{\ar_(0.6){\wp_{W_T}} "0";"2"};
{\ar^(0.6){\wp_{W_S}} "0";"4"};
{\ar_(0.6){\al_T} "2";"6"};
{\ar_(0.4){\be_T} "2";"8"};
{\ar^(0.4){\al_S} "4";"8"};
{\ar^(0.6){\be_S} "4";"10"};
{\ar@{=>}^{\chi} (-2,0);(2,0)};
\endxy
\]
Namely, 
\[ T\ci S=(\zk\ov{\gamma}{\lla}\frac{F}{L_T\times L_S}\ov{\delta}{\lra}\xg) \]
is given by
\begin{eqnarray*}
&F=\{(b,a,h)\in W_T\times W_S\times H\mid \be_T(b)=h\al_S(a) \},&\\
&\gamma\co F\to Z\ ;\ (b,a,h)\mapsto\al_T(b),&\\
&\delta\co F\to X\ ;\ (b,a,h)\mapsto\be_S(a).&
\end{eqnarray*}
%
Thus its range $\Rg(T\ci S)$ should be
\[ \Rg(T\ci S)=(K\times G\times F)/\sim, \]
where $(\xk)$ and $(\xkp)$ in $K\times G\times F$ are equivalent if and only if there exists $(t,s)\in L_T\times L_S$ satisfying the following equalities.
\begin{eqnarray*}
&(b\ppr,a\ppr,h\ppr)=(tb,sa,\thh_{\al_S,a}(s)h\thh_{\be_T,b}(t)\iv),&\\
&\kappa=\kappa\ppr\cdot\thh_{\al_T,b}(t),\ \ \xi=\xi\ppr\cdot\thh_{\be_S,a}(s).&
\end{eqnarray*}
The $K$-$G$-biset structure on $\Rg(T\ci S)$ is defined by
\begin{eqnarray*}
&k[\kappa,\xi,f]g=[k\kappa,g\iv\xi,f]&\\
&(\fa k\in K,\ \fa g\in G,\ \fa [\kappa,\xi,f]\in \Rg(T\ci S) ).&
\end{eqnarray*}

%
On the other hand, $\Rg(T)\times_H\Rg(S)$ is given by
\[ \Rg(T)\times_H\Rg(S)=\Set{ ([\kappa,\eta_1,b],[\eta_2,\xi,a])|%
\begin{array}{c} [\kappa,\eta_1,b]\in \Rg(T)\\ {[}\eta_2,\xi,a]\in \Rg(S) \end{array}%
}/\sim , \]
where $([\kappa,\eta_1,b],[\eta_2,\xi,a])\sim ([\kappa\ppr,\eta\ppr_1,b\ppr],[\eta\ppr_2,\xi\ppr,a\ppr])$ holds if and only if there exists $h\in H$ satisfying
\begin{eqnarray*}
&[\kappa\ppr,\eta_1\ppr,b\ppr]=[\kappa,\eta_1,b]h\iv=[\kappa,h\eta_1,b]\ \ \text{in}\ \Rg(T),&\\
&[\eta\ppr_2,\xi\ppr,a\ppr]=h[\eta_2,\xi,a]=[h\eta_2,\xi,a]\ \ \text{in}\ \Rg(S).&
\end{eqnarray*}

Using these descriptions, we can show that the map $\psi\co\Rg(T)\underset{H}{\times}\Rg(S)\to\Rg(T\ci S)$ defined by
\[ \psi([[\kappa,\eta_1,b],[\eta_2,\xi,a]])=[\kappa,\xi,(b,a,\eta_2\iv\eta_1)] \]
is a well-defined $K$-$G$-equivariant isomorphism, in a straightforward way. 

\end{proof}

\begin{rem}\label{RemSpBis}
Let $G,H$ be finite groups. For any $H$-$G$-biset $U$, the range of the span
\[ S_U=(\frac{\pt}{H}\ov{\frac{\pt}{\prh}}{\lla}\frac{U}{H\times G}\ov{\frac{\pt}{\prg}}{\lra}\frac{\pt}{G}) \]
is calculated by
\[
\xy
(0,24)*+{\frac{U}{H\times G}}="0";
(0,12)*+{\frac{U}{H\times G}}="2";
(-26,14)*+{}="3";
(0,0)*+{\frac{\pt}{H\times G}}="4";
(26,14)*+{}="5";
(-32,0)*+{\frac{\pt}{H}}="6";
(32,0)*+{\frac{\pt}{G}}="8";
{\ar_{\frac{\pt}{\prh}} "0";"6"};
{\ar^{\frac{\pt}{\prg}} "0";"8"};
{\ar@{=} "0";"2"};
{\ar_{\frac{\pt}{H\times G}} "2";"4"};
{\ar^(0.56){\frac{\pt}{\prh}} "4";"6"};
{\ar_(0.56){\frac{\pt}{\prg}} "4";"8"};
{\ar@{}|\circlearrowright "3";"4"};
{\ar@{}|\circlearrowright "4";"5"};
\endxy
\]
and thus we have $\Rg(S_U)=U$ as an $H$-$G$-biset. Remark that thus we have
\[ S_U=(\frac{\pt}{H}\ov{\frac{\pt}{\prh}}{\lla}\frac{\Rg(S_U)}{H\times G}\ov{\frac{\pt}{\prg}}{\lra}\frac{\pt}{G}). \]
\end{rem}


\begin{prop}\label{PropSIB}
Let $G$ be a finite group, and let $N\nm G$ be its normal subgroup.
Let $p\co G\to G/N$ denote the quotient homomorphism, and let
\begin{eqnarray*}
&\prg\co G\times (G/N)\to G,&\\
&\prgn\co G\times (G/N)\to G/N&
\end{eqnarray*}
be the projections. Then there is an equivalence of spans as follows, which implies $[\Rbf_{\frac{\pt}{p}}]=[S_{\Inf^G_N}]$. $($Here, $\Inf^G_N$ denotes the biset ${}_G(G/N)_{(G/N)}$ as in the introduction.$)$
\begin{equation}\label{Diag_GGGN}
\xy
(-22,0)*+{\frac{\pt}{G}}="0";
(8,-2)*+{}="1";
(0,12)*+{\frac{\pt}{G}}="2";
(0,-14)*+{\frac{(G/N)}{G\times(G/N)}}="4";
(22,0)*+{\frac{\pt}{(G/N)}}="6";
(-8,-2)*+{}="7";
{\ar_{\id} "2";"0"};
{\ar^{\frac{\pt}{p}} "2";"6"};
{\ar^{\simeq}_{\ups} "2";"4"};
{\ar^(0.6){\frac{\pt}{\prg}} "4";"0"};
{\ar_(0.6){\frac{\pt}{\prgn}} "4";"6"};
{\ar@{}|\circlearrowright "0";"1"};
{\ar@{}|\circlearrowright "6";"7"};
\endxy
\end{equation}
\end{prop}
\begin{proof}
Let $\iota\co G\hookrightarrow G\times (G/N)$ be the monomorphism defined by
\[ \iota(g)=(g,\overline{g})\quad(\fa g\in G), \]
where $\overline{g}$ denotes the residue class of $g$ in $G/N$.
Then we have an isomorphism of $G\times (G/N)$-sets
\[ \Ind_{\iota}(\pt)=\Ind_{\iota}(G/G)\ov{\cong}{\lra}G/N. \]
Thus by Proposition \ref{PropIndEquiv}, we have an equivalence
\[ \frac{\ups}{\iota}\co \frac{\pt}{G}\ov{\simeq}{\lra}\frac{(G/N)}{G\times (G/N)} \]
defined by $\ups(\pt)=e$. (Here, $\pt$ denotes the unique point in $\pt=G/G$.)
This $\ups$ makes diagram (\ref{Diag_GGGN}) commutative.
\end{proof}

\begin{rem}\label{RemSIB}
Under the same assumption as in Proposition \ref{PropSIB}, we also have an equivalence of spans as follows, which implies $[\Tbf_{\frac{\pt}{p}}]=[S_{\Def^G_N}]$.
\[
\xy
(-22,0)*+{\frac{\pt}{(G/N)}}="0";
(8,-2)*+{}="1";
(0,12)*+{\frac{\pt}{G}}="2";
(0,-14)*+{\frac{(G/N)}{(G/N)\times G}}="4";
(22,0)*+{\frac{\pt}{G}}="6";
(-8,-2)*+{}="7";
{\ar_{\frac{\pt}{p}} "2";"0"};
{\ar^{\id} "2";"6"};
{\ar^{\simeq} "2";"4"};
{\ar^(0.6){\frac{\pt}{\prgn}} "4";"0"};
{\ar_(0.6){\frac{\pt}{\prg}} "4";"6"};
{\ar@{}|\circlearrowright "0";"1"};
{\ar@{}|\circlearrowright "6";"7"};
\endxy
\]
\end{rem}

\begin{cor}\label{CorSIB}
In particular for $N=e$, we obtain $[S_{{}_G\Id_G}]=[\Id]$.
\end{cor}

\begin{prop}\label{PropSumSpanBiset}
Let $S$ and $T$ be any pair of spans to $\yh$ from $\xg$. Then there is an isomorphism
\[ \Rg(S+T)\cong\Rg(S)\am\Rg(T) \]
of $H$-$G$-bisets.
\end{prop}
\begin{proof}
For $S=(\spSc)$ and $T=(\spTc)$, we have
\[ \Rg(S+T)=(H\times G\times (W_S\am W_T))/\sim. \]
It is straightforward to check this is isomorphic to
\[ \Rg(S)\am\Rg(T)=((H\times G\times W_S)/\sim)\am ((H\times G\times W_T)/\sim). \]
\end{proof}

\subsection{From Mackey functors to biset functors}
Now we relate Mackey functors to biset functors. As in the introduction, $\Bcal=\Bcal_{\Zbb}$ is the biset category which satisfies
\begin{itemize}
\item[-] $\Ob(\Bcal)=\Ob(\Grp)$,
\item[-] $\Bcal(G,H)=K_0(\{ H\text{-}G\text{-bisets} \}/\cong)$ for any $G,H\in\Ob(\Bcal)$.
\end{itemize}

\begin{prop}\label{PropMtoBObj}
Let $F$ be an object in $\Add_{\dfl}(\Tcal,\RMod)$. Then an object $B_F$ in $\ABR$ is associated to $F$ as follows.
\begin{itemize}
\item[{\rm (i)}] For any finite group $G$, put $B_F(G)=F(\frac{\pt}{G})$.
\item[{\rm (ii)}] For any $H$-$G$-biset $U$, $B_F(U)=F([S_U])$. By the linearity, this is extended to any morphism in $\Bcal$.
\end{itemize}
\end{prop}
\begin{proof}
By Claim \ref{RemSU}, {\rm (ii)} is well-defined.
For any finite group $G$, we have
\[ B_F({}_G\Id_G)=F([S_{({}_G\Id_G)}])=F([\Id])=\id_{F(\frac{\pt}{G})}=\id_{B_F(G)} \]
by Corollary \ref{CorSIB}.

By Propositions \ref{PropRangeDeflative}, \ref{PropComposBSSP} and Remark \ref{RemSpBis}, for any consecutive pair of bisets ${}_HU_G$ and ${}_KV_H$ we have
\begin{eqnarray*}
B_F(V)\ci B_F(U)&=&F([S_V])\ci F([S_U])\ =\ F([S_V\ci S_U])\\
&=&F([S_{\Rg(S_V\ci S_U)}])\ =F([S_{(\Rg(S_V)\times_H\Rg(S_U))}])\\
&=&F([S_{(V\times_HU)}])\ =\ B_F(V\times_HU).
\end{eqnarray*}
By linearity, it follows that $B_F$ becomes in fact an additive functor.
\end{proof}

\begin{prop}\label{PropMtoBMorph}
Let $\varphi\co E\to F$ be a morphism in $\Add_{\dfl}(\Tcal,\RMod)$. Then a morphism $B_{\varphi}\co B_E\to B_F$ is associated as follows.
\begin{itemize}
\item[-] For any $G\in \Ob(\Bcal)$,
\[ (B_{\varphi})_G\co B_E(G)\to B_F(G) \]
is defined to be $\varphi_{\frac{\pt}{G}}\co E(\frac{\pt}{G})\to F(\frac{\pt}{G})$.
\end{itemize}
\end{prop}
\begin{proof}
For any $H$-$G$-biset $U$, we have a commutative diagram
\[
\xy
(-14,7)*+{E(\frac{\pt}{G})}="0";
(14,7)*+{E(\frac{\pt}{H})}="2";
(-14,-7)*+{F(\frac{\pt}{G})}="4";
(14,-7)*+{F(\frac{\pt}{H})}="6";
{\ar^{E([S_U])} "0";"2"};
{\ar_{\varphi_{\frac{\pt}{G}}} "0";"4"};
{\ar^{\varphi_{\frac{\pt}{H}}} "2";"6"};
{\ar_{F([S_U])} "4";"6"};
{\ar@{}|\circlearrowright "0";"6"};
\endxy
\]
Thus 
\[
\xy
(-12,7)*+{B_E(G)}="0";
(12,7)*+{B_E(H)}="2";
(-12,-7)*+{B_F(G)}="4";
(12,-7)*+{B_F(H)}="6";
{\ar^{B_E(U)} "0";"2"};
{\ar_{(B_{\varphi})_G} "0";"4"};
{\ar^{(B_{\varphi})_H} "2";"6"};
{\ar_{B_F(U)} "4";"6"};
{\ar@{}|\circlearrowright "0";"6"};
\endxy
\]
is commutative for any $H$-$G$-biset $U$. By linearity, this implies $B_{\varphi}\co B_E\to B_F$ is in fact a natural transformation.
\end{proof}

\begin{cor}\label{CorMtoB}
We obtain a functor
\[ \Phi\co\Add_{\dfl}(\Tcal,\RMod)\to\ABR\ ;\ F\mapsto B_F. \]
\end{cor}
\begin{proof}
This follows from Propositions \ref{PropMtoBObj} and \ref{PropMtoBMorph}.
\end{proof}

\subsection{From biset functors to Mackey functors}
In this subsection, we associate biset functors to deflative Mackey functors. For a 0-cell $\xg$ in $\Sbb$, let $x_1,\ldots,x_s\in X$ be a complete set of representatives of $G$-orbits of $X$. We make this choice once and for all, and use it in the rest of the paper.
\begin{dfn}\label{DefEta}
Let $\xg$ be a 0-cell, as above. 
Since
\[ \Ind^G_{G_{x_i}}(G_{x_i}/G_{x_i})=G/G_{x_i}\cong Gx_i \]
as $G$-sets, we have an equivalence
\[ \ups_i^{(X)}\co \frac{\pt}{G_{x_i}}\ov{\sim}{\lra}\frac{Gx_i}{G}. \]
By composing this with the inclusion $Gx_i\hookrightarrow X$, we define a 1-cell $\eix\co \frac{\pt}{G_{x_i}}\to \xg$.
\end{dfn}
\begin{rem}
The union of $\eix\ (1\le i\le s)$ gives an equivalence
\[ \eta=\underset{1\le i\le s}{\bigcup}\eix\co\underset{1\le i\le s}{\coprod}(\frac{\pt}{G_{x_i}})\ov{\simeq}{\lra}\xg. \]
Thus for any $F\in\Ob(\Add(\Tcal,\RMod))$,
\[ \left(\begin{array}{c}F(\Rbf_{\eta^{(X)}_1})\\ \vdots\\ F(\Rbf_{\eta^{(X)}_s})\end{array}\right)\co F(\xg)\lra \underset{1\le i\le s}{\bigoplus} F(\frac{\pt}{G_{x_i}})  \]
becomes an isomorphism of $R$-modules, with the inverse
\[ \big( F(\Tbf_{\eta_1^{(X)}})\ \cdots\ F(\Tbf_{\eta_s^{(X)}})\big)\co \underset{1\le i\le s}{\bigoplus} F(\frac{\pt}{G_{x_i}})\to F(\xg). \]
\end{rem}

\begin{dfn}\label{DefGamma}
Let $\xg,\yh$ be 1-cells, and let 
$y_1,\ldots,y_t\in Y$ be a complete set of representatives of $H$-orbits of $Y$.
For any span $S$ to $\yh$ from $\xg$, for any $1\le i\le s$ and $1\le j\le t$, define an equivalence class of span $\Gamma_{ji}^{(S)}$ to $\frac{\pt}{H_{y_j}}$ from $\frac{\pt}{G_{x_i}}$ by
\[ \Gamma_{ji}^{(S)}=\Rbf_{\eta_j^{(Y)}}\ci [S]\ci\Tbf_{\eta_i^{(X)}}. \]
\end{dfn}

\begin{rem}\label{RemGamma}
If $X=\pt$ and $Y=\pt$, then $\Gamma_{11}^{(S)}=[S]$.
\end{rem}

\begin{rem}\label{RemSTGG}
By the well-definedness of compositions of morphisms in $\Sp$, we have
\[ [S]=[T]\ \Longrightarrow\ \Gamma_{ji}^{(S)}=\Gamma_{ji}^{(T)}\quad(\fa i,j) \]
for any pair of spans $S$ and $T$ to $\yh$ from $\xg$.
\end{rem}

\begin{lem}\label{LemST1}
For the identity span ${}_{\xg}\Id_{\xg}$, we have
\[ \Gamma_{ji}^{(\Id)}=\Rbf_{\eta_j^{(X)}}\ci [\Id]\ci\Tbf_{\eta_i^{(X)}} =%
\begin{cases}
0&(i\ne j),\\
[\Id]&(i=j)
\end{cases}
\]
for any $1\le i,j\le s$.
\end{lem}
\begin{proof}
This follows from Proposition \ref{PropPullbackAdjEquivEx} and Corollary \ref{Cor2PBIncl}.
\end{proof}

\begin{lem}\label{LemST2}
Let $\xg,\yh,\zk$ be 0-cells, and let $z_1,\ldots,z_u\in Z$ 
be a complete set of representatives of $K$-orbits.
For an consecutive pair of spans ${}_{\yh}S_{\xg}$ and ${}_{\zk}T_{\yh}$, we have
\[ \sum_{j=1}^t\Gamma_{\ell j}^{(T)}\ci\Gamma_{ji}^{(S)}=\Gamma_{\ell i}^{(T\ci S)}, \]
for any $1\le i\le s$ and $1\le \ell\le u$.
\end{lem}
\begin{proof}
This follows from
\begin{eqnarray*}
\sum_{j=1}^t\Gamma_{\ell j}^{(T)}\ci\Gamma_{ji}^{(S)}
&=&\sum_{j=1}^t\, \big(\Rbf_{\eta_{\ell}^{(Z)}}\ci [T]\ci\Tbf_{\eta_j^{(Y)}}\ci\Rbf_{\eta_{j}^{(Y)}}\ci [S]\ci\Tbf_{\eta_i^{(X)}}\big)\\
&=&\Rbf_{\eta_{\ell}^{(Z)}}\ci [T]\ci\sum_{j=1}^t\big(\Tbf_{\eta_j^{(Y)}}\ci\Rbf_{\eta_{j}^{(Y)}}\big)\ci [S]\ci\Tbf_{\eta_i^{(X)}}\\
&=&\Rbf_{\eta_{\ell}^{(Z)}}\ci [T]\ci [\Id] \ci [S]\ci\Tbf_{\eta_i^{(X)}}\\
&=&\Rbf_{\eta_{\ell}^{(Z)}}\ci [T\ci S]\ci\Tbf_{\eta_i^{(X)}}\ =\ \Gamma_{\ell i}^{(T\ci S)}.
\end{eqnarray*}
\end{proof}

\begin{prop}\label{PropBtoMObj}
For any object $B$ in $\ABR$, we can associate an object $F_B$ in $\Add(\Tcal,\RMod)$ as follows. Moreover, $F_B$ becomes deflative.
\begin{itemize}
\item[{\rm (i)}] For any $\xg\in\Ob(\Tcal)$, take a complete set of representatives of $G$-orbits $x_1,\ldots,x_s$, and put
\[ F_B(\xg)=B(G_{x_1})\oplus\cdots\oplus B(G_{x_s}). \]
\item[{\rm (ii)}] Let $\xg$ and $\yh$ be objects in $\Tcal$. Let $x_1,\ldots,x_s\in X$ and $y_1,\ldots,y_t\in Y$ be the set of representatives chosen in {\rm (i)}. For any $[S]\in\Sp(\xg,\yh)$, define
\[ F_B([S])\co F_B(\xg)\to F_B(\yh) \]
to be the matrix
\[ M_B^{(S)}=(m_{ji}^{(S)})_{ji}, \]
with $(j,i)$-component $m_{ji}^{(S)}\co B(G_{x_i})\to B(H_{y_j})$ given by
\[ m_{ji}^{(S)}=B(\Rg(\Gamma_{ji}^{(S)})). \]
This extends to any morphism in $\Tcal$ by linearity.
\end{itemize}
\end{prop}
\begin{proof}
By Remark \ref{RemSTGG}, matrix $M^{(S)}_B$ is well-defined. For any $\xg\in\Ob(\Tcal)$, we have
\begin{eqnarray*}
m_{ji}^{(\Id)}&=&B(\Rg(\Gamma_{ji}^{(\Id)}))\\
&=&\begin{cases}B(\Rg(0))=0&(i\ne j),\\ B(\Rg(\Id))=\id&(i=j)\end{cases}
\end{eqnarray*}
by Lemma \ref{LemST1}, which means $F_B(\Id)=\id_{F_B(\xg)}$.
For any consecutive spans ${}_{\yh}S_{\xg}$ and ${}_{\zk}T_{\yh}$, we have
\begin{eqnarray*}
\sum_{j=1}^tm_{\ell j}^{(T)}\ci m_{ji}^{(S)}&=&\sum_{j=1}^tB(\Rg(\Gamma_{\ell j}^{(T)}))\ci B(\Rg(\Gamma_{ji}^{(S)}))\\
&=&B\big(\underset{1\le j\le t}{\coprod}\big(\Rg(\Gamma_{\ell j}^{(T)})\times_{H_{y_j}}\Rg(\Gamma_{ji}^{(S)})\big)\big)\\
&=&B\big(\underset{1\le j\le t}{\coprod}\Rg(\Gamma_{\ell j}^{(T)}\ci\Gamma_{ji}^{(S)})\big)\\
&=&B\big(\Rg\big(\sum_{j=1}^t\Gamma_{\ell j}^{(T)}\ci\Gamma_{ji}^{(S)}\big)\big)\\
&=&B(\Rg(\Gamma_{\ell i}^{(T\ci S)}))\ =\ m_{\ell i}^{(T\ci S)}
\end{eqnarray*}
for any $1\le i\le s$ and $1\le \ell\le u$ by Claim \ref{RemSU}, Propositions \ref{PropComposBSSP}, \ref{PropSumSpanBiset} and Lemma \ref{LemST2}.
This means $M_{B}^{(T\ci S)}=M_B^{(T)}\ci M_B^{(S)}$, and thus we obtain
\[ F_B([T]\ci [S])=F_B([T])\ci F_B([S]). \]
By linearity, this implies $F_B$ preserves compositions for arbitrary morphisms in $\Tcal$. Thus $F_B\co\Tcal\to\RMod$ is in fact a functor. By construction, $F_B$ preserves finite products.

To show $F_B$ is deflative, let $G$ be any finite group, and let $N\nm G$ be any normal subgroup. Let $p\co G\to G/N$ denote the quotient homomorphism.
Then by Proposition \ref{PropSIB} and Remarks \ref{RemSpBis}, \ref{RemSIB}, \ref{RemGamma}, we have
\begin{eqnarray*}
F_B([\frac{\pt}{(G/N)}\ov{\frac{\pt}{p}}{\lla}\frac{\pt}{G}\ov{\frac{\pt}{p}}{\lra}\frac{\pt}{(G/N)}])&=&F_B(\Tbf_{\frac{\pt}{p}})\ci F_B(\Rbf_{\frac{\pt}{p}})\\
&=&F_B([S_{\Def^G_N}])\ci F_B([S_{\Inf^G_N}])\\
&=&B(\Rg(S_{\Def^G_N}))\ci B(\Rg(S_{\Inf^G_N}))\\
&=&B(\Def^G_N)\ci B(\Inf^G_N)\ =\ \id.
\end{eqnarray*}
\end{proof}

Thus the characterization of the deflativity in Corollary \ref{CorFDef} corresponds to the equation $(\ref{Eq1.3})$ for a biset functor.

\begin{prop}\label{PropBtoMMorph}
For any morphism $\varphi\co B\to B\ppr$ in $\ABR$, we can associate a morphism $F_{\varphi}\co F_B\to F_{B\ppr}$ in $\Add_{\dfl}(\Tcal,\RMod)$ as follows.
\begin{itemize}
\item[-] For any object $\xg\in\Ob(\Tcal)$, 
we define $(F_{\varphi})_{\xg}\co F_B(\xg)\to F_{B\ppr}(\xg)$ to be
\[ \varphi_{G_{x_1}}\oplus\cdots\oplus\varphi_{G_{x_s}}\co B(G_{x_1})\oplus\cdots\oplus B(G_{x_s})\to B\ppr(G_{x_1})\oplus\cdots\oplus B\ppr(G_{x_s}). \]
\end{itemize}
\end{prop}
\begin{proof}
It suffices to show that $F_{\varphi}$ is in fact a natural transformation. By linearity, it is enough to show the commutativity of
\[
\xy
(-12,7)*+{F_B(\xg)}="0";
(12,7)*+{F_{B\ppr}(\xg)}="2";
(-12,-7)*+{F_B(\yh)}="4";
(12,-7)*+{F_{B\ppr}(\yh)}="6";
{\ar^{(F_{\varphi})_{\xg}} "0";"2"};
{\ar_{F_B([S])} "0";"4"};
{\ar^{F_{B\ppr}([S])} "2";"6"};
{\ar_{(F_{\varphi})_{\yh}} "4";"6"};
{\ar@{}|\circlearrowright "0";"6"};
\endxy
\]
for any span ${}_{\yh}S_{\xg}$. However this follows from the commutativity of
\[
\xy
(-30,10)*+{ B(G_{x_1})\oplus\cdots\oplus B(G_{x_s})}="0";
(30,10)*+{ B\ppr(G_{x_1})\oplus\cdots\oplus B\ppr(G_{x_s})}="2";
(-30,-10)*+{ B(H_{y_1})\oplus\cdots\oplus B(H_{y_t})}="4";
(30,-10)*+{ B\ppr(H_{y_1})\oplus\cdots\oplus B\ppr(H_{y_t})}="6";
{\ar^{\varphi_{G_{x_1}}\oplus\cdots\oplus\varphi_{G_{x_s}}} "0";"2"};
{\ar_{M_B^{(S)}} "0";"4"};
{\ar^{M_{B\ppr}^{(S)}} "2";"6"};
{\ar_{\varphi_{H_{y_1}}\oplus\cdots\oplus\varphi_{H_{y_t}}} "4";"6"};
{\ar@{}|\circlearrowright "0";"6"};
\endxy.
\]
\end{proof}

\begin{cor}
We obtain a functor
\[ \Psi\co\ABR\to\Add_{\dfl}(\Tcal,\RMod)\ ;\ B\mapsto F_B. \]
\end{cor}
\begin{proof}
This follows from Propositions \ref{PropBtoMObj} and \ref{PropBtoMMorph}.
\end{proof}

\begin{thm}\label{ThmBF}
There is an equivalence of categories
\[ \Mack_{\dfl}^R(\Sbb)\simeq\BisetFtr. \]
\end{thm}
\begin{proof}
By Proposition \ref{PropAddFtrMackRCoeff}, it suffices to show the equivalence $\Add_{\dfl}(\Tcal,\RMod)\simeq\ABR$.
So far we constructed functors
\[ \Phi\co\Add_{\dfl}(\Tcal,\RMod)\to\ABR\ ;\ F\mapsto B_F, \]
and 
\[ \Psi\co\ABR\to\Add_{\dfl}(\Tcal,\RMod)\ ;\ B\mapsto F_B. \]
It suffices to show $\Phi\ci\Psi=\Id$ and $\Psi\ci\Phi\cong\Id$.

\medskip

\noindent {\rm (i)} $\Phi\ci\Psi=\Id$.

Let $B$ be any object in $\ABR$.
For any finite group $G$, we have
\[ B_{F_B}(G)=F_B(\frac{\pt}{G})=B(G). \]
For any $H$-$G$-biset $U$, we have
\[ B_{F_B}(U)=F_B(S_U)=B(\Rg(S_U))=B(U). \]
Thus we have $B_{F_B}=B$. This gives $\Phi\ci\Psi=\Id$.

\medskip

\noindent {\rm (ii)} $\Psi\ci\Phi\cong\Id$.

Let $F$ be any object in $\Add_{\dfl}(\Tcal,\RMod)$.
For any object $\xg\in\Ob(\Tcal)$, we have an equivalence
\[ \eta^{(X)}=\underset{1\le i\le s}{\bigcup}\eta_i^{(X)}\co\underset{1\le i\le s}{\coprod}(\frac{\pt}{G_{x_i}})\ov{\simeq}{\lra}\xg \]
as in Definition \ref{DefEta}, and thus an isomorphism
\[ \tau_{\xg}=\left(\begin{array}{c}F(\Rbf_{\eta_1^{(X)}})\\ \vdots \\ F(\Rbf_{\eta_s^{(X)}})\end{array}\right)\co F(\xg)\ov{\cong}{\lra}\underset{1\le i\le s}{\bigoplus}F(\frac{\pt}{G_{x_i}})=F_{B_F}(\xg). \]

Let $S$ be any span to $\yh$ from $\xg$. 
Remark that by definition of $\Gamma_{ji}^{(S)}$ we have a commutative diagram
\begin{equation}\label{DiagFin}
\xy
(-26,10)*+{\underset{1\le i\le s}{\bigoplus}F(\frac{\pt}{G_{x_i}})}="0";
(26,10)*+{F(\xg)}="2";
(-26,-10)*+{\underset{1\le j\le t}{\bigoplus}F(\frac{\pt}{H_{y_j}})}="4";
(26,-10)*+{F(\yh)}="6";
{\ar^(0.54){\big( F(\Tbf_{\eta_1^{(X)}})\ \cdots\ F(\Tbf_{\eta_s^{(X)}}) \big)}_{\cong} "0";"2"};
{\ar_{(F(\Gamma_{ji}^{(S)}))_{ji}} "0";"4"};
{\ar^{F([S])} "2";"6"};
{\ar^{\tau_{\yh}}_{\cong} "6";"4"};
{\ar@{}|\circlearrowright "0";"6"};
\endxy.
\end{equation}
As in Proposition \ref{PropBtoMObj}, the map $F_{B_F}(S)$ is defined by using
\[ M_{B_F}^{(S)}=(m_{ji}^{(S)}), \]
where $m_{ji}^{(S)}=B_F(\Rg(\Gamma_{ji}^{(S)}))$.
Since $F$ is deflative, we have
\[ F(\Gamma_{ji}^{(S)})=F([S_{\Rg(\Gamma_{ji}^{(S)})}])=B_F(\Rg(\Gamma_{ji}^{(S)}))=m_{ji}^{(S)}\ \ \ (\fa i,j), \]
and thus we obtain $F_{B_F}(S)=(F(\Gamma_{ji}^{(S)}))_{ji}$. By
\[ \big( F(\Tbf_{\eta_1^{(X)}})\ \cdots\ F(\Tbf_{\eta_s^{(X)}}) \big)=\left(\begin{array}{c}F(\Rbf_{\eta_1^{(X)}})\\ \vdots \\ F(\Rbf_{\eta_s^{(X)}})\end{array}\right)^{-1}, \]
diagram (\ref{DiagFin}) means the commutativity of
\[
\xy
(-12,8)*+{F_{B_F}(\xg)}="0";
(12,8)*+{F(\xg)}="2";
(-12,-8)*+{F_{B_F}(\yh)}="4";
(12,-8)*+{F(\yh)}="6";
{\ar_(0.46){\tau_{\xg}}^(0.46){\cong} "2";"0"};
{\ar_{F_{B_F}([S])} "0";"4"};
{\ar^{F([S])} "2";"6"};
{\ar^(0.46){\tau_{\yh}}_(0.46){\cong} "6";"4"};
{\ar@{}|\circlearrowright "0";"6"};
\endxy.
\]
By the linearity, this shows that $\tau$ becomes an isomorphism $\tau\co F\ov{\cong}{\lra} F_{B_F}$ of objects in $\Add_{\dfl}(\Tcal,\RMod)$.
It can be easily checked this gives a natural isomorphism $\Psi\ci\Phi\cong\Id$.
\end{proof}

\section*{Acknowledgement}
This work was completed when the author was staying at LAMFA, l'Universit\'{e} de Picardie-Jules Verne, by the support of JSPS Postdoctoral Fellowships for Research Abroad. He wishes to thank the hospitality of Professor Serge Bouc, Professor Radu Stancu and the members of LAMFA.

\end{document}